\documentclass[12pt, reqno]{amsart}
\makeatletter
\DeclareMathAlphabet{\mathpzc}{OT1}{pzc}{m}{it}

\usepackage[margin=2cm]{geometry}
\usepackage{amsmath, amssymb, amsfonts, amstext, verbatim, amsthm, mathrsfs}
\usepackage[mathcal]{eucal}
\usepackage{microtype}
\usepackage[all,arc,knot,poly]{xy}
\usepackage{mathrsfs}
\usepackage{mathtools}
\usepackage{amsxtra}
\usepackage{needspace}
\usepackage{graphicx}
\usepackage{setspace}
\usepackage[usenames]{color}
\usepackage{aliascnt} 
\usepackage[colorlinks=true,linkcolor=blue,citecolor=blue,urlcolor=blue,citebordercolor={0 0 1},urlbordercolor={0 0 1},linkbordercolor={0 0 1}]{hyperref} 
\usepackage{enumerate}
\usepackage{enumitem}
\usepackage{xspace}
\usepackage{stmaryrd}

\theoremstyle{plain}

\newcommand{\refnewtheoremn}[4]{
\newaliascnt{#1}{#2}
\newtheorem{#1}[#1]{#3}
\aliascntresetthe{#1}
\expandafter\providecommand\csname #1autorefname\endcsname{#4}}

\newcommand{\refnewtheorem}[3]{\refnewtheoremn{#1}{#2}{#3}{#3}}

\def\makeCal#1{
\expandafter\newcommand\csname c#1\endcsname{\mathcal{#1}}}
\def\makeBB#1{
\expandafter\newcommand\csname b#1\endcsname{\mathbb{#1}}}
\def\makeFrak#1{
\expandafter\newcommand\csname f#1\endcsname{\mathfrak{#1}}}

\count@=0
\loop
\advance\count@ 1
\edef\y{\@Alph\count@}
\expandafter\makeCal\y
\expandafter\makeBB\y
\expandafter\makeFrak\y
\ifnum\count@<26
\repeat

\newtheorem{theorem}{Theorem}[section]
\refnewtheorem{lemma}{thm}{Lemma}
\refnewtheorem{claim}{thm}{Claim}
\refnewtheorem{cor}{thm}{Corollary}
\refnewtheorem{conj}{thm}{Conjecture}
\refnewtheorem{prop}{thm}{Proposition}
\refnewtheorem{proposition}{thm}{Proposition}
\refnewtheorem{quest}{thm}{Question}
\refnewtheorem{assumption}{thm}{Assumption}
\refnewtheorem{problem}{thm}{Problem}

\theoremstyle{definition}
\refnewtheorem{remark}{thm}{Remark}
\refnewtheorem{remarks}{thm}{Remarks}
\refnewtheorem{defn}{thm}{Definition}
\refnewtheorem{definition}{thm}{Definition}
\refnewtheorem{notn}{thm}{Notation}
\refnewtheorem{const}{thm}{Construction}
\refnewtheorem{example}{thm}{Example}
\refnewtheorem{examples}{thm}{Examples}
\refnewtheorem{fact}{thm}{Fact}

\DeclareMathOperator{\PGL}{\operatorname{PGL}}
\DeclareMathOperator{\Hom}{\operatorname{Hom}}
\DeclareMathOperator{\MCG}{\operatorname{MCG}}
\DeclareMathOperator{\Tilt}{\operatorname{Tilt}}
\DeclareMathOperator{\Aut}{\operatorname{Aut}}
\DeclareMathOperator{\Sph}{\operatorname{Sph}}
\DeclareMathOperator{\Exch}{\operatorname{Exch}}
\DeclareMathOperator{\Stab}{\operatorname{Stab}}
\DeclareMathOperator{\Tri}{\operatorname{Tri}}
\DeclareMathOperator{\Zer}{\operatorname{Zer}}
\DeclareMathOperator{\Pol}{\operatorname{Pol}}
\DeclareMathOperator{\Crit}{\operatorname{Crit}}
\DeclareMathOperator{\Res}{\operatorname{Res}}
\DeclareMathOperator{\sgn}{\operatorname{sgn}}
\DeclareMathOperator{\Ext}{\operatorname{Ext}}
\DeclareMathOperator{\Nil}{\operatorname{Nil}}
\DeclareMathOperator{\Tw}{\operatorname{Tw}}
\DeclareMathOperator{\Quad}{\operatorname{Quad}}
\newcommand{\cAut}{\mathpzc{Aut}}
\newcommand{\cSph}{\mathpzc{Sph}}

\renewcommand{\Im}{\operatorname{Im}}
\renewcommand{\Re}{\operatorname{Re}}

\begin{document}

\title[Stability conditions, cluster varieties, and Riemann-Hilbert problems]{Stability conditions, cluster varieties, \protect\\and Riemann-Hilbert problems from surfaces}
\author{Dylan G.L. Allegretti}

\date{}

\maketitle

\begin{abstract}
We consider two interesting spaces associated to a quiver with potential: a space of stability conditions and a cluster variety. In the case where the quiver with potential arises from an ideal triangulation of a marked bordered surface, we construct a natural map from a dense subset of the space of stability conditions to the cluster variety. Using this construction, we give solutions to a family of Riemann-Hilbert problems arising in Donaldson-Thomas theory.
\end{abstract}

\section{Introduction}

This paper is the main work in a series \cite{Allegretti17, AllegrettiBridgeland, Allegretti19, Allegretti20} on the relationship between two spaces. One of these spaces is a complex manifold parametrizing Bridgeland stability conditions on a certain 3-Calabi-Yau triangulated category, and the other is a cluster variety. The structure of both spaces is controlled by the combinatorics of quiver mutations, and yet the two spaces look quite different geometrically. Indeed, the space of stability conditions has a cell decomposition, whereas the cluster variety is composed of algebraic tori glued together by birational maps. The aim of this paper is to understand the highly nontrivial relationship between these spaces in a large class of examples arising from triangulated surfaces.

A hint that there might be some deep relationship between stability conditions and the cluster variety comes from the work of Gaiotto, Moore, and Neitzke in physics \cite{GMN1, GMN2}. Their work paints a remarkable conjectural picture involving Higgs bundles, Donaldson-Thomas invariants, and the Kontsevich-Soibelman wall-crossing formula. While these papers have led to a great deal of recent mathematical work, a complete and mathematically rigorous approach to the physical theories studied in \cite{GMN1, GMN2} does not yet exist.

The present paper is a step towards a mathematical understanding of the work of Gaiotto, Moore, and Neitzke. Rather than attempt to formalize their ideas in complete generality, we work in the so called conformal limit, which was first studied in the physics literature in~\cite{Gaiotto} and later in works by mathematicians~\cite{DFKMMN, CollierWentworth}. By combining the results of our previous papers \cite{AllegrettiBridgeland, Allegretti19, Allegretti20}, we construct a canonical map from a dense subset of the space of stability conditions to the cluster variety and show that this map relates various features of the two spaces. This construction generalizes our earlier results~\cite{Allegretti17} and gives a way of understanding~\cite{GMN2} mathematically in the conformal limit.

At the heart of the construction of Gaiotto, Moore, and Neitzke is a certain Riemann-Hilbert problem. In this paper, we will consider the conformal limit of this Riemann-Hilbert problem, which was recently studied by Bridgeland in the context of Donaldson-Thomas theory~\cite{Bridgeland19}. A solution of this problem is a piecewise meromorphic function on~$\mathbb{C}^*$ having prescribed discontinuities along a collection of rays. Previously, solutions of this problem were known only in a handful of special cases~\cite{Bridgeland19, Bridgeland17, Barbieri}. Here, using the relationship between the space of stability conditions and the cluster variety, we give solutions of the Riemann-Hilbert problem in the much larger class of examples associated to triangulated surfaces. We suggest that solving these Riemann-Hilbert problems is the key to gaining a deeper understanding of the relationship between stability conditions and the cluster variety.

\subsection{Quadratic differentials and local systems}

In order to relate the space of stability conditions and the cluster variety, we will focus on a class of examples in which these spaces can be interpreted as moduli spaces of geometric structures on surfaces. In this class of examples, the work of Bridgeland and Smith~\cite{BridgelandSmith} shows that the space of stability conditions is isomorphic to a moduli space of meromorphic quadratic differentials, while the work of Fock and Goncharov~\cite{FockGoncharov1} shows that the cluster variety is birational to a moduli space of local systems equipped with additional framing data. Before describing the relationship between these two moduli spaces, we will discuss an analogous relationship between holomorphic differentials and unframed local systems, where the story is quite classical.

A holomorphic quadratic differential on a Riemann surface~$S$ is defined as a holomorphic section of $\omega_S^{\otimes2}$ where $\omega_S$ is the holomorphic cotangent bundle of~$S$. If $\mathbb{S}$ is any closed, oriented surface of genus~$g\geq2$, let $\mathscr{T}(\mathbb{S})$ denote the Teichm\"uller space of $\mathbb{S}$, viewed as the moduli space of Riemann surfaces~$S$ equipped with a marking, that is, an isotopy class of orientation preserving diffeomorphisms $\theta:\mathbb{S}\rightarrow S$. There is a vector bundle 
\[
q:\mathscr{Q}(\mathbb{S})\rightarrow\mathscr{T}(\mathbb{S})
\]
whose fiber over $(S,\theta)$ is the vector space 
\begin{equation}
\label{eqn:holomorphicdifferentials}
H^0(S,\omega_S^{\otimes2})\cong\mathbb{C}^{3g-3}
\end{equation}
of holomorphic quadratic differentials.

The notion of a quadratic differential is closely related to the notion of a projective structure. A projective structure on a Riemann surface $S$ is defined as an atlas of holomorphic charts $z_i:U_i\rightarrow\mathbb{P}^1$ where the domains $U_i$ form an open cover of~$S$ and each transition function $g_{ij}=z_i\circ z_j^{-1}$ is the restriction of an element of $\PGL_2(\mathbb{C})$. Suppose we are given a projective structure $\mathcal{P}$ and quadratic differential $\phi$ on a Riemann surface~$S$. If $z:U\rightarrow\mathbb{P}^1$ is a chart of~$\mathcal{P}$ and we write 
\[
\phi(z)=\varphi(z)dz^{\otimes2}
\]
for some holomorphic function $\varphi(z)$ in this local coordinate, then we obtain a chart in a new projective structure $\mathcal{P}+\phi$ by considering the ratio of two linearly independent solutions of the differential equation 
\begin{equation}
\label{eqn:introschrodinger}
y''(z)-\varphi(z)\cdot y(z)=0.
\end{equation}
This construction gives the set of projective structures on a Riemann surface the structure of an affine space for the vector space~\eqref{eqn:holomorphicdifferentials}.

Given a compact oriented surface $\mathbb{S}$ as before, we can consider the set $\mathscr{P}(\mathbb{S})$ of equivalence classes of triples $(S,\mathcal{P},\theta)$ where $S$ is a Riemann surface equipped with a projective structure~$\mathcal{P}$, and $\theta:\mathbb{S}\rightarrow S$ is a marking. Two triples $(S_1,\mathcal{P}_1,\theta_1)$ and $(S_2,\mathcal{P}_2,\theta_2)$ are considered to be equivalent if there is a biholomorphism $f:S_1\rightarrow S_2$ which preserves the projective structures and commutes with the markings in the obvious way. The set $\mathscr{P}(\mathbb{S})$ has the natural structure of a complex manifold of dimension $6g-6$, and there is an obvious forgetful map 
\[
p:\mathscr{P}(\mathbb{S})\rightarrow\mathscr{T}(\mathbb{S})
\]
which is an affine bundle for the vector bundle $q$ of quadratic differentials.

It follows that after choosing a continuous section of the bundle~$p$, we get a (not necessarily holomorphic) homeomorphism 
\begin{equation}
\label{eqn:identificationholomorphic}
\mathscr{Q}(\mathbb{S})\cong\mathscr{P}(\mathbb{S}).
\end{equation}
One natural choice of section is the uniformizing section. Given a point $(S,\theta)$ in the Teichm\"uller space $\mathscr{T}(\mathbb{S})$, its image under this section is the triple $(S,\mathcal{P},\theta)$ where the charts of $\mathcal{P}$ are defined as local sections of the universal covering map $\widetilde{S}\rightarrow S$. By the uniformization theorem, the universal cover $\widetilde{S}$ is isomorphic to $\mathbb{P}^1$,~$\mathbb{C}$, or $\mathbb{H}$, with deck transformations acting as elements of $\PGL_2(\mathbb{C})$. Hence this construction defines a projective structure.

Associated to the surface $\mathbb{S}$ is the quotient stack 
\[
\mathscr{X}(\mathbb{S})=\Hom(\pi_1(\mathbb{S}),G)/G
\]
parametrizing representations of the fundamental group of~$\mathbb{S}$ into $G=\PGL_2(\mathbb{C})$ up to conjugation. Any projective structure determines an associated $G$-local system and hence a point of~$\mathscr{X}(\mathbb{S})$ called the monodromy of the projective structure. The stack $\mathscr{X}(\mathbb{S})$ contains an open substack $\mathscr{X}^*(\mathbb{S})$ having the structure of a complex manifold, and there is a holomorphic map $F:\mathscr{P}(\mathbb{S})\rightarrow\mathscr{X}^*(\mathbb{S})$ sending a projective structure to its monodromy. By composing with the identification~\eqref{eqn:identificationholomorphic} defined using the uniformizing section, we obtain a mapping class group equivariant continuous map 
\[
\widehat{F}:\mathscr{Q}(\mathbb{S})\rightarrow\mathscr{X}^*(\mathbb{S}).
\]
A classical result known as Hejhal's theorem states that the monodromy map $F$ is a local biholomorphism~\cite{Hubbard}. This implies that the map $\widehat{F}$ is a local homeomorphism.

The construction we have just described is closely related to another construction in differential geometry. Given a compact Riemann surface~$S$, one can think of the space~\eqref{eqn:holomorphicdifferentials} as the base of Hitchin's integrable system. Then the choice of a holomorphic quadratic differential $\phi$ is equivalent to a choice of Higgs bundle on the Hitchin section. By the nonabelian Hodge~correspondence, this determines a corresponding family of local systems on the surface~$S$. A conjecture of Gaiotto~\cite{Gaiotto}, proved mathematically in~\cite{DFKMMN}, states that in a certain scaling limit known as the conformal limit, one recovers the local system~$\widehat{F}(\phi)$. This is the reason we defined $\widehat{F}$ in the manner described above. In particular, it is the reason we used the uniformizing section to make the identification~\eqref{eqn:identificationholomorphic}.

\subsection{Meromorphic differentials and framed local systems}

Let us now consider a compact, connected Riemann surface~$S$ and a meromorphic quadratic differential~$\phi$ on~$S$. In other words, $\phi$~is a meromorphic section of the line bundle $\omega_S^{\otimes2}$. If $p\in S$ is a pole of~$\phi$ of order $m\geq3$, then we will see below that there are $m-2$ distinguished tangent vectors at the point~$p$. We can define a compact oriented surface $\mathbb{S}$ with boundary by taking the oriented real blow up of~$S$ at each of the poles of $\phi$ order~$\geq3$. We also get a finite set $\mathbb{M}$ of marked points on~$\mathbb{S}$. It consists of the points on the boundary of~$\mathbb{S}$ corresponding to the distinguished tangent directions, together with the poles of~$\phi$ of order $\leq2$, regarded as marked points in the interior of~$\mathbb{S}$.

In general, a pair $(\mathbb{S},\mathbb{M})$ consisting of a compact oriented surface~$\mathbb{S}$ and a nonempty finite set $\mathbb{M}\subset\mathbb{S}$ of marked points such that each boundary component has at least one marked point is called a marked bordered surface. Given an arbitrary marked bordered surface $(\mathbb{S},\mathbb{M})$, a marking $\theta$ of~$(S,\phi)$ is defined as an isotopy class of isomorphisms from $(\mathbb{S},\mathbb{M})$ to the marked bordered surface determined by the pair $(S,\phi)$. In Section~\ref{sec:QuadraticDifferentials}, we define a complex manifold $\mathscr{Q}(\mathbb{S},\mathbb{M})$ parametrizing triples $(S,\phi,\theta)$ where $S$ is a compact Riemann surface equipped with a meromorphic differential~$\phi$ with simple zeros and a marking $\theta$ of~$(S,\phi)$ by~$(\mathbb{S},\mathbb{M})$. We also consider a ramified cover 
\[
\mathscr{Q}^\pm(\mathbb{S},\mathbb{M})\rightarrow\mathscr{Q}(\mathbb{S},\mathbb{M})
\]
of degree $2^{|\mathbb{P}|}$ where $\mathbb{P}\subset\mathbb{M}$ is the set of interior marked points. This cover is branched precisely over the locus of quadratic differentials having simple poles.

In~\cite{AllegrettiBridgeland}, we introduced the related notion of a meromorphic projective structure. This notion has meaning because of the fact mentioned above that the set of projective structures on a fixed Riemann surface is an affine space for the vector space of quadratic differentials. If we fix an ordinary projective structure on a compact Riemann surface~$S$, then the charts of a meromorphic projective structure are obtained by taking ratios of solutions of~\eqref{eqn:introschrodinger} where now the quadratic differential~$\phi$ is allowed to have poles. This quadratic differential~$\phi$ is called a polar differential for the meromorphic projective structure.

If $\mathcal{P}$ is a meromorphic projective structure on a Riemann surface~$S$, then a marking of the pair $(S,\mathcal{P})$ by~$(\mathbb{S},\mathbb{M})$ is defined as a marking of $(S,\phi)$ by~$(\mathbb{S},\mathbb{M})$ where $\phi$ is any polar differential for~$\mathcal{P}$. As we will review in Section~\ref{sec:ProjectiveStructuresAndMonodromy}, there is a moduli space $\mathscr{P}(\mathbb{S},\mathbb{M})$ parametrizing equivalence classes of triples $(S,\mathcal{P},\theta)$ where $S$ is a Riemann surface equipped with a meromorphic projective structure~$\mathcal{P}$, and $\theta$ is a marking of $(S,\mathcal{P})$ by~$(\mathbb{S},\mathbb{M})$. If we assume that $|\mathbb{M}|\geq3$ whenever $\mathbb{S}$ has genus zero, then it is a complex manifold. We will also define a ramified cover 
\[
\mathscr{P}^\pm(\mathbb{S},\mathbb{M})\rightarrow\mathscr{P}(\mathbb{S},\mathbb{M})
\]
of degree $2^{|\mathbb{P}|}$ over this manifold of projective structures.

In Section~\ref{sec:RelatingTheModuliSpaces}, we use uniformization to prove that there is a non-holomorphic open embedding $\mathscr{Q}(\mathbb{S},\mathbb{M})\hookrightarrow\mathscr{P}(\mathbb{S},\mathbb{M})$ analogous to the homeomorphism~\eqref{eqn:identificationholomorphic}. It is an embedding rather than a homeomorphism because our quadratic differentials are assumed to have simple zeros. There is a canonical lift to an open embedding 
\begin{equation}
\label{eqn:identificationmeromorphic}
\mathscr{Q}^\pm(\mathbb{S},\mathbb{M})\hookrightarrow\mathscr{P}^\pm(\mathbb{S},\mathbb{M})
\end{equation}
of the covering spaces.

The main result of~\cite{AllegrettiBridgeland} was the construction of a natural map taking a meromorphic projective structure to its monodromy. In the case of meromorphic projective structures, the monodromy is most naturally viewed as a framed rather than ordinary local system. A framed $G$-local system on a marked bordered surface $(\mathbb{S},\mathbb{M})$ was defined in~\cite{FockGoncharov1} as a $G$-local system $\mathcal{G}$ on the punctured surface $\mathbb{S}\setminus\mathbb{P}$ together with a flat section of the associated $\mathbb{P}^1$-bundle $\mathcal{L}=\mathcal{G}\times_G\mathbb{P}^1$ near each of the marked points. There is a moduli stack $\mathscr{X}(\mathbb{S},\mathbb{M})$ parametrizing framed $G$-local systems on~$(\mathbb{S},\mathbb{M})$. It contains an open substack $\mathscr{X}^*(\mathbb{S},\mathbb{M})$ which we can view as a non-Hausdorff manifold. If the surface~$\mathbb{S}$ has genus zero, let us assume that $|\mathbb{M}|\geq3$. Then the main result of~\cite{AllegrettiBridgeland} gives a mapping class group equivariant holomorphic map 
\[
F:\mathscr{P}^*(\mathbb{S},\mathbb{M})\rightarrow\mathscr{X}^*(\mathbb{S},\mathbb{M})
\]
from a dense open subset $\mathscr{P}^*(\mathbb{S},\mathbb{M})\subset\mathscr{P}^\pm(\mathbb{S},\mathbb{M})$ into this complex manifold. It sends a projective structure to its monodromy local system with the additional framing defined by the Stokes data for the equation~\eqref{eqn:introschrodinger}. If we write $\mathscr{Q}^*(\mathbb{S},\mathbb{M})\subset\mathscr{Q}^\pm(\mathbb{S},\mathbb{M})$ for the preimage of the set $\mathscr{P}^*(\mathbb{S},\mathbb{M})$ under the embedding~\eqref{eqn:identificationmeromorphic}, then $F$ gives rise to a mapping class group equivariant continuous map 
\begin{equation}
\label{eqn:map}
\widehat{F}:\mathscr{Q}^*(\mathbb{S},\mathbb{M})\rightarrow\mathscr{X}^*(\mathbb{S},\mathbb{M}).
\end{equation}
A recent generalization of Hejhal's theorem implies that if $(\mathbb{S},\mathbb{M})$ has no interior marked points, then this map is a local homeomorphism~\cite{GuptaMj}. We conjecture that this remains true without the extra assumption on~$(\mathbb{S},\mathbb{M})$.

\subsection{Stability conditions and the cluster variety}

The main idea of this paper is that~\eqref{eqn:map} can be viewed as a map from a dense subset of a space of stability conditions to a corresponding cluster variety, and that in general such maps are closely related to the Riemann-Hilbert problems of~\cite{Bridgeland19}. Both the space of stability conditions and the cluster variety can be constructed from the data of a quiver with potential, which in our situation arises from a choice of triangulation of~$(\mathbb{S},\mathbb{M})$.

More precisely, an ideal triangulation of a marked bordered surface $(\mathbb{S},\mathbb{M})$ is defined as a triangulation of~$\mathbb{S}$, all of whose edges begin and end at points of~$\mathbb{M}$. Given an ideal triangulation~$T$ of~$(\mathbb{S},\mathbb{M})$, one can construct an associated quiver $Q(T)$, an example of which is illustrated in Figure~\ref{fig:quivertriangulation}. This quiver has a vertex on each internal edge of the triangulation, and these vertices are connected in such a way that there is a small clockwise oriented 3-cycle inscribed in each internal triangle. For purposes of this discussion, we will assume that the triangulation is regular, meaning that each interior marked point has valency at least three.

\begin{figure}[ht]
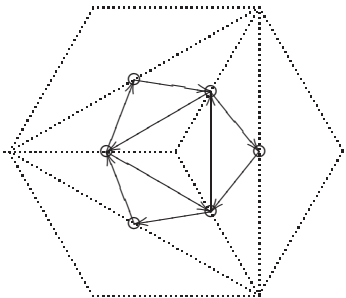
 \begin{center}
\[
\xy /l1.75pc/:
{\xypolygon6"A"{~:{(3,0):}~>{.}}},
{"A1"\PATH~={**@{.}}'"A3"'},
{"A3"\PATH~={**@{.}}'"A5"'},
{"A5"\PATH~={**@{.}}'"A1"'},
{(1,0)\PATH~={**@{.}}'"A3"'},
{(1,0)\PATH~={**@{.}}'"A5"'},
{(1,0)\PATH~={**@{.}}'"A1"'},
{\xypolygon3"B"{~:{(0,1.5):}~>{}}},
{\xypolygon3"C"{~:{(0,-1.25):}~>{}}},
"B1"*{\circ}, 
"B2"*{\circ},
"B3"*{\circ},
"C1"*{\circ}, 
"C2"*{\circ},
"C3"*{\circ},
\xygraph{
"C3":"B1",
"C2":"C3",
"B1":"C2",
"C1":"C2",
"B3":"C1",
"C2":"B3",
"C3":"C1",
"B2":"C3",
"C1":"B2",
}
\endxy
\]
\caption{The quiver associated to an ideal triangulation of a punctured disk.\label{fig:quivertriangulation}}
\end{center} \end{figure}

A potential for a quiver is a formal linear combination of oriented cycles. In the case of the quiver~$Q(T)$, there are two obvious types of oriented cycles, namely the clockwise oriented 3-cycle~$\tau(t)$ in each internal triangle $t$ of~$T$ and the counterclockwise oriented cycle $\pi(p)$ of length at least three encircling each interior marked point~$p$. There is a canonical potential for the quiver $Q(T)$ defined by the formula 
\[
W(T)=\sum_t\tau(t)-\sum_p\pi(p).
\]
This definition was extended to the case of non-regular triangulations by Labardini-Fragoso~\cite{LabardiniFragoso1}, who also showed that the quivers with potential associated to any two ideal triangulations of~$(\mathbb{S},\mathbb{M})$ are related by a sequence of elementary operations called mutations.

Associated to a quiver with potential $(Q,W)$ satisfying a certain nondegeneracy condition are various objects whose precise definitions will be given in Section~\ref{sec:StabilityConditionsAndTheClusterVariety}. Most important for us is an associated 3-Calabi-Yau triangulated category $\mathcal{D}(Q,W)$. Explicitly, it is defined as the full subcategory of the derived category of the complete Ginzburg algebra of $(Q,W)$ consisting of modules with finite-dimensional cohomology. This category is equipped with a canonical bounded t-structure whose heart $\mathcal{A}(Q,W)\subset\mathcal{D}(Q,W)$ is a full abelian subcategory encoding the quiver~$Q$. A result of Keller and Yang~\cite{KellerYang} says that if $(Q',W')$ is a quiver with potential obtained from $(Q,W)$ by mutation, then the associated categories $\mathcal{D}(Q,W)$ and $\mathcal{D}(Q',W')$ are related by a pair of canonical triangulated equivalences.

The structure of the category $\mathcal{D}=\mathcal{D}(Q,W)$ is controlled by the tilting graph $\Tilt(\mathcal{D})$, which has vertices in bijection with the finite length hearts in~$\mathcal{D}$ and where two vertices are connected by an edge if the associated hearts are related by the operation of tilting. There is a distinguished component $\Tilt_\Delta(\mathcal{D})\subset\Tilt(\mathcal{D})$ which contains the distinguished heart $\mathcal{A}(Q,W)$. The group $\Aut(\mathcal{D})$ of triangulated autoequivalences acts naturally on $\Tilt(\mathcal{D})$. We denote by $\Aut_\Delta(\mathcal{D})$ the subgroup that preserves the distinguished component $\Tilt_\Delta(\mathcal{D})$ and by $\cAut_\Delta(\mathcal{D})$ the quotient of $\Aut_\Delta(\mathcal{D})$ by the subgroup of autoequivalences which act trivially on~$\Tilt_\Delta(\mathcal{D})$. There is also a distinguished subgroup $\cSph_\Delta(\Delta)\subset\cAut_\Delta(\mathcal{D})$ generated by the spherical twist functors introduced by Seidel and Thomas~\cite{SeidelThomas}. The quotient 
\[
\Exch_\Delta(\mathcal{D})=\Tilt_\Delta(\mathcal{D})/\cSph_\Delta(\mathcal{D})
\]
is known as the heart exchange graph. It carries an action of the group 
\[
\mathcal{G}_\Delta(\mathcal{D})=\cAut_\Delta(\mathcal{D})/\cSph_\Delta(\mathcal{D}),
\]
which is known as the cluster modular group.

The space of stability conditions on a triangulated category was introduced by Bridgeland~\cite{Bridgeland07} to formalize ideas about stability of D-branes in string theory. It is a complex manifold equipped with commuting actions of the group of autoequivalences and the group of complex numbers. In the case of the triangulated category $\mathcal{D}=\mathcal{D}(Q,W)$ associated to a quiver with potential, the space of stability conditions contains a distinguished component $\Stab_\Delta(\mathcal{D})$. We will be interested in the quotient 
\[
\Sigma(Q,W)=\Stab_\Delta(\mathcal{D})/\cSph_\Delta(\mathcal{D}),
\]
which has a natural cell decomposition with dual graph $\Exch_\Delta(\mathcal{D})$ and an action by the group~$\mathcal{G}_\Delta(\mathcal{D})$. On the other hand, the cluster Poisson variety $\mathscr{X}^{\mathrm{cl}}(Q)$ was introduced by Fock and Goncharov in the context of higher Teichm\"uller theory~\cite{FockGoncharov1, FockGoncharov2}. It is a nonseparated scheme which is a union of algebraic tori corresponding to the vertices of $\Exch_\Delta(\mathcal{D})$ glued by explicit birational transformations. It also carries a natural action of the cluster modular group~$\mathcal{G}_\Delta(\mathcal{D})$.

Suppose now that we are given a marked bordered surface $(\mathbb{S},\mathbb{M})$ satisfying an amenability condition which we formulate in Definition~\ref{def:amenable} below. If $T_0$ is an ideal triangulation of~$(\mathbb{S},\mathbb{M})$, then we get an associated quiver with potential $(Q,W)=(Q(T_0),W(T_0))$ and an associated triangulated category $\mathcal{D}=\mathcal{D}(Q,W)$. In this case, the cluster variety is known by~\cite{FockGoncharov1} to be birational to the moduli space $\mathscr{X}(\mathbb{S},\mathbb{M})$, and we prove a slight extension of the result of~\cite{BridgelandSmith} to get an isomorphism of manifolds $\Sigma(Q,W)\cong\mathscr{Q}^\pm(\mathbb{S},\mathbb{M})$.

\begin{theorem}
\label{thm:intromain}
There is a dense open set $\Sigma^*(Q,W)\subset\Sigma(Q,W)$ and a $\mathcal{G}_\Delta(\mathcal{D})$-equivariant continuous map 
\[
\widehat{F}:\Sigma^*(Q,W)\rightarrow\mathscr{X}^{\mathrm{cl}}(Q)
\]
from this set to the cluster variety. If $\sigma\in\Sigma^*(Q,W)$ lies in the cell corresponding to some vertex $\mathcal{A}\in\Exch_\Delta(\mathcal{D})$, then for $z\in\mathbb{C}$ with $-\frac{1}{2}<\Re(z)<\frac{1}{2}$ and $\Im(z)\gg0$, the point $\widehat{F}(z\cdot\sigma)$ lies in the algebraic torus corresponding to~$\mathcal{A}$.
\end{theorem}

It follows from the results of~\cite{GuptaMj} that, at least when $(\mathbb{S},\mathbb{M})$ has no interior marked points, the map $\widehat{F}$ is a local homeomorphism. We note that Theorem~\ref{thm:intromain} generalizes the main results obtained in~\cite{Allegretti17} for quivers of Dynkin type~A and that a much simpler tropical analog of the map $\widehat{F}$ appeared in~\cite{Goncharov}.

\subsection{The Riemann-Hilbert problem}

Consider again the 3-Calabi-Yau triangulated category $\mathcal{D}=\mathcal{D}(Q,W)$ associated to a quiver with potential. Its Grothendieck group $\Gamma=K(\mathcal{D})\cong\mathbb{Z}^{\oplus n}$ is a lattice of finite rank equipped with a skew form $\langle -,-\rangle$ given by the Euler form. If we are given a stability condition $\sigma$ on~$\mathcal{D}$, then part of the data defining $\sigma$ is a group homomorphism $Z:\Gamma\rightarrow\mathbb{C}$ called the central charge. For a generic choice of~$\sigma$, the methods of Donaldson-Thomas theory can be used to define a collection of integer invariants $\Omega(\gamma)\in\mathbb{Z}$ for~$\gamma\in\Gamma$. The latter are called the BPS~invariants. They satisfy the symmetry $\Omega(\gamma)=\Omega(-\gamma)$ together with a finiteness condition called the support property.

In this paper, we will be interested in a certain Riemann-Hilbert problem associated to the data $(\Gamma,Z,\Omega)$. To state this Riemann-Hilbert problem, we consider rays in~$\mathbb{C}^*$ of the form $\mathbb{R}_{>0}\cdot Z(\gamma)$ where $\gamma\in\Gamma$ is a class satisfying $\Omega(\gamma)\neq0$. Such rays are said to be active, and their union forms a diagram in~$\mathbb{C}^*$, an example of which is illustrated in Figure~\ref{fig:raydiagram}.

\begin{figure}[ht]
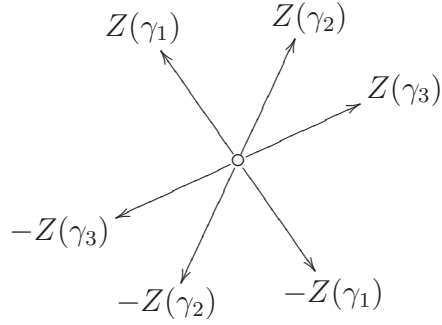

\begin{center}
\[
\xy /l1.5pc/:
{\xypolygon18"A"{~:{(2,2):}~>{}}};
(1,0)*{\circ}="a";
(3,-2.8)*{Z(\gamma_1)};
(-1,2.8)*{-Z(\gamma_1)};
(-0.5,-3)*{Z(\gamma_2)};
(2.5,3)*{-Z(\gamma_2)};
(-2.5,-1.5)*{Z(\gamma_3)};
(4.75,1.5)*{-Z(\gamma_3)};
\xygraph{
"a":"A2",
"a":"A11",
"a":"A9",
"a":"A18",
"a":"A5",
"a":"A14",
}
\endxy
\]
\end{center}
\caption{A ray diagram.\label{fig:raydiagram}}
\end{figure}

The lattice $\Gamma$ determines an object called the twisted torus:
\[
\mathbb{T}=\left\{g:\Gamma\rightarrow\mathbb{C}^*:g(\gamma_1+\gamma_2)=(-1)^{\langle\gamma_1,\gamma_2\rangle}g(\gamma_1)g(\gamma_2)\right\}.
\]
This set $\mathbb{T}$ has the natural structure of an algebraic variety whose coordinate ring is spanned as a vector space by the functions $x_{\gamma}:\mathbb{T}\rightarrow\mathbb{C}^*$ given by $x_\gamma(g)=g(\gamma)$. In Section~\ref{sec:TheRiemannHilbertProblem}, we explain how to associate, to each active ray $\ell\subset\mathbb{C}^*$, a birational automorphism $\mathbf{S}(\ell)$ of~$\mathbb{T}$. For a generic $\sigma$, the results of~\cite{Bridgeland19} imply that this transformation is given on functions by 
\[
\mathbf{S}(\ell)^*(x_\beta)=x_\beta\cdot\prod_{Z(\gamma)\in\ell}(1-x_\gamma)^{\Omega(\gamma)\cdot\langle\beta,\gamma\rangle}.
\]
This transformation is closely related to the transformations used to glue tori in the definition of the cluster variety in Section~\ref{sec:TheClusterPoissonVariety}.

The Riemann-Hilbert problem that we consider concerns maps from~$\mathbb{C}^*$ to~$\mathbb{T}$ with prescribed discontinuities along the active rays. In Section~\ref{sec:TheRiemannHilbertProblem}, we will give a careful formulation of this Riemann-Hilbert problem; for now we just give the rough idea.

\begin{problem}[\cite{Bridgeland19}]
Fix a point $\xi\in\mathbb{T}$. Construct a partially defined map 
\[
\mathcal{X}:\mathbb{C}^*\rightarrow\mathbb{T}
\]
such that the composition $\mathcal{X}_\gamma=x_\gamma\circ\mathcal{X}$ is meromorphic in the complement of the active rays for each $\gamma\in\Gamma$ and the following properties are satisfied:
\begin{enumerate}
\item[(RH1)] As $t\in\mathbb{C}^*$ crosses an active ray $\ell\subset\mathbb{C}^*$ in the counterclockwise direction, the function $\mathcal{X}(t)$ undergoes a discontinuous jump described by the formula 
\[
\mathcal{X}(t)\mapsto(\mathbf{S}(\ell))(\mathcal{X}(t)).
\]
\item[(RH2)] As $t\rightarrow0$, one has $\exp(Z(\gamma)/t)\cdot\mathcal{X}_\gamma(t)\rightarrow\xi(\gamma)$ for each $\gamma\in\Gamma$.
\item[(RH3)] As $t\rightarrow\infty$, the functions $\mathcal{X}_\gamma(t)$ for $\gamma\in\Gamma$ have at most polynomial growth.
\end{enumerate}
\end{problem}

We will be interested in the Riemann-Hilbert problem associated to a generic stability condition $\sigma\in\Sigma(Q,W)$ where $(Q,W)$ is the quiver with potential arising from an ideal triangulation of an amenable marked bordered surface $(\mathbb{S},\mathbb{M})$. In this case there is a distinguished choice of the point $\xi\in\mathbb{T}$. We will also consider a modified version of this problem which is obtained by dropping the condition (RH3). We will refer to this modified problem as the weak Riemann-Hilbert problem.

\begin{theorem}
Let $\sigma\in\Sigma(Q,W)$ be a generic stability condition where $(Q,W)$ is the quiver with potential associated to an ideal triangulation of an amenable marked bordered surface $(\mathbb{S},\mathbb{M})$. Then there is a canonical solution of the weak Riemann-Hilbert problem associated to~$\sigma$ and the distinguished point $\xi\in\mathbb{T}$. If we assume moreover that the surface $\mathbb{S}$ is closed, then it is a solution of the full Riemann-Hilbert problem.
\end{theorem}

We conjecture that the second statement in this theorem holds without the additional restriction on the surface~$\mathbb{S}$. As we will see in Section~\ref{sec:SolvingTheRiemannHilbertProblem}, the map~\eqref{eqn:map}, or equivalently the map $\widehat{F}$ of Theorem~\ref{thm:intromain}, is a crucial technical ingredient in the solution of the Riemann-Hilbert problem. This suggests the possibility of solving the Riemann-Hilbert problem for other quivers with potential by constructing a densely defined map from the space of stability conditions to the cluster variety. It is also interesting to ask if this can be reversed. That is, if we can solve the Riemann-Hilbert problem for any generic $\sigma\in\Sigma(Q,W)$, can we use this to construct a map as in Theorem~\ref{thm:intromain}, even in examples where the space of stability conditions and the cluster variety do not have alternative modular interpretations?

\subsection*{Acknowledgements.}
This project began as a collaboration with Tom~Bridgeland, who contributed many crucial insights throughout. The author acknowledges helpful conversations and correspondence with Giordano~Cotti, Laura~Fredrickson, Subhojoy~Gupta, Davide~Guzzetti, Alastair~King, Dmitry~Korotkin, Davide~Masoero, Andrew~Neitzke, Claude~Sabbah, and Richard~Wentworth. This work was supported by the European Research Council grant ERC-AdG StabilityDTCluster while the author was employed at the University of Sheffield and by the National Science Foundation grant DMS-1440140 while the author was in residence at the Mathematical Sciences Research Institute in Berkeley, California, during the Fall 2019 semester.

\section{Triangulated surfaces}

In this section, we introduce some basic definitions concerning marked bordered surfaces and ideal triangulations. Further details on this material can be found in~\cite{BridgelandSmith,FST}.

\subsection{Marked bordered surfaces}
\label{sec:MarkedBorderedSurfaces}

A \emph{marked bordered surface} is defined to be a pair $(\mathbb{S},\mathbb{M})$ where $\mathbb{S}$ is a compact, connected, oriented smooth surface with boundary, and $\mathbb{M}$ is a nonempty finite set of marked points on~$\mathbb{S}$ such that every boundary component contains at least one marked point. Marked points in the interior of~$\mathbb{S}$ are called \emph{punctures}, and the set of all punctures of the marked bordered surface will be denoted $\mathbb{P}\subset\mathbb{M}$.

To avoid various degenerate situations, we will often restrict attention to marked bordered surfaces of the following type considered in~\cite{BridgelandSmith}.

\begin{definition}
\label{def:amenable}
A marked bordered surface $(\mathbb{S},\mathbb{M})$ is \emph{amenable} if it is not one of the following:
\begin{enumerate}
\item A closed surface with a single puncture.
\item A sphere with $\leq5$ punctures.
\item An unpunctured disk with $\leq4$ marked points on its boundary.
\item A once-punctured disk with one, two, or four marked points on its boundary.
\item A twice punctured disk with two marked points on its boundary.
\item An annulus with one marked point on each boundary component.
\end{enumerate}
\end{definition}

An isomorphism of marked bordered surfaces $(\mathbb{S}_1,\mathbb{M}_1)$ and $(\mathbb{S}_2,\mathbb{M}_2)$ is an orientation preserving diffeomorphism $f:\mathbb{S}_1\rightarrow\mathbb{S}_2$ which induces a bijection $\mathbb{M}_1\rightarrow\mathbb{M}_2$. Two such isomorphisms are said to be isotopic if the underlying diffeomorphisms are related by an isotopy through diffeomorphisms $f_t:\mathbb{S}_1\rightarrow\mathbb{S}_2$ which also induce bijections $\mathbb{M}_1\rightarrow\mathbb{M}_2$. The \emph{mapping class group} $\MCG(\mathbb{S},\mathbb{M})$ is defined to be the group of all isotopy classes of isomorphisms from $(\mathbb{S},\mathbb{M})$ to itself.

It is sometimes convenient to replace $\mathbb{S}$ by the surface $\mathbb{S}'$ obtained by taking the real oriented blowup of~$\mathbb{S}$ at each puncture. This modified surface $\mathbb{S}'$ has a boundary component with no marked points corresponding to each puncture of the original surface~$\mathbb{S}$. A marked bordered surface $(\mathbb{S},\mathbb{M})$ is determined up to isomorphism by its genus $g=g(\mathbb{S})$ and a collection of nonnegative integers $\{k_1,\dots,k_d\}$ encoding the number of marked points on each boundary component of~$\mathbb{S}'$. An associated integer which will appear very often in what follows is 
\begin{equation}
\label{eqn:dimension}
n=6g-6+\sum_i(k_i+3).
\end{equation}

\subsection{Ideal triangulations}

Let $(\mathbb{S},\mathbb{M})$ be a marked bordered surface. By an \emph{arc} on $(\mathbb{S},\mathbb{M})$, we mean a smooth path $\gamma$ on~$\mathbb{S}$ connecting points of~$\mathbb{M}$ whose interior lies in the interior of~$\mathbb{S}\setminus\mathbb{M}$ and which has no self-intersections in its interior. We also require that $\gamma$ is not homotopic, relative to its endpoints, to a single point or to a path in~$\partial\mathbb{S}$ whose interior contains no marked points. A segment of the boundary of~$\mathbb{S}$ that connects two marked points (possibly coinciding) without passing through a third marked point is called a \emph{boundary segment}.

Two arcs are considered to be equivalent if they are related by a homotopy through arcs or a reversal of orientation. Two arcs are \emph{compatible} if there exist arcs in their respective equivalence classes which do not intersect in the interior of~$\mathbb{S}$. An \emph{ideal triangulation} of~$(\mathbb{S},\mathbb{M})$ is defined to be a maximal collection of pairwise compatible arcs on $(\mathbb{S},\mathbb{M})$, considered up to equivalence. Figure~\ref{fig:triangulation} shows an example of an ideal triangulation of a disk with five marked points on its boundary.

\begin{figure}[ht]
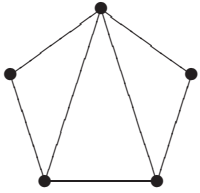

\begin{center}
\[
\xy /l1pc/:
{\xypolygon5"A"{~:{(-3,0):}}};
{"A5"\PATH~={**@{-}}'"A2"'"A4"};
"A1"*{\bullet};
"A2"*{\bullet};
"A3"*{\bullet};
"A4"*{\bullet};
"A5"*{\bullet};
\endxy
\]
\end{center}
\caption{An ideal triangulation of a disc with five marked points.}\label{fig:triangulation}
\end{figure}

When talking about an ideal triangulation of~$(\mathbb{S},\mathbb{M})$, we will always fix a collection of representatives for its arcs so that no two arcs intersect in the interior of~$\mathbb{S}$. Then a \emph{triangle} of an ideal triangulation~$T$ is defined to be the closure in~$\mathbb{S}$ of a connected component of the complement of all arcs of~$T$. An \emph{edge} of an ideal triangulation is an arc of the triangulation or a boundary segment. Any triangle is topologically a disk containing two or three distinct edges of the triangulation.

If a triangle contains just two distinct edges of the triangulation, then it is said to be \emph{self-folded}; an example is illustrated in Figure~\ref{fig:selffolded}. In this case, the edge in the interior of the triangle is called the \emph{self-folded edge} and the other edge is called the \emph{encircling edge}. The \emph{valency} of a puncture $p\in\mathbb{P}$ with respect to an ideal triangulation~$T$ is the number of half edges of~$T$ that are incident to~$p$. Note that a puncture has valency one if and only if it is contained in a self-folded triangle. An ideal triangulation will be called \emph{regular} if all punctures have valency $\geq3$. In particular, a regular ideal triangulation contains no self-folded triangles.

\begin{figure}[ht]
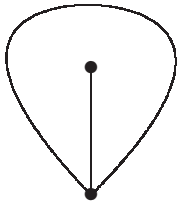

\begin{center}
\[
\xy /l0.5pc/:
(0,-4)*{}="N"; 
(0,8)*{}="S"; 
(0,0);"S" **\dir{-}; 
(0,0)*{\bullet}; 
"S"*{\bullet}; 
"S";"N" **\crv{(-12,-4) & (0,-4)}; 
"S";"N" **\crv{(12,-4.5) & (0,-4)}; 
\endxy
\]
\caption{A self-folded triangle.\label{fig:selffolded}}
\end{center}
\end{figure}

Suppose $T$ is an ideal triangulation of a marked bordered surface $(\mathbb{S},\mathbb{M})$ and $k$ is an arc of~$T$. We say that an ideal triangulation $T'$ is obtained from $T$ by a \emph{flip} of the arc~$k$ if $T'$ is different from~$T$ and there is an arc $k'$ of~$T'$ such that $T\setminus\{k\}=T'\setminus\{k'\}$. Figure~\ref{fig:flip} illustrates the triangles in a neighborhood of the arcs $k$ and~$k'$. Note that neither $k$ nor $k'$ is a self-folded edge.

\begin{figure}[ht]
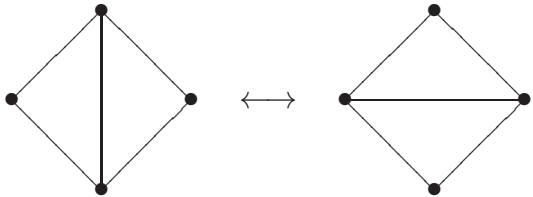

\begin{center}
\[
\xy /l1pc/:
{\xypolygon4"A"{~:{(2,2):}}};
{\xypolygon4"B"{~:{(2.5,0):}~>{}}};
{\xypolygon4"C"{~:{(0.8,0.8):}~>{}}};
{"A1"\PATH~={**@{-}}'"A3"};
"A1"*{\bullet};
"A2"*{\bullet};
"A3"*{\bullet};
"A4"*{\bullet};
\endxy
\quad
\longleftrightarrow
\quad
\xy /l1pc/:
{\xypolygon4"A"{~:{(2,2):}}};
{\xypolygon4"B"{~:{(2.5,0):}~>{}}};
{\xypolygon4"C"{~:{(0.8,0.8):}~>{}}};
{"A2"\PATH~={**@{-}}'"A4"};
"A1"*{\bullet};
"A2"*{\bullet};
"A3"*{\bullet};
"A4"*{\bullet};
\endxy
\]
\end{center}
\caption{A flip of an arc.\label{fig:flip}}
\end{figure}

It is well known that any two ideal triangulations of a marked bordered surface $(\mathbb{S},\mathbb{M})$ are related by a sequence of flips~\cite{FST}. Thus any two ideal triangulations of $(\mathbb{S},\mathbb{M})$ have the same number of arcs, namely the number $n$ from~\eqref{eqn:dimension}.

\subsection{Tagged triangulations}

Let $(\mathbb{S},\mathbb{M})$ be a marked bordered surface. A \emph{signing} is defined as a function $\epsilon:\mathbb{P}\rightarrow\{\pm1\}$ associating a sign $\epsilon(p)$ to each puncture $p\in\mathbb{P}$. A \emph{signed triangulation} of $(\mathbb{S},\mathbb{M})$ is a pair $(T,\epsilon)$ consisting of an ideal triangulation $T$ of $(\mathbb{S},\mathbb{M})$ and a signing~$\epsilon$. Two signed triangulations $(T_1,\epsilon_1)$ and $(T_2,\epsilon_2)$ are considered to be equivalent if we have $T_1=T_2$ and the signings $\epsilon_i$ differ only at punctures of valency one. A \emph{tagged triangulation} is defined as an equivalence class of signed triangulations.

Suppose $\tau$ is a tagged triangulation which is represented by a signed triangulation $(T,\epsilon)$. By a \emph{tagged arc} of $\tau$, we mean an arc of the underlying ideal triangulation~$T$. Let $(T,\epsilon')$ be another signed triangulation where $\epsilon'$ differs from~$\epsilon$ at a single puncture~$p$ of valency one with respect to~$T$. Let $j$ be the unique edge of $T$ which is incident to this puncture~$p$, and let $k$ be the encircling edge of the self-folded triangle containing~$j$. Then the tagged arc represented by~$j$ in~$(T,\epsilon)$ is considered to be equivalent to the tagged arc represented by~$k$ in the other signed triangulation $(T,\epsilon')$.

We say that a signed triangulation $(T',\epsilon)$ is obtained from the signed triangulation $(T,\epsilon)$ by a flip of an arc $k$ of~$T$ if the underlying ideal triangulation $T'$ is obtained from $T$ by a flip of~$k$. We say that a tagged triangulation $\tau'$ is obtained from the tagged triangulation $\tau$ by a flip of the tagged arc~$j$ if $\tau$ is represented by a signed triangulation $(T,\epsilon)$ and $\tau'$ by a signed triangulation~$(T',\epsilon)$ and these signed triangulations are related by a flip of~$k$. The important point is that while we cannot flip a self-folded edge in an ordinary triangulation, we can flip any tagged arc of a tagged triangulation. Thus we get an $n$-regular graph $\Tri_{\bowtie}(\mathbb{S},\mathbb{M})$ whose vertices are tagged triangulations of $(\mathbb{S},\mathbb{M})$ and where two vertices are connected by an edge if the corresponding tagged triangulations are related by a flip.

The mapping class group of $\MCG(\mathbb{S},\mathbb{M})$ acts on the set $\mathbb{P}$ of punctures, and we define the \emph{signed mapping class group} to be the corresponding semidirect product 
\[
\MCG^\pm(\mathbb{S},\mathbb{M})=\MCG(\mathbb{S},\mathbb{M})\ltimes\mathbb{Z}_2^{\mathbb{P}}.
\]
There is an action of $\MCG^\pm(\mathbb{S},\mathbb{M})$ on the set of signed triangulations, where the $\mathbb{Z}_2^\mathbb{P}$ factor acts by changing the signing. This descends to give an action on the set of tagged triangulations and in fact an action on the graph $\Tri_{\bowtie}(\mathbb{S},\mathbb{M})$ by automorphisms.

\subsection{The exchange matrix}

Let $T$ be an ideal triangulation of a marked bordered surface $(\mathbb{S},\mathbb{M})$. We will describe an $n\times n$ matrix which encodes the combinatorics of this triangulation~$T$. For each arc $j$ of~$T$, we will denote by $\pi_T(j)$ the arc defined as follows: If $j$ is the interior edge of a self-folded triangle, we let $\pi_T(j)$ be the encircling edge, and we define $\pi_T(j)=j$ otherwise. For each non-self-folded triangle $t$ of~$T$, we define a number $\varepsilon_{ij}^t$ by the following rules: 
\begin{enumerate}
\item $\varepsilon_{ij}^t=+1$ if $\pi_T(i)$ and $\pi_T(j)$ are edges of $t$ with $\pi_T(j)$ following $\pi_T(i)$ in the counterclockwise order defined by the orientation.
\item $\varepsilon_{ij}^t=-1$ if the same holds with the clockwise order.
\item $\varepsilon_{ij}^t=0$ otherwise.
\end{enumerate}
Finally, we define the $(i,j)$ element of the \emph{exchange matrix} associated to~$T$ to be the sum 
\[
\varepsilon_{ij}=\sum_t\varepsilon_{ij}^t
\]
over all non-self-folded triangles of~$T$.

\subsection{Quivers with potential}

The exchange matrix $\varepsilon_{ij}$ associated to an ideal triangulation $T$ of a marked bordered surface $(\mathbb{S},\mathbb{M})$ determines in a natural way a quiver. Namely, given the ideal triangulation $T$, we define $Q(T)$ to be the quiver whose vertices are the arcs of~$T$, with $\varepsilon_{ij}$ arrows from~$j$ to~$i$ whenever $\varepsilon_{ij}>0$. Note that since the exchange matrix is skew-symmetric, the associated quiver is 2-acyclic.

Recall that a potential for a quiver is a formal linear combination of oriented cycles. By the work of Labardini-Fragoro~\cite{LabardiniFragoso1}, we can associate to any signed triangulation $(T,\epsilon)$ of a marked bordered surface $(\mathbb{S},\mathbb{M})$ a canonical potential for the quiver $Q(T)$. We will describe this construction in the special case where the triangulation~$T$ is regular; details on the generalization of this construction to the case of non-regular ideal triangulations can be found in~\cite{LabardiniFragoso1}. In the case of regular~$T$, there are two canonical types of oriented cycles in the associated quiver~$Q(T)$. On the one hand, for each triangle $t$ whose edges are all arcs, there is a clockwise oriented 3-cycle~$\tau(t)$. On the other hand, for each puncture $p\in\mathbb{P}$, there is a counterclockwise oriented cycle~$\pi(p)$ of length at least three encircling the puncture. We define a potential $W(T,\epsilon)$ for~$Q(T)$ by 
\[
W(T,\epsilon)=\sum_t\tau(t)-\sum_p\epsilon(p)\pi(p)
\]
where the first sum runs over all triangles whose edges are arcs, and the second sum runs over all punctures.

In~\cite{DWZ}, Derksen, Weyman, and Zelevinsky defined an equivalence relation on quivers with potential known as \emph{right equivalence}. The following result was proved by Labardini-Fragoso.

\begin{theorem}[\cite{LabardiniFragoso2}, Theorem~6.1]
\label{thm:poprightequivalence}
Let $(\mathbb{S},\mathbb{M})$ be an amenable marked bordered surface. Then up to right equivalence, the quiver with potential associated to a signed triangulation of $(\mathbb{S},\mathbb{M})$ depends only on the underlying tagged triangulation.
\end{theorem}

A potential is said to be \emph{reduced} if it is a sum of cycles of length $\geq3$. If $(Q,W)$ is a quiver with reduced potential, and $k$ is a vertex of~$Q$ which is not contained in an oriented 2-cycle, then Derksen, Weyman, and Zelevinsky~\cite{DWZ} defined a new quiver with potential $\mu_k(Q,W)$ called the quiver with potential obtained by \emph{mutation} in the direction~$k$. It is well defined up to right equivalence and depends only on the right equivalence class of $(Q,W)$. The main result of~\cite{LabardiniFragoso1} on quivers with potential associated to ideal triangulations is the following.

\begin{theorem}[\cite{LabardiniFragoso1}, Theorem~30]
\label{thm:flipmutation}
Let $(\mathbb{S},\mathbb{M})$ be a marked bordered surface and $(T,\epsilon)$ a signed triangulation of $(\mathbb{S},\mathbb{M})$. If $T'$ is the ideal triangulation obtained from~$T$ by a flip of the edge~$k$, then up to right equivalence 
\[
(Q(T'),W(T',\epsilon))=\mu_k(Q(T),W(T,\epsilon)).
\]
\end{theorem}

A subtle point concerning the definition of mutation in~\cite{DWZ} is that if $(Q,W)$ is a 2-acyclic quiver with potential, then $\mu_k(Q,W)$ exists for any $k$, but it may not be 2-acyclic. We say that $(Q,W)$ is \emph{nondegenerate} if the quiver with potential obtained by applying any finite sequence of mutations exists and is 2-acyclic. A consequence of Theorems~\ref{thm:poprightequivalence} and~\ref{thm:flipmutation} is the following.

\begin{theorem}[\cite{LabardiniFragoso2}, Corollary~9.1]
Let $(\mathbb{S},\mathbb{M})$ be an amenable marked bordered surface. Then the quiver with potential associated to any tagged triangulation of~$(\mathbb{S},\mathbb{M})$ is nondegenerate.
\end{theorem}

\section{Quadratic differentials}
\label{sec:QuadraticDifferentials}

This section contains basic material on meromorphic quadratic differentials on Riemann surfaces. Further details can be found in~\cite{BridgelandSmith}.

\subsection{GMN differentials}

Let $S$ be a Riemann surface and $\omega_S$ its holomorphic cotangent bundle. Then a meromorphic \emph{quadratic differential} on $S$ is defined as a meromorphic section of the line bundle $\omega_S^{\otimes2}$. Choosing a local coordinate $z$ on~$S$, we can write this as 
\[
\phi(z)=\varphi(z)dz^{\otimes2}
\]
where $\varphi(z)$ is a meromorphic function in the local coordinate. We denote by $\Zer(\phi)$,~$\Pol(\phi)\subset S$ the sets of zeros and poles of $\phi$, respectively. The union $\Crit(\phi)=\Zer(\phi)\cup\Pol(\phi)$ forms the set of \emph{critical points} of~$\phi$.

It is important to consider the set of critical points as a disjoint union 
\[
\Crit(\phi)=\Crit_{<\infty}(\phi)\cup\Crit_\infty(\phi)
\]
where $\Crit_{<\infty}(\phi)$ consists of \emph{finite critical points}, defined as zeros and simple poles, and $\Crit_\infty(\phi)$ consists of \emph{infinite critical points}, defined as poles of order $\geq2$. We will write 
\[
S^\circ=S\setminus\Crit_\infty(\phi)
\]
for the complement of the set of infinite critical points.

In this paper, we will be concerned primarily with meromorphic quadratic differentials of the following special type.

\begin{definition}
A \emph{Gaiotto-Moore-Neitzke (GMN) differential} is a meromorphic quadratic differential $\phi$ on a compact, connected Riemann surface~$S$ satisfying the following conditions:
\begin{enumerate}
\item $\phi$ has no zero of order $>1$.
\item $\phi$ has at least one pole.
\item $\phi$ has at least one finite critical point.
\end{enumerate}
A GMN differential is said to be \emph{complete} if it has no simple poles so that every pole has order $\geq2$.
\end{definition}

\subsection{The spectral cover}

Let $\phi$ be a GMN differential on a compact Riemann surface $S$ with poles of order $m_i$ at the points $p_i\in S$. We can also view $\phi$ as a holomorphic section 
\[
\varphi\in H^0(S,\omega_S(E)^{\otimes2}), \quad E=\sum_i\left\lceil\frac{m_i}{2}\right\rceil p_i
\]
with simple zeros at the zeros and odd order poles of~$\phi$. Then the \emph{spectral cover} is defined as 
\[
\Sigma_\phi=\{(p,\psi(p)):p\in S,\psi(p)\in F_p,\psi(p)\otimes\psi(p)=\varphi(p)\}\subset F
\]
where $F$ denotes the total space of the line bundle $\omega_S(E)$. This space $\Sigma_\phi$ is a manifold because $\varphi$ is assumed to have simple zeros. The natural projection $\pi:\Sigma_\phi\rightarrow S$ is a double cover branched precisely at the simple zeros and odd order poles of~$\phi$. There is also a natural involution $\tau:\Sigma_\phi\rightarrow\Sigma_\phi$ which exchanges the two sheets of the cover and commutes with the projection map~$\pi$.

We define the \emph{hat-homology} $\widehat{H}(\phi)$ of the differential $\phi$ as the group 
\[
\widehat{H}(\phi)=H_1(\Sigma_\phi^\circ,\mathbb{Z})^-
\]
where $\Sigma_\phi^\circ=\pi^{-1}(S^\circ)$ and the superscript means we take the anti-invariant part for the action of the covering involution~$\tau$. It was proved in~\cite{BridgelandSmith}, Lemma~2.2, that the hat homology is free of finite rank. There is a natural integer-valued bilinear pairing on this group coming from the intersection pairing on $H_1(\Sigma_\phi^\circ,\mathbb{Z})$.

Finally, we note that there is a globally-defined meromorphic 1-form $\psi$ on the spectral cover with the property that 
\[
\pi^*(\phi)=\psi\otimes\psi.
\]
This 1-form $\psi$ is holomorphic on~$\Sigma_\phi^\circ$ and anti-invariant under the action of the covering involution~$\tau$. We define the \emph{period} of $\phi$ to be the group homomorphism 
\[
Z_\phi:\widehat{H}(\phi)\rightarrow\mathbb{C}, \quad \gamma\mapsto\int_\gamma\psi.
\]

\subsection{Trajectories}

Let $\phi$ be a meromorphic quadratic differential on a compact Riemann surface~$S$. Near any point of $S\setminus\Crit(\phi)$, there is a distinguished local coordinate $w$, unique up to transformations of the form $w\mapsto\pm w+\text{constant}$, with respect to which the quadratic differential $\phi$ is given by 
\[
\phi(w)=dw\otimes dw.
\]
Indeed, if we have $\phi(z)=\varphi(z)dz^{\otimes2}$ for some local coordinate $z$, then $w$ is given by $w=\int\sqrt{\varphi(z)}dz$ for some choice of the square root. The \emph{horizontal foliation} for the differential $\phi$ is the foliation of $S\setminus\Crit(\phi)$ by the lines $\Im(w)=\text{constant}$.

By a \emph{straight arc} in~$S$, we mean a smooth path $\alpha:I\rightarrow S\setminus\Crit(\phi)$, defined on an open interval $I\subset\mathbb{R}$, which makes a constant angle $\pi\theta$ with the leaves of the horizontal foliation. In terms of the distinguished local coordinate~$w$, this is equivalent to the condition that the function $\Im(w/e^{i\pi\theta})$ is constant along~$\alpha$. The phase $\theta$ of a straight arc is well defined in $\mathbb{R}/\mathbb{Z}$, and a straight arc of phase $\theta=0$ is said to be \emph{horizontal}. By convention, straight arcs will be parametrized by arc length in the flat metric induced by the distinguished local coordinates, and two straight arcs will be regarded as the same if they are related by a reparametrization of the form $t\mapsto\pm t+\text{constant}$.

A straight arc is called a \emph{trajectory} if it is not the restriction of a straight arc defined on a larger interval. Thus a horizontal trajectory is the same thing as a leaf of the horizontal foliation. A \emph{saddle connection} is a trajectory of some phase~$\theta$ whose domain of definition is a finite length interval. A saddle connection is said to be \emph{closed} if its endpoints coincide. Note that if a trajectory intersects itself in $S\setminus\Crit(\phi)$ then it must be periodic and have domain $I=\mathbb{R}$. In this case it is called a \emph{closed trajectory}. By a \emph{finite-length trajectory}, we mean either a saddle connection or a closed trajectory.

\subsection{Hat-homology classes}

Let us consider again a meromorphic quadratic differential~$\phi$ on a compact Riemann surface~$S$. If $\alpha:I\rightarrow S$ is a finite-length trajectory for~$\phi$ which is horizontal, then we can consider the preimage $\widehat{\alpha}=\pi^{-1}(\alpha)$ of this trajectory in the spectral cover~$\Sigma_\phi$. This preimage is a closed curve which may be disconnected if $\alpha$ is a closed trajectory. As we have seen, there is a canonical 1-form $\psi$ on the spectral cover with the property that $\pi^*(\phi)=\psi\otimes\psi$. We can endow the closed curve $\widehat{\alpha}$ with a canonical orientation by requiring that $\psi$ evaluated on a tangent vector to the oriented curve be real and positive. Similarly, if $\alpha:I\rightarrow S$ is a finite-length trajectory with some nonzero phase $\theta$, then we can lift $\alpha$ to a closed curve $\widehat{\alpha}$ in the spectral cover. We can once again endow this closed curve with an orientation, but in this case, we require that $\psi$ evaluated on a tangent vector to the oriented curve have positive imaginary part.

Thus we associate to any finite-length trajectory $\alpha$ of the differential a corresponding cycle $\widehat{\alpha}$ in the spectral cover. The covering involution reverses the orientation of this cycle, and so we obtain a class $\widehat{\alpha}\in\widehat{H}(\phi)$ in hat-homology, which we call the class of $\alpha$.

\subsection{Critical points}

To understand the geometry of the horizontal trajectories for a differential~$\phi$, we first consider the behavior of these trajectories near a critical point of~$\phi$. Near a zero of order $k\geq1$, it is known that the horizontal trajectories form a $(k+2)$-pronged singularity as illustrated in Figure~\ref{fig:zeros} for $k=1,2$.

\begin{figure}[ht]
\begin{center}
\[
\xy /l3pc/:
(1,0)*{}="O";  
(-0.35,0.72)*{}="U";  
(-0.75,-0.05)*{}="X1";
(-0.6,0.2)*{}="X2"; 
(-0.45,0.45)*{}="X3"; 
(-0.2,1)*{}="X4";  
(0,1.25)*{}="X5";  
(0.15,1.5)*{}="X6"; 
(1.85,1.5)*{}="Y1";
(2,1.25)*{}="Y2";
(2.2,1)*{}="Y3";
(2.45,0.45)*{}="Y4"; 
(2.6,0.2)*{}="Y5";
(2.75,-0.05)*{}="Y6"; 
(1.85,-1.5)*{}="Z1";
(1.55,-1.5)*{}="Z2"; 
(1.25,-1.5)*{}="Z3";
(0.75,-1.5)*{}="Z4";
(0.45,-1.5)*{}="Z5";
(0.15,-1.5)*{}="Z6";
(2.35,0.72)*{}="V";  
(1,-1.5)*{}="W"; 
"O";"U" **\dir{-};  
"O";"V" **\dir{-}; 
"O";"W" **\dir{-}; 
"X4";"Y3" **\crv{(0.9,0.2) & (1.1,0.2)};
"X5";"Y2" **\crv{(0.9,0.5) & (1.1,0.5)}; 
"X6";"Y1" **\crv{(0.9,0.8) & (1.1,0.8)};
"Y4";"Z3" **\crv{(1.35,0) & (1.15,0)};
"Y5";"Z2" **\crv{(1.5,-0.2) & (1.5,-0.3)};
"Y6";"Z1" **\crv{(1.65,-0.4) & (1.85,-0.6)};
"Z4";"X3" **\crv{(0.85,0) & (0.65,0)};
"Z5";"X2" **\crv{(0.5,-0.3) & (0.5,-0.2)};
"Z6";"X1" **\crv{(0.15,-0.6) & (0.35,-0.4)};
(1,0)*{\times};
(1,2)*{k=1};
\endxy
\qquad 
\xy /l3pc/:
(1,0)*{}="O";  
(1,-1.75)*{}="T"; 
(-0.75,0)*{}="U"; 
(1,1.75)*{}="V";  
(2.75,0)*{}="W"; 
(-0.75,-0.75)*{}="U1";
(-0.75,-0.5)*{}="U2";
(-0.75,-0.25)*{}="U3"; 
(-0.75,0.25)*{}="U4"; 
(-0.75,0.5)*{}="U5";
(-0.75,0.75)*{}="U6";
(0.15,1.75)*{}="V1";  
(0.45,1.75)*{}="V2";
(0.75,1.75)*{}="V3";
(1.25,1.75)*{}="V4";  
(1.55,1.75)*{}="V5"; 
(1.85,1.75)*{}="V6"; 
(2.75,0.75)*{}="W1";
(2.75,0.5)*{}="W2";
(2.75,0.25)*{}="W3"; 
(2.75,-0.25)*{}="W4";
(2.75,-0.5)*{}="W5";
(2.75,-0.75)*{}="W6";
(1.85,-1.75)*{}="T1";
(1.55,-1.75)*{}="T2";
(1.25,-1.75)*{}="T3";
(0.75,-1.75)*{}="T4";
(0.45,-1.75)*{}="T5"; 
(0.15,-1.75)*{}="T6";
"O";"T" **\dir{-};   
"O";"U" **\dir{-}; 
"O";"V" **\dir{-};  
"O";"W" **\dir{-}; 
"U4";"V3" **\crv{(0.8,0.2) & (0.8,0.2)};
"U5";"V2" **\crv{(0.5,0.5) & (0.5,0.5)};
"U6";"V1" **\crv{(0.15,0.7) & (0.15,0.7)}; 
"V4";"W3" **\crv{(1.2,0.2) & (1.2,0.2)};
"V5";"W2" **\crv{(1.5,0.5) & (1.5,0.5)}; 
"V6";"W1" **\crv{(1.85,0.7) & (1.85,0.7)};
"W4";"T3" **\crv{(1.2,-0.2) & (1.2,-0.2)};
"W5";"T2" **\crv{(1.5,-0.5) & (1.5,-0.5)}; 
"W6";"T1" **\crv{(1.85,-0.7) & (1.85,-0.7)};
"T4";"U3" **\crv{(0.8,-0.2) & (0.8,-0.2)}; 
"T5";"U2" **\crv{(0.5,-0.5) & (0.5,-0.5)};
"T6";"U1" **\crv{(0.15,-0.7) & (0.15,-0.7)}; 
(1,0)*{\times};
(1,2.25)*{k=2};
\endxy
\qquad
\dots
\]
\end{center}
\caption{The horizontal foliation near a zero of order~$k\geq1$.\label{fig:zeros}}
\end{figure}

On the other hand, near a pole of order two, there is a local coordinate $t$ such that the differential can be written 
\[
\phi(t)=\frac{r}{t^2}dt^{\otimes2}
\]
for some well defined constant $r\in\mathbb{C}^*$. We define the \emph{residue} of $\phi$ at $p$ to be the quantity 
\[
\Res_p(\phi)=\pm4\pi i\sqrt{r},
\]
which is well defined up to a sign.

Near the double pole $p$, the horizontal foliation can exhibit three possible behaviors in the $t$-plane depending on the value of the residue at~$p$: 
\begin{enumerate}
\item If $\Res_p(\phi)\in\mathbb{R}$, then the horizontal trajectories are concentric circles centered on the pole.
\item If $\Res_p(\phi)\in i\mathbb{R}$, then the horizontal trajectories are radial arcs emanating from the pole.
\item If $\Res_p(\phi)\not\in\mathbb{R}\cup i\mathbb{R}$, then the horizontal trajectories are logarithmic spirals that wrap around the pole.
\end{enumerate}
Figure~\ref{fig:doublepole} illustrates the three types of foliations.

\begin{figure}[ht]
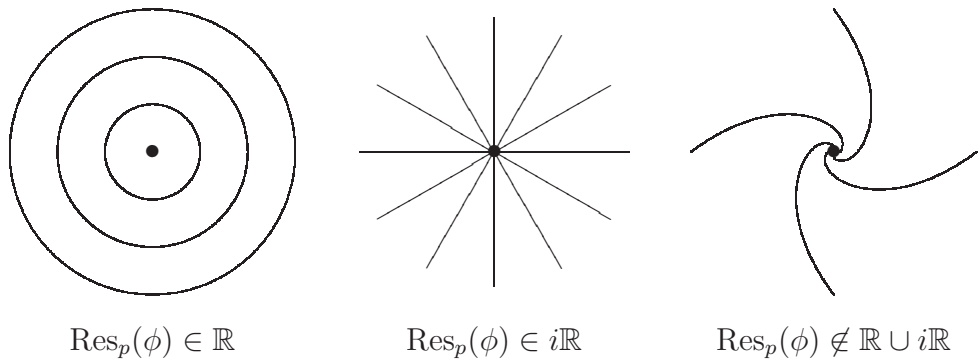

\begin{center}
\[
\xy /l1.5pc/:
(1,-3)*\xycircle(3,3){-};
(1,-2)*\xycircle(2,2){-};
(1,-1)*\xycircle(1,1){-};
(1,0)*{\bullet};
(1,4)*{\Res_p(\phi)\in\mathbb{R}};
\endxy
\qquad
\qquad
\xy /l1.5pc/:
{\xypolygon12"A"{~:{(2,2):}~>{}}};
{(1,0)\PATH~={**@{-}}'"A1"};
{(1,0)\PATH~={**@{-}}'"A2"};
{(1,0)\PATH~={**@{-}}'"A3"};
{(1,0)\PATH~={**@{-}}'"A4"};
{(1,0)\PATH~={**@{-}}'"A5"};
{(1,0)\PATH~={**@{-}}'"A6"};
{(1,0)\PATH~={**@{-}}'"A7"};
{(1,0)\PATH~={**@{-}}'"A8"};
{(1,0)\PATH~={**@{-}}'"A9"};
{(1,0)\PATH~={**@{-}}'"A10"};
{(1,0)\PATH~={**@{-}}'"A11"};
{(1,0)\PATH~={**@{-}}'"A12"};
(1,0)*{\bullet};
(1,4)*{\Res_p(\phi)\in i\mathbb{R}};
\endxy
\qquad
\xy /l1.5pc/:
(1,0)*{\bullet};
(1.13,0);(-2,0) **\crv{(1.5,0.5) & (0,1.5)};
(1,-0.13);(1,3) **\crv{(1.5,-0.5) & (2.5,1)};
(0.87,0);(4,0) **\crv{(0.5,-0.5) & (2,-1.5)};
(1,0.13);(1,-3) **\crv{(0.5,0.5) & (-0.5,-1)};
(1,4)*{\Res_p(\phi)\not\in\mathbb{R}\cup i\mathbb{R}};
\endxy
\]
\end{center}
\caption{The horizontal foliation near a pole of order two.\label{fig:doublepole}}
\end{figure}

Finally, if $p$ is a pole of order $m\geq3$, then after choosing a local coordinate $z$ with $z(p)=0$, we can write $\phi(z)=\varphi(z)dz^{\otimes2}$ where 
\[
\varphi(z)=a_0z^{-m}+a_1z^{-m+1}+a_1z^{-m+2}+\dots.
\]
We define the \emph{asymptotic horizontal directions} of $\phi$ at~$p$ to be the $m-2$ tangent vectors to the rays defined by the condition that the expression $a_0\cdot z^{2-m}$ is real and positive. The reason for the name is that there is a neighborhood $U$ of $p$ such that any horizontal trajectory that enters $U$ eventually tends to $p$ and is asymptotic to one of the asymptotic horizontal directions. We illustrate this in Figure~\ref{fig:higherorderpole} for $m=5,6$.

\begin{figure}[ht]
\begin{center}
\[
\xy /l3pc/:
{\xypolygon3"T"{~:{(2,0):}~>{}}},
{\xypolygon3"S"{~:{(1.5,0):}~>{}}},
{\xypolygon3"R"{~:{(1,0):}~>{}}},
(1,0)*{}="O"; 
(-0.35,0.72)*{}="U"; 
(2.35,0.72)*{}="V"; 
(1,-1.5)*{}="W"; 
"O";"U" **\dir{-}; 
"O";"V" **\dir{-}; 
"O";"W" **\dir{-}; 
"O";"T1" **\crv{(2,0.75) & (2.25,1.85)};
"O";"T1" **\crv{(0,0.75) & (-0.25,1.85)};
"O";"T2" **\crv{(0,0.4) & (-1.2,0.25)};
"O";"T2" **\crv{(1,-1) & (0,-2.2)};
"O";"T3" **\crv{(1,-1) & (2,-2.2)};
"O";"T3" **\crv{(2,0.4) & (3.2,0.25)};
"O";"S1" **\crv{(1.75,0.56) & (2,1.5)};
"O";"S1" **\crv{(0.25,0.56) & (0,1.5)};
"O";"S2" **\crv{(0,0.3) & (-0.9,0.19)};
"O";"S2" **\crv{(0.9,-0.8) & (0.3,-1.7)};
"O";"S3" **\crv{(2,0.3) & (2.9,0.19)};
"O";"S3" **\crv{(1.1,-0.8) & (1.7,-1.7)};
"O";"R1" **\crv{(1.5,0.5) & (1.75,1)};
"O";"R1" **\crv{(0.5,0.5) & (0.25,1)};
"O";"R2" **\crv{(0.5,0.1) & (-0.3,0.25)};
"O";"R2" **\crv{(0.75,-0.8) & (0.5,-1)};
"O";"R3" **\crv{(1.5,0.1) & (2.3,0.25)};
"O";"R3" **\crv{(1.25,-0.8) & (1.5,-1)};
(1,0)*{\bullet};
(1,2.25)*{m=5};
\endxy
\qquad
\xy /l3pc/:
{\xypolygon4"A"{~:{(2,0):}~>{}}},
{\xypolygon4"B"{~:{(1.5,0):}~>{}}},
{\xypolygon4"C"{~:{(1,0):}~>{}}},
(1,0)*{}="O"; 
(1,-1.75)*{}="T"; 
(-0.75,0)*{}="U"; 
(1,1.75)*{}="V"; 
(2.75,0)*{}="W"; 
"O";"T" **\dir{-}; 
"O";"U" **\dir{-}; 
"O";"V" **\dir{-}; 
"O";"W" **\dir{-}; 
"O";"A1" **\crv{(1,1.5) & (1.5,2.25)};
"O";"A1" **\crv{(2.5,0) & (3.25,0.5)};
"O";"A2" **\crv{(1,1.5) & (0.5,2.25)};
"O";"A2" **\crv{(-0.5,0) & (-1.25,0.5)};
"O";"A3" **\crv{(1,-1.5) & (0.5,-2.25)};
"O";"A3" **\crv{(-0.5,0) & (-1.25,-0.5)};
"O";"A4" **\crv{(1,-1.5) & (1.5,-2.25)};
"O";"A4" **\crv{(2.5,0) & (3.25,-0.5)};
"O";"B1" **\crv{(1,1) & (1.5,1.6)};
"O";"B1" **\crv{(2,0) & (2.6,0.5)};
"O";"B2" **\crv{(1,1) & (0.5,1.6)};
"O";"B2" **\crv{(0,0) & (-0.6,0.5)};
"O";"B3" **\crv{(1,-1) & (0.5,-1.6)};
"O";"B3" **\crv{(0,0) & (-0.6,-0.5)};
"O";"B4" **\crv{(1,-1) & (1.5,-1.6)};
"O";"B4" **\crv{(2,0) & (2.6,-0.5)};
"O";"C1" **\crv{(1,0.75) & (1.25,1)};
"O";"C1" **\crv{(1.75,0) & (2,0.25)};
"O";"C2" **\crv{(1,0.75) & (0.75,1)};
"O";"C2" **\crv{(0.25,0) & (0,0.25)};
"O";"C3" **\crv{(1,-0.75) & (0.75,-1)};
"O";"C3" **\crv{(0.25,0) & (0,-0.25)};
"O";"C4" **\crv{(1,-0.75) & (1.25,-1)};
"O";"C4" **\crv{(1.75,0) & (2,-0.25)};
(1,0)*{\bullet};
(1,2.25)*{m=6};
\endxy
\qquad
\dots
\]
\end{center}
\caption{The horizontal foliation near a pole of order $m\geq3$.\label{fig:higherorderpole}}
\end{figure}

\subsection{Global trajectories}

We will now study the global behavior of horizontal trajectories for a GMN differential~$\phi$ on a compact Riemann surface~$S$. It was shown in Sections~9--11 of~\cite{Strebel} (see also Section 3.4 of~\cite{BridgelandSmith}) that every horizontal trajectory of~$\phi$ is one of the following:
\begin{enumerate}
\item A \emph{saddle trajectory}, which connects two finite critical points of~$\phi$.
\item A \emph{separating trajectory}, which connects a finite and an infinite critical point of~$\phi$.
\item A \emph{generic trajectory}, which connects two infinite critical points of~$\phi$.
\item A \emph{closed trajectory}, which is a simple closed curve in $S\setminus\Crit(\phi)$.
\item A \emph{recurrent trajectory}, which has a limit set with nonempty interior in~$S$.
\end{enumerate}
Since there are only finitely many horizontal trajectories incident to any finite critical point of a quadratic differential, the horizontal foliation includes at most finitely many saddle trajectories and separating trajectories. If we remove these from~$S$, together with the critical points of~$\phi$, then the remaining surface splits as a disjoint union of connected components, which can be classified as follows:
\begin{enumerate}
\item A \emph{horizontal strip} is a maximal domain in~$S$ which is mapped by the distinguished local coordinate to a region 
\[
\{w\in\mathbb{C}:a<\Im(w)<b\}\subset\mathbb{C}.
\]
The trajectories in a horizontal strip are generic, connecting two (not necessarily distinct) poles. Each component of the boundary is composed of saddle trajectories and separating trajectories.

\item A \emph{half plane} is a maximal domain in~$S$ which is mapped by the distinguished local coordinate to 
\[
\{w\in\mathbb{C}:\Im(w)>0\}\subset\mathbb{C}.
\]
The trajectories in a half plane are generic, connecting a fixed pole of order $>2$ to itself. The boundary is composed of saddle trajectories and separating trajectories.

\item A \emph{ring domain} is a maximal domain in~$S$ which is mapped by the distinguished local coordinate to 
\[
\{w\in\mathbb{C}:a<|w|<b\}\subset\mathbb{C}^*.
\]
The trajectories in a ring domain are closed trajectories which are mapped to concentric circles by the distinguished local coordinate. Each component of the boundary is either composed of saddle trajectories or is a double pole with real residue. If one of the boundary components is a double pole, then the ring domain is said to be \emph{degenerate}.

\item A \emph{spiral domain} is the interior of the closure of a recurrent trajectory. The boundary of a spiral domain is composed of saddle trajectories.
\end{enumerate}

\subsection{Moduli spaces}

A quadratic differential $\phi$ on a compact Riemann surface $S$ with at least one pole determines a marked bordered surface $(\mathbb{S},\mathbb{M})$ by the following construction. To define the surface $\mathbb{S}$, we take the underlying smooth surface of~$S$ and perform an oriented real blowup at each pole of the differential $\phi$ having order $>2$. Then the set $\mathbb{M}$ consists of the poles of $\phi$ of order $\leq2$ considered as points in the interior of~$\mathbb{S}$, together with the points on the boundary of $\mathbb{S}$ corresponding to the asymptotic horizontal directions.

Now suppose we fix a marked bordered surface $(\mathbb{S},\mathbb{M})$. If $\phi$ is a GMN differential on a compact Riemann surface~$S$ as above, then a \emph{marking} of $(S,\phi)$ by $(\mathbb{S},\mathbb{M})$ is defined to be an isotopy class of isomorphisms between $(\mathbb{S},\mathbb{M})$ and the marked bordered surface defined by~$(S,\phi)$. A \emph{marked quadratic differential} is a triple $(S,\phi,\theta)$ where $S$ is a compact Riemann surface equipped with a GMN differential $\phi$ and $\theta$ is a marking of the pair $(S,\phi)$ by~$(\mathbb{S},\mathbb{M})$. Two such triples $(S_1,\phi_1,\theta_1)$ and $(S_2,\phi_2,\theta_2)$ are considered to be equivalent if there is a biholomorphism $S_1\rightarrow S_2$ between the underlying Riemann surfaces which preserves the quadratic differentials $\phi_i$ and commutes with the markings $\theta_i$ in the obvious way. There is a moduli space $\mathscr{Q}(\mathbb{S},\mathbb{M})$ parametrizing equivalence classes of marked quadratic differentials. We will see later that under mild assumptions it has the structure of a complex manifold of dimension~$n$ given by~\eqref{eqn:dimension}. The group $\MCG(\mathbb{S},\mathbb{M})$ acts naturally on this space by changing the marking.

It will be important later to specify additional data associated with a pole of order two of a quadratic differential. Recall that if $p$ is a pole of~$\phi$ of order two, then the residue $\Res_p(\phi)$ is well defined up to a sign. We define a \emph{signing} for $\phi$ as a choice of sign for the residue at each pole of order two. A \emph{signed differential} is a quadratic differential together with a signing. There is a branched cover 
\[
\mathscr{Q}^\pm(\mathbb{S},\mathbb{M})\rightarrow\mathscr{Q}(\mathbb{S},\mathbb{M})
\]
obtained by choosing a signing for each differential with a pole of order two. This cover has degree $2^{|\mathbb{P}|}$ and is branched precisely over the locus of quadratic differentials with simple poles. The group $\MCG^\pm(\mathbb{S},\mathbb{M})$ acts on this space by changing the markings and signings.

\subsection{Saddle-free differentials}

A GMN differential $\phi$ on a Riemann surface~$S$ is called \emph{saddle-free} if it has no horizontal saddle connections. If $\phi$ is a saddle-free differential for which $\Crit_\infty(\phi)$ is nonempty, it was shown in Lemma~3.1 of~\cite{BridgelandSmith} that $\phi$ has no closed or recurrent trajectories. It follows in this case that after removing the finitely many separating trajectories from $S\setminus\Crit(\phi)$, we obtain an open surface which is a disjoint union of horizontal strips and half planes.

In particular, if $\phi$ is a complete, saddle-free GMN differential on a Riemann surface~$S$, then the separating trajectories of~$\phi$ divide the surface $S$ into horizontal strips and half planes. Choosing one generic trajectory from each of the horizontal strips defines an ideal triangulation $T(\phi)$ of the associated marked bordered surface. This ideal triangulation $T(\phi)$ is called the \emph{WKB triangulation}. Figure~\ref{fig:WKBtriangulation} shows an example of a collection of horizontal strips and the corresponding triangles. In this picture, we indicate zeros of the quadratic differential by~$\times$ and poles by~$\bullet$.

\begin{figure}[ht]
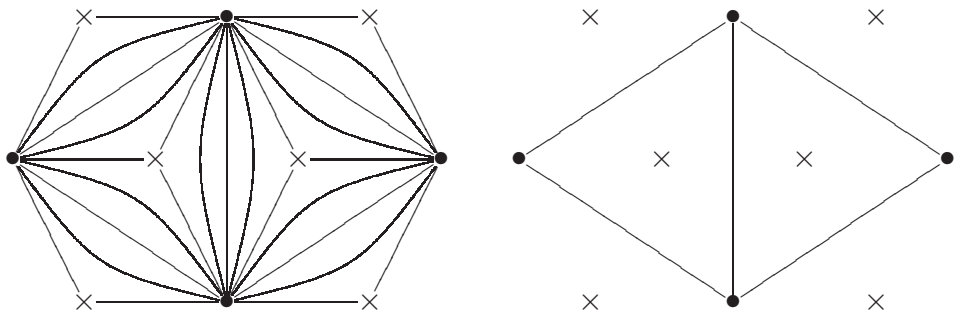

\begin{center}
\[
\xy /l2.25pc/:
(2,-2)*{\times}="11";
(0,-2)*{\bullet}="21";
(-2,-2)*{\times}="31";
(3,0)*{\bullet}="12";
(1,0)*{\times}="22";
(-1,0)*{\times}="32";
(-3,0)*{\bullet}="42";
(2,2)*{\times}="13";
(0,2)*{\bullet}="23";
(-2,2)*{\times}="33";
{"11"\PATH~={**@{-}}'"21"},
{"21"\PATH~={**@{-}}'"31"},
{"12"\PATH~={**@{-}}'"22"},
{"32"\PATH~={**@{-}}'"42"},
{"13"\PATH~={**@{-}}'"23"},
{"23"\PATH~={**@{-}}'"33"},
{"11"\PATH~={**@{-}}'"12"},
{"21"\PATH~={**@{-}}'"22"},
{"21"\PATH~={**@{-}}'"32"},
{"31"\PATH~={**@{-}}'"42"},
{"12"\PATH~={**@{-}}'"13"},
{"22"\PATH~={**@{-}}'"23"},
{"32"\PATH~={**@{-}}'"23"},
{"42"\PATH~={**@{-}}'"33"},
{"12"\PATH~={**@{-}}'"21"},
{"21"\PATH~={**@{-}}'"42"},
{"12"\PATH~={**@{-}}'"23"},
{"23"\PATH~={**@{-}}'"42"},
{"21"\PATH~={**@{-}}'"23"},
"12";"21" **\crv{(1.84,-1.56) & (1.84,-1.56)};
"12";"21" **\crv{(1.16,-0.44) & (1.16,-0.44)};
"12";"23" **\crv{(1.16,0.44) & (1.16,0.44)};
"12";"23" **\crv{(1.84,1.56) & (1.84,1.56)};
"21";"42" **\crv{(-1.84,-1.56) & (-1.84,-1.56)};
"21";"42" **\crv{(-1.16,-0.44) & (-1.16,-0.44)};
"23";"42" **\crv{(-1.84,1.56) & (-1.84,1.56)};
"23";"42" **\crv{(-1.16,0.44) & (-1.16,0.44)};
"21";"23" **\crv{(-0.5,0) & (-0.5,0)};
"21";"23" **\crv{(0.5,0) & (0.5,0)};
\endxy
\qquad
\xy /l2.25pc/:
(2,-2)*{\times}="11";
(0,-2)*{\bullet}="21";
(-2,-2)*{\times}="31";
(3,0)*{\bullet}="12";
(1,0)*{\times}="22";
(-1,0)*{\times}="32";
(-3,0)*{\bullet}="42";
(2,2)*{\times}="13";
(0,2)*{\bullet}="23";
(-2,2)*{\times}="33";
{"12"\PATH~={**@{-}}'"21"},
{"21"\PATH~={**@{-}}'"42"},
{"12"\PATH~={**@{-}}'"23"},
{"23"\PATH~={**@{-}}'"42"},
{"21"\PATH~={**@{-}}'"23"},
\endxy
\]
\end{center}
\caption{Construction of the WKB triangulation.\label{fig:WKBtriangulation}}
\end{figure}

If we choose a signing for the differential $\phi$, we obtain in a natural way a signed triangulation of the associated marked bordered surface $(\mathbb{S},\mathbb{M})$. To see this, note that if $p\in\mathbb{P}$ is a puncture with real residue, then $p$ is a double pole at the center of a ring domain. But the boundary of such a ring domain consists of saddle connections, contradicting the assumption that $\phi$ is saddle-free. It follows that the residue $\Res_p(\phi)$ is not real. Suppose we are given a choice of signing for the differential $\phi$. Then the signing $\epsilon=\epsilon(\phi)$ of the WKB triangulation $T(\phi)$ is defined so that for a puncture $p\in\mathbb{P}$, we have 
\[
\epsilon(p)\cdot\Res_p(\phi)\in\mathbb{H}
\]
where $\mathbb{H}\subset\mathbb{C}$ denotes the upper half plane. We refer to the pair $(T(\phi),\epsilon(\phi))$ as the \emph{signed WKB triangulation} of~$\phi$.

\subsection{BPS invariants}

We will now define an integer which ``counts'' finite length trajectories with a given class in hat-homology. We start by defining a GMN differential $\phi$ to be \emph{generic} if, for any two hat-homology classes $\gamma_1$,~$\gamma_2\in\widehat{H}(\phi)$, we have 
\[
\mathbb{R}\cdot Z_\phi(\gamma_1)=\mathbb{R}\cdot Z_\phi(\gamma_2) \implies \mathbb{Z}\cdot\gamma_1=\mathbb{Z}\cdot\gamma_2.
\]
Let $\phi$ be a generic GMN differential. Note that a closed trajectory of phase $\theta$ lies in a ring domain for the differential $e^{i\pi\theta}\cdot\phi$, and any other closed trajectory in this ring domain has the same class. Therefore we can speak of the class of the ring domain. The \emph{BPS invariant} associated to $\phi$ and a class $\gamma\in\widehat{H}(\phi)$ is the integer 
\begin{align*}
\Omega_\phi(\gamma) &= |\{\text{non-closed saddle connections of class $\pm\gamma$}\}| \\
&\quad - 2\cdot|\{\text{nondegenerate ring domains of class $\pm\gamma$}\}|.
\end{align*}
It was shown in~\cite{BridgelandSmith} that the integers defined in this way count stable objects in certain categories of quiver representations. The reason for the coefficient $-2$ in the above formula is that a ring domain leads to a moduli space of quiver representations isomorphic to~$\mathbb{P}^1$. Further details can be found in~\cite{BridgelandSmith}.

\section{Framed local systems}

In this section, we recall the definition of the moduli space of framed local systems from~\cite{FockGoncharov1} and describe the cluster structure of this moduli space. Our treatment is based on~\cite{AllegrettiBridgeland}.

\subsection{Framed local systems}

Let $(\mathbb{S},\mathbb{M})$ be a marked bordered surface, and let $\mathbb{S}^*=\mathbb{S}\setminus\mathbb{P}$ be the associated punctured surface. In the following, we will always write $G$ for the group $\PGL_2(\mathbb{C})$. Recall that a $G$-local system is defined as a principal $G$-bundle equipped with a flat connection. If $\mathcal{G}$ is a $G$-local system on~$\mathbb{S}^*$, then since the group $G$ has a natural left action on~$\mathbb{P}^1$, we can form the associated bundle 
\[
\mathcal{L}=\mathcal{G}\times_G\mathbb{P}^1.
\]
For each marked point $p\in\mathbb{M}$, let us fix a small contractible open neighborhood $p\in U(p)\subset\mathbb{S}$.

\begin{definition}
If $\mathcal{G}$ is a $G$-local system on the surface $\mathbb{S}^*$ then a \emph{framing} for $\mathcal{G}$ is a choice of flat section $\ell(p)$ of the associated bundle $\mathcal{L}$ over each of the sets $V(p)=U(p)\cap\mathbb{S}^*$. A \emph{framed $G$-local system} $(\mathcal{G},\ell(p))$ on $(\mathbb{S},\mathbb{M})$ is a $G$-local system together with a framing $\ell(p)$.
\end{definition}

An isomorphism of two framed $G$-local systems $(\mathcal{G}_1,\ell_1(p))$ and $(\mathcal{G}_2,\ell_2(p))$ on $(\mathbb{S},\mathbb{M})$ is an isomorphism $\theta:\mathcal{G}_1\rightarrow\mathcal{G}_2$ of the underlying $G$-local systems on~$\mathbb{S}^*$ which preserves the framings in the sense that $\theta(\ell_1(p))=\ell_2(p)$ for all $p\in\mathbb{M}$.

\subsection{Moduli spaces}

Let us fix a basepoint $x\in\mathbb{S}^*$. By a \emph{rigidified framed $G$-local system} on~$(\mathbb{S},\mathbb{M})$, we mean a framed $G$-local system $(\mathcal{G},\ell(p))$ together with a chosen point $s$ of the fiber~$\mathcal{G}_x$. An isomorphism of rigidified framed $G$-local systems $(\mathcal{G}_1,\ell_1(p),s_1)$ and $(\mathcal{G}_2,\ell_2(p),s_2)$ is an isomorphism~$\theta$ of the underlying framed $G$-local systems $(\mathcal{G}_i,\ell_i(p))$ satisfying $\theta(s_1)=s_2$. We will write $X(\mathbb{S},\mathbb{M})$ for the set of all isomorphism classes of rigidified framed $G$-local systems on~$(\mathbb{S},\mathbb{M})$

To better understand the notion of a rigidified framed local system, let us choose, for each point $p\in\mathbb{M}$, a path $\beta_p$ connecting $x$ to~$p$ whose interior lies in~$\mathbb{S}^*$. For each puncture $p\in\mathbb{P}$, we can define an element $\delta_p\in\pi_1(\mathbb{S}^*,x)$ by traveling from $x$ to~$V(p)$ along the path~$\beta_p$, encircling $p$ counterclockwise by a small loop, and then returning to $x$ along $\beta_p$. Let us consider the complex quasi-projective variety 
\[
\mathcal{V}\subset\Hom(\pi_1(\mathbb{S}^*,x),G)\times(\mathbb{P}^1)^{\mathbb{M}}
\]
consisting of pairs $(\rho,\lambda)$ such that $\rho(\delta_p)(\lambda(p))=\lambda(p)$ for all $p\in\mathbb{P}$. By Lemma~4.2 of~\cite{AllegrettiBridgeland}, there is a natural bijection between the set $X(\mathbb{S},\mathbb{M})$ and this quasi-projective variety~$\mathcal{V}$. Under this bijection, the representation $\rho$ is the monodromy representation of the local system and the point $\lambda(p)\in\mathbb{P}^1$ is defined by parallel transporting the framing $\ell(p)$ along the path $\beta_p$ into the fiber over~$x$. This fiber $\mathcal{G}_x$ is a $G$-torsor, and so we can use the chosen point $s$ to identify $\mathcal{L}_x$ with~$\mathbb{P}^1$.

There is an action of the group $G$ on the set of isomorphism classes of rigidified framed local systems $(\mathcal{G},\ell(p),s)$. An element $g\in G$ acts by fixing the underlying framed local system and mapping $s\mapsto s\cdot g$. The corresponding action on $\mathcal{V}$ is given by 
\[
g\cdot(\rho,\lambda)=(g\cdot\rho\cdot g^{-1},g\circ\lambda).
\]
We define the moduli stack of framed $G$-local systems on~$(\mathbb{S},\mathbb{M})$ as the quotient 
\[
\mathscr{X}(\mathbb{S},\mathbb{M})=X(\mathbb{S},\mathbb{M})/G
\]
by this group action.

\subsection{Coordinates from ideal triangulations}

Our next goal is to describe the cluster structure of the space $\mathscr{X}(\mathbb{S},\mathbb{M})$ following~\cite{FockGoncharov1}. To do this, it is convenient to take the universal cover of our surface. More precisely, we will equip the surface $\mathbb{S}'=\mathbb{S}\setminus\mathbb{M}$ with a complete, finite-area hyperbolic metric with totally geodesic boundary and consider its universal cover. In the chosen metric, the surface $\mathbb{S}'$ has a cusp at each of the deleted marked points, and its universal cover is a subset of the hyperbolic plane $\mathbb{H}$ with totally geodesic boundary. The cusps of the surface $\mathbb{S}'$ correspond to points on the boundary $\partial\mathbb{H}$ of~$\mathbb{H}$. The set of all such points is called the \emph{Farey set} and denoted $\mathcal{F}_\infty(\mathbb{S},\mathbb{M})$. As a cyclically ordered set, it is independent of the choice of hyperbolic metric. The action of $\pi_1(\mathbb{S}^*)$ by deck transformations on the universal cover gives rise to an action of $\pi_1(\mathbb{S}^*)$ on the Farey set.

A point of the variety $X(\mathbb{S},\mathbb{M})$ naturally determines a map $\psi:\mathcal{F}_\infty(\mathbb{S},\mathbb{M})\rightarrow\mathbb{P}^1$. Indeed, let $\tilde{x}$ be a point in the closure of the hyperbolic plane that projects to~$x\in\mathbb{S}^*$. Given any point $c\in\mathcal{F}_\infty(\mathbb{S},\mathbb{M})$ corresponding to $p\in\mathbb{M}$, we can choose a path $\tilde{\beta}$ connecting the basepoint $\tilde{x}$ to~$c$ whose interior lies in the universal cover of~$\mathbb{S}'$. If we let $\beta$ denote the image of this path under the projection to~$\mathbb{S}'$, then $\psi(c)$ is defined as the point of~$\mathbb{P}^1$ obtained by parallel transporting the framing $\ell(p)$ along $\beta$ into the fiber over~$x$. Thus we get a map $\psi:\mathcal{F}_\infty(\mathbb{S},\mathbb{M})\rightarrow\mathbb{P}^1$ which satisfies the identity 
\[
\psi(\gamma c)=\rho(\gamma)\psi(c)
\]
for all $\gamma\in\pi_1(\mathbb{S}^*)$.

Now suppose we are given an ideal triangulation $T$ of~$(\mathbb{S},\mathbb{M})$ and a general point of $X(\mathbb{S},\mathbb{M})$. We can lift the arcs of~$T$ to a collection of geodesic arcs in~$\mathbb{H}$ decomposing the universal cover into triangular regions, and the endpoints of any lifted arc are points of $\mathcal{F}_\infty(\mathbb{S},\mathbb{M})$. Thus we can use the map $\psi$ to assign a point of~$\mathbb{P}^1$ to each endpoint of a lifted arc. We can then define a cluster coordinate $X_j$ for each arc $j$ of~$T$ by the following two-step procedure: 
\begin{enumerate}
\item Let $\tilde{j}$ be a lift of the arc $j$ to the universal cover. Then there are two triangles of the lifted triangulation that share the side~$\tilde{j}$, and these form a quadrilateral in~$\mathbb{H}$. Let $c_1,\dots,c_4$ be the vertices of this quadrilateral in the counterclockwise order so that $\tilde{j}$ joins the vertices~$c_1$ and~$c_3$. For each index~$i$, let $z_i=\psi(c_i)$ and define the cross ratio 
\[
Y_j=\frac{(z_1-z_2)(z_3-z_4)}{(z_2-z_3)(z_1-z_4)}.
\]
Note that there are two ways of ordering the points $c_i$, and they give the same value for the cross ratio.

\item If $j$ is not the interior edge of a self-folded triangle, then we define $X_j=Y_j$. If $j$ is the interior edge of a self-folded triangle, let $k$ be the encircling edge of this triangle. In this case, we define $X_j=Y_jY_k$.
\end{enumerate}
In this way, we associate to a general point of $X(\mathbb{S},\mathbb{M})$ a tuple of numbers $X_j\in\mathbb{C}^*$ indexed by the arcs of the ideal triangulation~$T$. These numbers are known as the \emph{Fock-Goncharov coordinates} of the framed local system. They are invariant under the action of~$G$ on $X(\mathbb{S},\mathbb{M})$, and therefore we get a rational map 
\[
X_T:\mathscr{X}(\mathbb{S},\mathbb{M})\dashrightarrow(\mathbb{C}^*)^n.
\]
By Lemma~9.3 of~\cite{AllegrettiBridgeland}, this map is in fact a birational equivalence.

\subsection{Coordinates from tagged triangulations}

Let $(\mathbb{S},\mathbb{M})$ be a marked bordered surface and consider a framed $G$-local system on~$(\mathbb{S},\mathbb{M})$. Near each puncture $p\in\mathbb{P}$, there is a flat section $\ell(p)$ of the associated bundle $\mathcal{L}$. This flat section is invariant under the monodromy of the local system around~$p$. Therefore, if the monodromy has distinct eigenvalues, there are exactly two possible choices for $\ell(p)$, and we can consider the operation that exchanges these two choices.

\begin{proposition}[\cite{AllegrettiBridgeland}, Lemma~9.4]
There is a natural birational action of the group $\{\pm1\}$ on the stack $\mathscr{X}(\mathbb{S},\mathbb{M})$ of framed $G$-local systems where $-1$ acts on a framed $G$-local system by fixing the underlying local system and exchanging the two generically possible choices of framing at the puncture~$p$.
\end{proposition}

More generally, there is a birational action of the group $\{\pm1\}^{\mathbb{P}}$ on $\mathscr{X}(\mathbb{S},\mathbb{M})$. This group also acts on the set of signed and tagged triangulations in an obvious way. The \emph{Fock-Goncharov coordinate} of a framed $G$-local system $(\mathcal{G},\ell(p))$ with respect to an arc~$j$ of the signed triangulation $(T,\epsilon)$ is defined to be the Fock-Goncharov coordinate with respect to~$j$ of the framed local system obtained by applying the group element $\epsilon\in\{\pm1\}^{\mathbb{P}}$ to $(\mathcal{G},\ell(p))$, whenever this quantity is well-defined. These coordinates provide a birational map 
\[
X_{(T,\epsilon)}:\mathscr{X}(\mathbb{S},\mathbb{M})\dashrightarrow(\mathbb{C}^*)^n.
\]
By Lemma~9.6 of~\cite{AllegrettiBridgeland}, the Fock-Goncharov coordinate of a framed local system with respect to an arc of a signed triangulation depends only on the underlying tagged arc, and therefore we can speak of the Fock-Goncharov coordinate with respect to any tagged arc in a tagged triangulation. For each tagged triangulation~$\tau$ of $(\mathbb{S},\mathbb{M})$, these coordinates provide a well defined birational equivalence $X_\tau:\mathscr{X}(\mathbb{S},\mathbb{M})\dashrightarrow(\mathbb{C}^*)^n$.

\subsection{Change of coordinates}

Note that if $(T',\epsilon)$ is the signed triangulation obtained from $(T,\epsilon)$ by performing a flip of the arc~$k$, then the arcs of~$T$ are naturally in bijection with the arcs of~$T'$. Abusing notation we will use the same symbol for an arc of $T$ and the corresponding arc of~$T'$. We will denote by $X_j$ and $X_j'$ the Fock-Goncharov coordinates with respect to an arc~$j$ of the signed triangulations $(T,\epsilon)$ and $(T',\epsilon)$, respectively.

\begin{proposition}[\cite{AllegrettiBridgeland}, Proposition~9.8]
\label{prop:flipcoordinate}
Under the assumptions of the last paragraph, the coordinates $X_j$ and $X_j'$ are related by the birational transformation 
\begin{align*}
X_j'=
\begin{cases}
X_k^{-1} & \mbox{if } j=k \\
X_j{(1+X_k^{-\sgn(\varepsilon_{jk})})}^{-\varepsilon_{jk}} & \mbox{if } j\neq k
\end{cases}
\end{align*}
where $\varepsilon_{ij}$ is the exchange matrix associated to the ideal triangulation~$T$.
\end{proposition}

By what we have said, the formula in Proposition~\ref{prop:flipcoordinate} describes the rule for transforming between the coordinates associated with two tagged triangulations related by a flip of a tagged arc~$k$.

\subsection{Generic framed local systems}

We say that a point of $\mathscr{X}(\mathbb{S},\mathbb{M})$ is \emph{generic} if it lies in the domain of the map $X_\tau$ for some tagged triangulation $\tau$, and we write $\mathscr{X}^*(\mathbb{S},\mathbb{M})$ for the set of all generic framed $G$-local systems on~$(\mathbb{S},\mathbb{M})$. This set $\mathscr{X}^*(\mathbb{S},\mathbb{M})$ is an open substack of $\mathscr{X}(\mathbb{S},\mathbb{M})$. The charts $X_\tau$ for $\tau$ a tagged triangulation provide this open substack with the structure of a (possibly non-Hausdorff) complex manifold of dimension~$n$.

We note that our notation here differs from the notation used in~\cite{AllegrettiBridgeland}. There we defined the notion of a ``nondegenerate'' framed local system, and we wrote $\mathscr{X}^*(\mathbb{S},\mathbb{M})$ for the open substack of nondegenerate framed local systems. According to Theorem~1.2 in~\cite{AllegrettiBridgeland}, every nondegenerate framed local system is generic.

\section{Projective structures and monodromy}
\label{sec:ProjectiveStructuresAndMonodromy}

In this section, we recall the definition of a meromorphic projective structure from~\cite{AllegrettiBridgeland}. We describe the monodromy map from the space of meromorphic projective structures to the space of framed local systems.

\subsection{Projective structures}

We begin by recalling the notion of an ordinary holomorphic projective structure on a Riemann surface.

\begin{definition}
Let $S$ be a Riemann surface.
\begin{enumerate}
\item A \emph{chart} on~$S$ is a biholomorphism $z:U\rightarrow V$ where $U\subset S$ and $V\subset\mathbb{P}^1$ are nonempty open sets.
\item Two charts $z_1:U_1\rightarrow V_1$ and $z_2:U_2\rightarrow V_2$ are said to be \emph{projectively compatible} if the transition map $z_2\circ z_1^{-1}$ is the restriction of some element of $G=\PGL_2(\mathbb{C})$.
\item A \emph{projective structure} on~$S$ is a maximal collection of projectively compatible charts whose domains cover~$S$.
\end{enumerate}
\end{definition}

If $S$ is a Riemann surface, then there is a canonical choice of projective structure on~$S$. Indeed, by the uniformization theorem, we can write $S=\widetilde{S}/\Gamma$ where the universal cover $\widetilde{S}$ is either the disk, the complex plane, or the Riemann sphere, and $\Gamma\subset G$ is a discrete subgroup. In particular, $\widetilde{S}$ can be identified with an open subset of~$\mathbb{P}^1$, and we get a projective structure on~$S$ whose charts are local sections of the covering map. We call this the \emph{uniformizing projective structure} for~$S$.

If $\mathcal{P}$ is a projective structure on a Riemann surface~$S$ with universal cover $\pi:\widetilde{S}\rightarrow S$, then by~\cite{Hubbard}, Lemma~1, there exists an analytic map $f:\tilde{S}\rightarrow\mathbb{P}^1$ such that on any contractible open subset $U\subset S$, the composition $f\circ\pi^{-1}$ is a chart of~$\mathcal{P}$. Such a map is called a \emph{developing map} and is essentially unique: Any two developing maps differ by post-composition with an element of $G$.

If $f$ is a developing map for a projective structure $\mathcal{P}$, then for any $\gamma\in\pi_1(S)$, the composition $f\circ\gamma$ is another developing map for $\mathcal{P}$. It follows that there exists some $\rho(\gamma)\in G$ such that 
\[
f\circ\gamma=\rho(\gamma)\circ f.
\]
In this way, we obtain a group homomorphism $\rho:\pi_1(S)\rightarrow G$, well defined up to conjugation by elements of $G$. This group homomorphism is known as the \emph{monodromy} of the projective structure.

\subsection{Affine space structure}

An important property of projective structures is the relation between these objects and holomorphic quadratic differentials.

\begin{theorem}
The set of projective structures on a Riemann surface $S$ is an affine space for the vector space $H^0(S,\omega_S^{\otimes2})$ of holomorphic quadratic differentials.
\end{theorem}

For a proof of this theorem, see~\cite{Hubbard}, Section~2. To understand the statement more concretely, suppose we are given a projective structure $\mathcal{P}_1$ on~$S$ and a quadratic differential $\phi\in H^0(S,\omega_S^{\otimes2})$. For any chart $z:U\rightarrow\mathbb{C}$ belonging to the projective structure, we can express $\phi$ in the form $\phi(z)=\varphi(z)dz^{\otimes2}$. Then there is a new projective structure whose charts on~$U$ are obtained by taking ratios $w(z)=y_1(z)/y_2(z)$ where $y_1(z)$ and $y_2(z)$ are linearly independent solutions of the differential equation 
\begin{equation}
\label{eqn:diffeq}
y''(z)-\varphi(z)\cdot y(z)=0.
\end{equation}
We think of this new projective structure as the projective structure obtained by adding the differential $\phi$ to~$\mathcal{P}_1$, and we denote it by $\mathcal{P}_1+\phi$.

Conversely, given the projective structures $\mathcal{P}_1$ and $\mathcal{P}_2$, we can find charts $z_1,z_2:U\rightarrow\mathbb{C}$ of these projective structures defined on a common domain~$U$ and write $z_2=f(z_1)$. Then there is a quadratic differential given locally by the expression  
\[
\phi=-\frac{1}{2}\left(\left(\frac{f''}{f'}\right)'-\frac{1}{2}\left(\frac{f''}{f'}\right)^2\right) dz_1^{\otimes2},
\]
which is known as the \emph{Schwarzian derivative}. We think of this quadratic differential as the difference of $\mathcal{P}_2$ and $\mathcal{P}_1$ and denote it by $\phi=\mathcal{P}_2-\mathcal{P}_1$.

\subsection{Meromorphic projective structures}

We can now define a notion of projective structure on a Riemann surface with poles at a discrete set of points.

\begin{definition}[\cite{AllegrettiBridgeland}, Definition~3.1]
\label{def:meromorphicprojective}
A \emph{meromorphic projective structure} $\mathcal{P}$ on a Riemann surface~$S$ is defined to be a projective structure $\mathcal{P}^*$ on the complement $S^*=S\setminus P$ of a discrete subset $P\subset S$ such that, for any holomorphic projective structure $\mathcal{P}_0$ on~$S$, the quadratic differential $\phi$ on~$S^*$ defined as the difference $\phi=\mathcal{P}^*-\mathcal{P}_0|_{S^*}$ extends to a meromorphic quadratic differential on~$S$.
\end{definition}

Note that the meromorphic differential $\phi$ appearing in Definition~\ref{def:meromorphicprojective} is uniquely defined up to addition of holomorphic differentials. We call $\phi$ the \emph{polar differential} of the meromorphic projective structure. We say that the projective structure has a pole of order $m$ at a point $p\in S$ if $\phi$ has a pole of order~$m$ at~$p$.

\subsection{Local study of singularities}

We will now study the properties of a meromorphic projective structure in a neighborhood of a pole. Suppose $\mathcal{P}$ is a meromorphic projective structure on~$S$ and $p\in S$ is a pole of~$\mathcal{P}$ of order~$m\geq1$.  Let us choose an ordinary projective structure $\mathcal{P}_0$ on~$S$ and a chart $z:U\rightarrow\mathbb{C}$ in~$\mathcal{P}_0$ such that $z(p)=0$ and the set $U$ contains no pole other than~$p$. Then the polar differential $\phi=\mathcal{P}^*-\mathcal{P}_0|_{S^*}$ can be written $\phi(z)=\varphi(z)dz^{\otimes2}$ where 
\[
\varphi(z)=a_0z^{-m}+a_1z^{-m+1}+a_1z^{-m+2}+\dots.
\]
There are two possible behaviors depending on whether we have $m\leq2$ or $m>2$. In the first case, $p$ is called a \emph{regular singularity} while in the second case it is called an \emph{irregular singularity}.

Suppose first that $p$ is a regular singularity. In this case we have the following result.

\begin{proposition}[\cite{AllegrettiBridgeland}, Lemma~5.1]
\label{lem:eigenvalues}
If $m\leq2$ then the eigenvalues of the monodromy of $\mathcal{P}^*$ around~$p$ are 
\[
\lambda_\pm=-\exp(\pm r/2)
\]
where $r=\pm2\pi i\cdot\sqrt{1+4a_0}$.
\end{proposition}

In particular, when $r=2\pi in$ with $n\in\mathbb{Z}$, it can happen that the monodromy is the identity as an element of $\PGL_2(\mathbb{C})$. In this case, the pole $p$ is called an \emph{apparent singularity}.

Next, suppose that $p$ is an irregular singularity. In this case, when studying the monodromy of the projective structure~$\mathcal{P}^*$, it is important to take into account the Stokes data for the differential equation~\eqref{eqn:diffeq}. To do this, we define the \emph{anti-Stokes rays} of the equation~\eqref{eqn:diffeq} at $z=0$ to be the $m-2$ rays where the expression $a_0\cdot z^{2-m}$ is real and negative. The anti-Stokes rays bound $m-2$ closed sectors which we call the \emph{Stokes sectors}.

\begin{proposition}[\cite{AllegrettiBridgeland}, Theorem~5.2]
\label{prop:subdominant}
In the interior of each Stokes sector, there is a unique-up-to-scale solution $y(z)$ of the equation~\eqref{eqn:diffeq} such that $y(z)\rightarrow0$ as $z\rightarrow0$.
\end{proposition}

A solution of the type described in Proposition~\ref{prop:subdominant} is said to be \emph{subdominant} in the given Stokes sector.

\subsection{Moduli spaces}

Let $\mathcal{P}$ be a meromorphic projective structure on a compact Riemann surface~$S$ with at least one pole. The pair $(S,\mathcal{P})$ determines a corresponding marked bordered surface, namely the marked bordered surface associated to the Riemann surface $S$ and a polar differential of~$\mathcal{P}$. Note that if $\phi_1$ and~$\phi_2$ are two polar differentials for~$\mathcal{P}$, then the two differentials have the same poles and the same distinguished tangent directions at a pole of order $>2$ since the difference $\phi_1-\phi_2$ is holomorphic. Thus the marked bordered surface associated to $(S,\mathcal{P})$ is well defined.

We define a \emph{marking} of the pair $(S,\mathcal{P})$ by a marked bordered surface $(\mathbb{S},\mathbb{M})$ to be a marking of~$(S,\phi)$ where $\phi$ is a polar differential of the projective structure $\mathcal{P}$. A \emph{marked projective structure} on~$(\mathbb{S},\mathbb{M})$ is a triple $(S,\mathcal{P},\theta)$ where $S$ is a compact Riemann surface equipped with a meromorphic projective structure~$\mathcal{P}$ and $\theta$ is a marking of the pair $(S,\mathcal{P})$ by $(\mathbb{S},\mathbb{M})$. Two such triples $(S_1,\mathcal{P}_1,\theta_1)$ and $(S_2,\mathcal{P}_2,\theta_2)$ are considered to be equivalent if there is a biholomorphism $S_1\rightarrow S_2$ between the underlying Riemann surfaces which preserves the projective structures $\mathcal{P}_i$ and commutes with the markings $\theta_i$ in the obvious way.

Let us fix a marked bordered surface $(\mathbb{S},\mathbb{M})$, and if $\mathbb{S}$ has genus $g(\mathbb{S})=0$, let us assume that $|\mathbb{M}|\geq3$. In~\cite{AllegrettiBridgeland}, we showed that the set $\mathscr{P}(\mathbb{S},\mathbb{M})$ of equivalence classes of marked projective structures on~$(\mathbb{S},\mathbb{M})$ has the natural structure of a complex manifold with dimension $n$ given by~\eqref{eqn:dimension}. The group $\MCG(\mathbb{S},\mathbb{M})$ acts on this space by changing the marking.

It will be convenient to modify this space in two ways. First, we consider a dense open subset 
\[
\mathscr{P}^\circ(\mathbb{S},\mathbb{M})\subset\mathscr{P}(\mathbb{S},\mathbb{M})
\]
whose complement is the locus of projective structures with apparent singularities. Second, we consider a branched cover 
\[
\mathscr{P}^*(\mathbb{S},\mathbb{M})\rightarrow\mathscr{P}^\circ(\mathbb{S},\mathbb{M})
\]
of degree $2^{|\mathbb{P}|}$ whose points parametrize projective structures in $\mathscr{P}^\circ(\mathbb{S},\mathbb{M})$ together with a choice of eigenline for the monodromy around each pole of order $\leq2$. By~\cite{AllegrettiBridgeland}, Proposition~8.4, this space $\mathscr{P}^*(\mathbb{S},\mathbb{M})$ is again a complex manifold of dimension~$n$.

\subsection{The monodromy map}

Suppose we are given a marked projective structure $(S,\mathcal{P},\theta)$ on a marked bordered surface $(\mathbb{S},\mathbb{M})$, and let us write $S^*=S\setminus P$ for the complement of the set $P$ of poles of~$\mathcal{P}$. By definition, $\mathcal{P}$ is given by a holomorphic projective structure $\mathcal{P}^*$ on~$S^*$, and hence we get a local system $\mathcal{G}$ on~$S^*$ defined as the monodromy of this projective structure. The surface $S^*$ can be identified with $\mathbb{S}^*=\mathbb{S}\setminus\mathbb{P}$, so we can view $\mathcal{G}$ as a local system on~$\mathbb{S}^*$.

In fact, if $(S,\mathcal{P},\theta)$ is equipped with a choice of eigenline for the monodromy around each pole of order~$\leq2$ so that we have a point of $\mathscr{P}^*(\mathbb{S},\mathbb{M})$, then this local system $\mathcal{G}$ has a natural framing. To see this, let $p\in S$ be a pole of the projective structure~$\mathcal{P}$. Choose an ordinary projective structure~$\mathcal{P}_0$ on~$S$ and a chart $z:U\rightarrow\mathbb{C}$ of~$\mathcal{P}_0$ so that $z(p)=0$. Then the polar differential $\phi=\mathcal{P}^*-\mathcal{P}_0|_{S^*}$ can be written in the form $\phi(z)=\varphi(z)dz^{\otimes2}$. There is a rank two vector bundle $E\rightarrow U$ whose fiber over a point $x\in U$ parametrizes germs of solutions of the solutions of the differential equation~\eqref{eqn:diffeq} at~$x$. In Section~6.1 of~\cite{AllegrettiBridgeland}, we explained how the projectivization $\mathbb{P}(E)$ is isomorphic to the associated bundle $\mathcal{L}=\mathcal{G}\times_G\mathbb{P}^1$. Thus, if $p$ is a regular singularity, we can define the framing near~$p$ as the chosen eigenline of the monodromy around~$p$. On the other hand, if $p$ is an irregular singularity, then by Proposition~\ref{prop:subdominant}, there is a 1-dimensional space of subdominant solutions in each Stokes sector which defines the framing near the corresponding marked point on~$\partial\mathbb{S}$.

Note that the group $\MCG^\pm(\mathbb{S},\mathbb{M})$ acts on $\mathscr{P}^*(\mathbb{S},\mathbb{M})$ by changing the marking and the choice of eigenlines. It also acts birationally on $\mathscr{X}^*(\mathbb{S},\mathbb{M})$. Using above construction, one can show the following.

\begin{theorem}[\cite{AllegrettiBridgeland}, Theorem~1.1]
\label{thm:monodromymap}
Let $(\mathbb{S},\mathbb{M})$ be a marked bordered surface, and if $g(\mathbb{S})=0$, assume that $|\mathbb{M}|\geq3$. Then there is an $\MCG^\pm(\mathbb{S},\mathbb{M})$-equivariant holomorphic map 
\[
F:\mathscr{P}^*(\mathbb{S},\mathbb{M})\rightarrow\mathscr{X}^*(\mathbb{S},\mathbb{M})
\]
sending a marked projective structure to its monodromy local system with a natural framing defined by the above construction.
\end{theorem}

\section{Relating the moduli spaces}
\label{sec:RelatingTheModuliSpaces}

We will now explain how to construct a densely defined map from the space of signed quadratic differentials to the space of generic framed local systems. Later, this will be interpreted as a densely defined map from a certain space of stability conditions to a cluster variety.

\subsection{Manifold structure of moduli spaces}

Recall that a \emph{family of Riemann surfaces} is defined to be a holomorphic map $\pi:X\rightarrow B$ of complex manifolds which is everywhere submersive and whose fibers $X(b)=\pi^{-1}(b)$ have complex dimension one. Let us assume that $\pi$ is proper so that the fibers $X(b)$ are compact. Choose disjoint holomorphic sections $p_i:B\rightarrow X$ for $i=1,\dots,d$ so that each Riemann surface $X(b)$ has $d$ marked points $p_i(b)$. Let us also choose $d$ positive integers~$m_i$ and consider the effective divisor 
\[
D=\sum_im_iD_i, \quad D_i=p_i(B)\subset X
\]
which restricts to give a divisor $D(b)=\sum_im_ip_i(b)$ on each surface $X(b)$. There is a vector bundle 
\[
q:Q(X/B;D)\rightarrow B
\]
whose fiber over a point~$b\in B$ is the vector space 
\[
q:Q(X/B;D)_b=H^0(X(b),\omega_{X(b)}^{\otimes2}(D(b)))
\]
of meromorphic quadratic differentials on $X(b)$ having a pole of order $\leq m_i$ at~$p_i(b)$ for $i=1,\dots,d$, and no other poles.

Now let $(\mathbb{S},\mathbb{M})$ be a fixed marked bordered surface. As in Section~\ref{sec:MarkedBorderedSurfaces}, it is determined by its genus $g=g(\mathbb{S})$ and a collection of nonnegative integers $\{k_1,\dots,k_d\}$ encoding the number of marked points on each boundary component of a modified surface $\mathbb{S}'$. Note that the integers $m_i=k_i+2$ are the pole orders of a meromorphic quadratic differential $\phi$ on a Riemann surface $S$ with associated marked bordered surface~$(\mathbb{S},\mathbb{M})$. The set of markings of the pair $(S,\phi)$ by $(\mathbb{S},\mathbb{M})$ is either empty or is a torsor for the mapping class group $\MCG(\mathbb{S},\mathbb{M})$.

\begin{lemma}
\label{lem:coveropenset}
The manifold $Q(X/B;D)$ is either empty or contains a dense open subset parametrizing differentials with simple zeros and associated marked bordered surface $(\mathbb{S},\mathbb{M})$. In the latter case, there is a principal $\MCG(\mathbb{S},\mathbb{M})$-bundle over this open set whose fiber over a differential $\phi(b)$ on the Riemann surface $X(b)$ is the set of markings of the pair $(X(b),\phi(b))$ by~$(\mathbb{S},\mathbb{M})$.
\end{lemma}

\begin{proof}
The subset is defined as the set of quadratic differentials in $Q(X/B;D)$ whose zeros are simple and disjoint from those divisors $D_i$ for which $m_i\neq2$. This subset is easily seen to be open and dense. The result then follows immediately once one knows that the asymptotic horizontal directions vary continuously. But these directions are determined by the leading coefficient, which varies continuously.
\end{proof}

\begin{proposition}
Assume that if $g(\mathbb{S})=0$ then $|\mathbb{M}|\geq3$. Then the set $\mathscr{Q}(\mathbb{S},\mathbb{M})$ of marked quadratic differentials has the structure of a complex manifold of dimension $n$ given by~\eqref{eqn:dimension}.
\end{proposition}

\begin{proof}
Assume first that if $g=0$ then $d\geq3$. Let $B=\mathcal{T}(g,d)$ denote the Teichm\"uller space parametrizing marked Riemann surfaces of genus $g$ with $d$ marked points. There is a universal curve $\pi:X\rightarrow B$ with $d$ disjoint sections $p_1,\dots,p_d$. Thus, we can apply the above construction to get a manifold $Q(X/B;D)$ parametrizing meromorphic quadratic differentials; it follows from the Riemann-Roch theorem that this space has dimension~$n$.

By Lemma~\ref{lem:coveropenset}, we can pass to a covering space of an open subset, and this covering space parametrizes quadratic differentials together with a marking by $(\mathbb{S},\mathbb{M})$. Then $\mathscr{Q}(\mathbb{S},\mathbb{M})$ is identified with the subset of points where the marking by $(\mathbb{S},\mathbb{M})$ is compatible, after blowing down all boundary components of~$\mathbb{S}$, with the marking of the corresponding point in Teichm\"uller space. Since the covering space of Lemma~\ref{lem:coveropenset} has discrete fibers, this subset is open and hence a complex manifold. Thus we have proved the proposition under our assumption. The cases when $g=0$ and $d\leq2$ can be handled directly. In these cases, the moduli spaces $\mathscr{Q}(\mathbb{S},\mathbb{M})$ parametrize differentials on~$\mathbb{P}^1$ of the form $P(z)dz^{\otimes2}$ with $P(z)$ a Laurent polynomial. The assumption $|\mathbb{M}|\geq3$ implies that $P(z)$ is not constant, so the only possible automorphisms are maps rescaling the coordinate~$z$ by roots of unity, and these are easily seen to act nontrivially on the markings. Hence $\mathscr{Q}(\mathbb{S},\mathbb{M})$ is a manifold rather than an orbifold.
\end{proof}

Recall that $\mathscr{Q}^{\pm}(\mathbb{S},\mathbb{M})$ is defined as the branched cover of $\mathscr{Q}(\mathbb{S},\mathbb{M})$ obtained by choosing a signing for each differential with a pole of order two. There is an analogous branched cover 
\[
\mathscr{P}^{\pm}(\mathbb{S},\mathbb{M})\rightarrow\mathscr{P}(\mathbb{S},\mathbb{M})
\]
obtained by choosing an eigenvalue for the monodromy around each pole of order $\leq2$. The group $\MCG^\pm(\mathbb{S},\mathbb{M})$ acts naturally on~$\mathscr{P}^\pm(\mathbb{S},\mathbb{M})$ by changing the marking and the choice of eigenvalues.

\begin{proposition}
\label{prop:signedspaces}
Assume that if $g(\mathbb{S})=0$ then $|\mathbb{M}|\geq3$. Then the sets $\mathscr{P}^{\pm}(\mathbb{S},\mathbb{M})$ and $\mathscr{Q}^{\pm}(\mathbb{S},\mathbb{M})$ are complex manifolds.
\end{proposition}

\begin{proof}
There is a holomorphic map 
\[
a:\mathscr{Q}(\mathbb{S},\mathbb{M})\rightarrow\mathbb{C}^{\mathbb{P}}
\]
sending a quadratic differential to the leading coefficient of its Laurent expansion at each of the punctures $p\in\mathbb{P}$. The branched cover $\mathscr{Q}^{\pm}(\mathbb{S},\mathbb{M})$ is defined by choosing a sign for the residue. It is smooth provided that $a$ is a submersion, which holds by the proof of~\cite{BridgelandSmith}, Lemma~6.1. The proof that $\mathscr{P}^{\pm}(\mathbb{S},\mathbb{M})$ is a complex manifold is completely analogous; the argument is identical to the proof of Proposition~8.4 in~\cite{AllegrettiBridgeland}.
\end{proof}

\subsection{An embedding of moduli spaces}
\label{sec:AnEmbeddingOfModuliSpaces}

Let $(\mathbb{S},\mathbb{M})$ be a marked bordered surface, and if $g(\mathbb{S})=0$ assume that $|\mathbb{M}|\geq3$. Let us choose a point $(S,\phi)\in\mathscr{Q}(\mathbb{S},\mathbb{M})$. Then we can define $\mathcal{P}^*$ to be the uniformizing projective structure for the punctured surface $S^*=S\setminus\Pol(\phi)$. The following result shows that $\mathcal{P}^*$ defines a meromorphic projective structure $\mathcal{P}$ on~$S$ whose polar differential has a particular form.

\begin{lemma}[\cite{Allegretti19}, Lemma~7.9]
\label{lem:Q2}
Let $\mathcal{P}^*$ and $S^*$ be as in the last paragraph. If $\chi(S^*)\leq0$ and $\mathcal{P}_0$ is any holomorphic projective structure on~$S$, then locally around any pole of $\phi$, the difference $\mathcal{P}^*-\mathcal{P}_0|_{S^*}$ can be written $Q_2(z)dz^{\otimes2}$ where 
\[
Q_2(z)=-\frac{1}{4z^2}+O(1) \quad\text{as $z\rightarrow0$}.
\]
\end{lemma}

We note that the hypothesis $\chi(S^*)\leq0$ is essential. If $\chi(S^*)>0$ then since $\Pol(\phi)\neq\emptyset$, the surface $S^*$ is the once punctured sphere, and the uniformizing projective structure $\mathcal{P}^*$ is the standard projective structure on the complex plane.  Letting $\mathcal{P}_0$ be the standard projective structure on $\mathbb{P}^1$, one finds $\mathcal{P}^*-\mathcal{P}_0|_{S^*}=0$.

Using Lemma~\ref{lem:Q2}, we define a meromorphic projective structure $\mathcal{P}$ on~$S$. We also consider the meromorphic projective structure given by 
\[
\mathcal{P}_\phi=\mathcal{P}+\phi.
\]
Since $\mathcal{P}$ has a pole of order two at each point of~$\Pol(\phi)$, this projective structure $\mathcal{P}_\phi$ has the same associated marked bordered surface as~$\mathcal{P}$. Since the pair $(S,\phi)$ is equipped with a choice of marking by $(\mathbb{S},\mathbb{M})$, the projective structure $\mathcal{P}_\phi$ is a point of the space $\mathscr{P}(\mathbb{S},\mathbb{M})$ of marked projective structures. 

\begin{lemma}
\label{lem:residueeigenvalues}
The monodromy of $\mathcal{P}_\phi$ around a pole $p$ of~$\phi$ of order two has eigenvalues 
\[
\lambda_\pm=-\exp(\pm\Res_p(\phi)/2).
\]
\end{lemma}

\begin{proof}
If we have $\chi(S^*)>0$, then $S$ is the Riemann sphere and $\phi$ has exactly one pole. By our assumption on the associated marked bordered surface, the order of this pole is necessarily~$>2$. We can therefore assume $\chi(S^*)\leq0$. Then locally around the pole $p$ a polar differential for $\mathcal{P}$ can be written $Q_2(z)dz^{\otimes2}$ where $Q_2(z)$ is a meromorphic function as in Lemma~\ref{lem:Q2}. The differential $\phi$ is given in this local coordinate by an expression $\phi(z)=Q_0(z)dz^{\otimes2}$ where 
\[
Q_0(z)=az^{-2}+O(z^{-1}) \quad\text{as $z\rightarrow0$}
\]
for some $a\neq0$. The meromorphic projective structure $\mathcal{P}_\phi$ thus has a polar differential given locally in a neighborhood of~$p$ by $Q(z)dz^{\otimes2}$ where 
\[
Q(z)=Q_0(z)+Q_2(z).
\]
Therefore, by Lemma~\ref{lem:eigenvalues}, the monodromy of $\mathcal{P}_\phi(t)$ around $p$ has eigenvalues $-\exp(\pm r/2)$ where 
\[
r=\pm2\pi i\cdot\sqrt{1+4\lim_{z\rightarrow0}z^2Q(z)}=\pm4\pi i\cdot\sqrt{a}.
\]
The leading coefficient of $Q_0(z)$ is invariant under changes of coordinates, and hence this last expression equals $\pm\Res_p(\phi)$ as desired.
\end{proof}

\begin{proposition}
\label{prop:openembedding}
There is an $\MCG(\mathbb{S},\mathbb{M})$-equivariant open embedding of moduli spaces 
\[
\mathscr{Q}(\mathbb{S},\mathbb{M})\hookrightarrow\mathscr{P}(\mathbb{S},\mathbb{M})
\]
with dense image defined by sending a differential $\phi$ to the projective structure $\mathcal{P}_\phi$. It lifts to an $\MCG^\pm(\mathbb{S},\mathbb{M})$-equivariant open embedding 
\[
\iota:\mathscr{Q}^{\pm}(\mathbb{S},\mathbb{M})\hookrightarrow\mathscr{P}^{\pm}(\mathbb{S},\mathbb{M})
\]
with dense image.
\end{proposition}

\begin{proof}
The map sending $\phi$ to $\mathcal{P}_\phi$ is continuous since the uniformizing projective structure $\mathcal{P}$ depends continuously (although not holomorphically) on the differential $\phi$ and the moduli of the Riemann surface on which $\phi$ is defined. The map would be a homeomorphism except for the fact that $\mathscr{Q}(\mathbb{S},\mathbb{M})$ parametrizes differentials with simple rather than arbitrary zeros. Hence the image is dense. If we are given a point in $\mathscr{Q}^{\pm}(\mathbb{S},\mathbb{M})$, then we have a choice of sign for the residue at each pole of order two. By Lemma~\ref{lem:residueeigenvalues}, this is equivalent to a choice of eigenvalue of the monodromy of~$\mathcal{P}_\phi$ around each pole of order two, and hence we get a point of $\mathscr{P}^{\pm}(\mathbb{S},\mathbb{M})$.
\end{proof}

\subsection{The monodromy map}

Let $(\mathbb{S},\mathbb{M})$ be a marked bordered surface, and if $g(\mathbb{S})=0$ assume that $|\mathbb{M}|\geq3$. Recall that $\mathscr{P}^*(\mathbb{S},\mathbb{M})$ is defined as the moduli space parametrizing projective structures without apparent singularities together with a choice of eigenline for the monodromy around each pole of order $\leq2$. We will write 
\[
\mathscr{Q}^*(\mathbb{S},\mathbb{M})\subset\mathscr{Q}^{\pm}(\mathbb{S},\mathbb{M})
\]
for the preimage of $\mathscr{P}^*(\mathbb{S},\mathbb{M})$ under the map $\iota$ of Proposition~\ref{prop:openembedding}. It is a dense open subset of the moduli space of signed differentials by Proposition~\ref{prop:openembedding}, and hence a complex manifold by Proposition~\ref{prop:signedspaces}. By applying Theorem~\ref{thm:monodromymap}, we immediately obtain the following statement.

\begin{proposition}
\label{prop:monodromydifferentials}
Let $(\mathbb{S},\mathbb{M})$ be a marked bordered surface, and if $g(\mathbb{S})=0$ assume that $|\mathbb{M}|\geq3$. Then there is an $\MCG^\pm(\mathbb{S},\mathbb{M})$-equivariant continuous map 
\[
\widehat{F}:\mathscr{Q}^*(\mathbb{S},\mathbb{M})\rightarrow\mathscr{X}^*(\mathbb{S},\mathbb{M}).
\]
\end{proposition}

The results of~\cite{GuptaMj} show that when $\mathbb{P}=\emptyset$, the monodromy map of Theorem~\ref{thm:monodromymap} is a local isomorphism. It follows that the map~$\widehat{F}$ is a local homeomorphism in this case. We conjecture that this remains true without the assumption $\mathbb{P}=\emptyset$. It would be interesting to know whether there exists a \emph{holomorphic} map $\mathscr{Q}^*(\mathbb{S},\mathbb{M})\rightarrow\mathscr{X}^*(\mathbb{S},\mathbb{M})$, equivariant with respect to the mapping class group action. The possible existence of such a map is discussed in more detail in Section~1.7.3 of~\cite{AllegrettiBridgeland}.

\section{The Riemann-Hilbert problem}
\label{sec:TheRiemannHilbertProblem}

In this section, we formulate the Riemann-Hilbert problem associated to a quadratic differential. The version of the problem that we consider here was first formulated in~\cite{Bridgeland19} and is the conformal limit of the Riemann-Hilbert problem considered by Gaiotto, Moore, and Neitzke in~\cite{GMN1}.

\subsection{BPS structures}

To formulate our Riemann-Hilbert problem, we employ the notion of a BPS structure introduced in~\cite{Bridgeland19}. This concept axiomatizes the output of Donaldson-Thomas theory applied to a 3-Calabi-Yau triangulated category equipped with a stability condition.

\begin{definition}[\cite{Bridgeland19}, Definition~2.1]
A \emph{BPS structure} consists of 
\begin{enumerate}[label=(\alph*)]
\item A lattice $\Gamma$ called the \emph{charge lattice} equipped with a skew-symmetric bilinear form 
\[
\langle -,- \rangle:\Gamma\times\Gamma\rightarrow\mathbb{Z}
\]
called the \emph{intersection form}.
\item A group homomorphism $Z:\Gamma\rightarrow\mathbb{C}$ called the \emph{central charge}.
\item A collection of rational numbers $\Omega(\gamma)$ ($\gamma\in\Gamma$) called \emph{BPS invariants}.
\end{enumerate}
These data are required to satisfy the following properties:
\begin{enumerate}
\item \emph{Symmetry:} $\Omega(-\gamma)=\Omega(\gamma)$ for all $\gamma\in\Gamma$.
\item \emph{Support property:} Fix a norm $\|\cdot\|$ on the finite-dimensional vector space $\Gamma\otimes_{\mathbb{Z}}\mathbb{R}$. Then there exists a constant $C>0$ such that if $\Omega(\gamma)\neq0$ then $|Z(\gamma)|>C\cdot \|\gamma\|$.
\end{enumerate}
\end{definition}

If $(\Gamma,Z,\Omega)$ is a BPS structure, then the \emph{Donaldson-Thomas (DT) invariant} for $\gamma\in\Gamma$ is defined by the formula 
\begin{equation}
\label{eqn:DTinvariant}
DT(\gamma)=\sum_{\gamma=m\alpha}\frac{1}{m^2}\Omega(\alpha)\in\mathbb{Q}
\end{equation}
where the sum is over all integers $m>0$ such that $\gamma$ is divisible by~$m$ in the lattice~$\Gamma$. The BPS and DT invariants are equivalent data since, by M\"obius inversion, we can write 
\[
\Omega(\gamma)=\sum_{\gamma=m\alpha}\frac{\mu(m)}{m^2}DT(\alpha)
\]
where $\mu(m)$ is the M\"obius function.

\subsection{The ray diagram}

Suppose we are given a BPS structure $(\Gamma,Z,\Omega)$. Then an element $\gamma\in\Gamma$ will be called \emph{active} if $\Omega(\gamma)\neq0$. It follows from the support property that in any bounded region of $\mathbb{C}$ there are only finitely many points of the form $Z(\gamma)$ for $\gamma\in\Gamma$ active. This property also implies that all such points are necessarily nonzero.

We will associate to the given BPS structure a certain diagram in the complex plane. By a \emph{ray} in~$\mathbb{C}^*$, we mean a subset of the form $\ell=\mathbb{R}_{>0}\cdot z$ for some $z\in\mathbb{C}^*$. A ray will be called \emph{active} if it contains a point $Z(\gamma)$ for some active $\gamma\in\Gamma$. The \emph{ray diagram} of the BPS~structure is the union of all active rays in~$\mathbb{C}^*$. An example is illustrated in Figure~\ref{fig:raydiagram}.

The \emph{height} of an active ray $\ell\subset\mathbb{C}^*$ is defined to be the number 
\[
H(\ell)=\inf\left\{|Z(\gamma)|:\gamma\in\Gamma\text{ such that }Z(\gamma)\in\ell\text{ and }\Omega(\gamma)\neq0\right\}.
\]
A non-active ray is considered to have infinite height. The support property guarantees that for any $H>0$, there are at most finitely many rays of height $<H$.

\subsection{The twisted torus}

Suppose we have a lattice $\Gamma\cong\mathbb{Z}^n$ equipped with a skew form $\langle -,- \rangle$ as in the definition of a BPS structure. Then there is an associated algebraic torus 
\[
\mathbb{T}_+=\Hom_{\mathbb{Z}}(\Gamma,\mathbb{C}^*)\cong(\mathbb{C}^*)^n.
\]
We will consider an associated torsor 
\[
\mathbb{T}_-=\left\{g:\Gamma\rightarrow\mathbb{C}^*:g(\gamma_1+\gamma_2)=(-1)^{\langle\gamma_1,\gamma_2\rangle}g(\gamma_1)g(\gamma_2)\right\}
\]
which we call the \emph{twisted torus}. The torus $\mathbb{T}_+$ acts naturally on the twisted torus~$\mathbb{T}_-$ by 
\[
(f\cdot g)(\gamma)=f(\gamma)g(\gamma)\in\mathbb{C}^*
\]
for $f\in\mathbb{T}_+$ and $g\in\mathbb{T}_-$, and this action is free and transitive. Thus, after choosing a basepoint in the twisted torus, we get an identification of $\mathbb{T}_-$ with $\mathbb{T}_+$. We can use this identification to give $\mathbb{T}_-$ the structure of an algebraic variety. This variety structure is independent of the choice of basepoint since the translation maps on~$\mathbb{T}_+$ are algebraic. The coordinate ring $\mathbb{C}[\mathbb{T}_-]$ of the twisted torus is spanned as a vector space by the functions 
\[
x_\gamma:\mathbb{T}_-\rightarrow\mathbb{C}^*, \quad x_\gamma(g)=g(\gamma)\in\mathbb{C}^*,
\]
which are called \emph{twisted characters}. The intersection form induces a natural Poisson bracket on $\mathbb{C}[\mathbb{T}_-]$ given on the twisted characters by 
\[
\{x_\alpha,x_\beta\}=\langle\alpha,\beta\rangle\cdot x_\alpha\cdot x_\beta.
\]
In the following, we will often denote $\mathbb{T}_-$ simply by $\mathbb{T}$.

\subsection{BPS automorphisms}

Given any ray $\ell\subset\mathbb{C}^*$, we can consider the associated formal generating series 
\begin{equation}
\label{eqn:generatingseries}
DT(\ell)=\sum_{Z(\gamma)\in\ell}DT(\gamma)\cdot x_\gamma
\end{equation}
for the Donaldson-Thomas invariants~\eqref{eqn:DTinvariant}. We would like to view this generating series as a well defined holomorphic function on the twisted torus $\mathbb{T}$. To do this, we will need to consider BPS~structures satisfying an additional property. Namely, we say that a BPS structure is \emph{convergent} if, for some $R>0$,
\[
\sum_{\gamma\in\Gamma}|\Omega(\gamma)|\cdot e^{-R|Z(\gamma)|}<\infty.
\]
For any acute sector $\Delta\subset\mathbb{C}^*$ and real number $R>0$, we consider the domain $U_\Delta(R)\subset\mathbb{T}$ defined as the interior of the set 
\[
\left\{g\in\mathbb{T}:Z(\gamma)\in\Delta\text{ and }\Omega(\gamma)\neq0\implies|g(\gamma)|<\exp(-R\|\gamma\|)\right\}\subset\mathbb{T} 
\]
which is nonempty by~\cite{Bridgeland19}, Lemma~B.2. Then we have the following.

\begin{proposition}[\cite{Bridgeland19}, Proposition~4.1]
\label{prop:BPSautomorphism}
Let $(\Gamma,Z,\Omega)$ be a convergent BPS structure, and let $\Delta\subset\mathbb{C}^*$ be a convex sector. Then for sufficiently large $R>0$, the following statements hold: 
\begin{enumerate}
\item For each ray $\ell\subset\Delta$, the power series~\eqref{eqn:generatingseries} is absolutely convergent on $U_\Delta(R)$ and thus defines a holomorphic function 
\[
DT(\ell):U_\Delta(R)\rightarrow\mathbb{C}.
\]
\item The time-1 Hamiltonian flow $\exp\{DT(\ell),-\}$ of the function $DT(\ell)$ defines a holomorphic embedding 
\[
\mathbf{S}(\ell):U_\Delta(R)\rightarrow\mathbb{T}.
\]
\item For every $H>0$, the composition
\[
\mathbf{S}_{<H}(\Delta)=\mathbf{S}(\ell_1)\circ\mathbf{S}(\ell_2)\circ\dots\circ\mathbf{S}(\ell_k)
\]
exists where $\ell_1,\ell_2,\dots,\ell_k\subset\Delta$ are the rays of height $<H$ in the sector $\Delta$ in the clockwise order, and the pointwise limit 
\[
\mathbf{S}(\Delta)=\lim_{H\rightarrow\infty}\mathbf{S}_{<H}(\Delta):U_\Delta(R)\rightarrow\mathbb{T}
\]
is a well defined holomorphic embedding.
\end{enumerate}
\end{proposition}

We think of the map $\mathbf{S}(\Delta)$ defined by Proposition~\ref{prop:BPSautomorphism} as a partially defined automorphism of the twisted torus and call it the \emph{BPS automorphism} associated to the sector~$\Delta$.

\subsection{BPS structures from quadratic differentials}

In this paper, we will be concerned with a particular class of BPS~structures arising from quadratic differentials. To any generic GMN differential $\phi$, we associate the triple $(\Gamma_\phi,Z_\phi,\Omega_\phi)$ where 
\begin{enumerate}[label=(\alph*)]
\item $\Gamma_\phi=\widehat{H}(\phi)$ is the hat-homology lattice of~$\phi$ and $\langle -,- \rangle$ is given by the intersection pairing on homology.
\item $Z_\phi(\gamma)=\int_\gamma\sqrt{\phi}$ is the period of~$\phi$.
\item $\Omega_\phi(\gamma)$ is the invariant counting finite-length trajectories of class $\gamma$.
\end{enumerate}
By Claim~7.1 of~\cite{Bridgeland19}, the triple $(\Gamma_\phi,Z_\phi,\Omega_\phi)$ defined in this way is a BPS structure. This BPS structure satisfies $|\Omega_\phi(\gamma)|\leq2$ for all $\gamma\in\Gamma_\phi$, and hence this BPS structure is easily seen to be convergent.

Assume now that the differential $\phi$ is complete and its associated marked bordered surface is amenable. If $\Delta\subset\mathbb{C}^*$ is a convex sector whose boundary rays are non-active with phases~$\theta_1$ and~$\theta_2$, then the rotated differentials $\phi_i=e^{-2i\theta_i}\cdot\phi$ are complete and saddle-free and therefore determine a pair of tagged WKB triangulations $\tau_i$. The associated Fock-Goncharov coordinates provide maps 
\[
X_{\tau_i}:\mathcal{X}(\mathbb{S},\mathbb{M})\dashrightarrow\Hom_{\mathbb{Z}}(\Gamma_i,\mathbb{C}^*)
\]
where $\Gamma_i\cong\mathbb{Z}^n$ is the lattice spanned by the set of tagged arcs of~$\tau_\pm$. By Lemma~10.3 in~\cite{BridgelandSmith}, the lattice $\Gamma_i$ is canonically isomorphic to the hat-homology $\widehat{H}(\phi_i)$ group. Note that we have a family of Riemann surfaces over $\mathbb{R}$, where the Riemann surface over $\theta\in\mathbb{R}$ is the spectral cover for the rotated differential $\phi_\theta=e^{-2i\theta}\cdot\phi$. It follows that the hat-homology groups $\widehat{H}(\phi_\theta)$ form a local system of lattices over $\mathbb{R}$ with flat connection given by the Gauss-Manin connection. Using this flat connection, we can identify the lattices $\Gamma_i$ with~$\Gamma_\phi=\widehat{H}(\phi)$. We can therefore think of the maps $X_{\tau_i}$ as taking values in the torus~$\mathbb{T}_+$.

\Needspace*{3\baselineskip}
\begin{proposition}[\cite{Allegretti20}]
\label{prop:basepoint}
Take notation as in the last paragraph. Then
\begin{enumerate}
\item There is a distinguished basepoint $\xi\in\mathbb{T}_-$ such that $\xi(\gamma)=-1$ if $\gamma\in\Gamma_\phi$ is the class of a non-closed saddle connection and $\xi(\gamma)=+1$ if $\gamma$ is the class of a closed saddle connection.
\item $\mathbf{S}(\Delta)$ extends to a birational automorphism of~$\mathbb{T}_-$. If we use the basepoint $\xi$ to identify $\mathbb{T}_-$ with $\mathbb{T}_+$, then this is precisely the birational transformation of~$\mathbb{T}_+$ relating the maps 
\[
X_{\tau_i}:\mathcal{X}(\mathbb{S},\mathbb{M})\dashrightarrow\mathbb{T}_+.
\]
\end{enumerate}
\end{proposition}

\subsection{Statement of the problem}

We now formulate the Riemann-Hilbert problem associated to a convergent BPS structure. In the following, we will consider, for any ray $r\subset\mathbb{C}^*$, the half plane 
\[
\mathbb{H}_r=\{t\in\mathbb{C}^*:t=u\cdot v, \ u\in r, \ \Re(v)>0\}\subset\mathbb{C}^*
\]
centered around~$r$. We will be interested in certain meromorphic functions 
\[
\mathcal{X}_r:\mathbb{H}_r\rightarrow\mathbb{T}
\]
which we can equivalently describe by specifying the compositions $\mathcal{X}_{r,\gamma}=x_\gamma\circ\mathcal{X}_r$ with the twisted characters $x_\gamma$ for $\gamma\in\Gamma$. The statement that $\mathcal{X}_r$ is meromorphic means that these compositions are meromorphic functions $\mathbb{H}_r\rightarrow\mathbb{C}^*$. We also fix a basepoint $\xi\in\mathbb{T}$ in the twisted torus.

\begin{problem}[\cite{Bridgeland19}]
Let $(\Gamma,Z,\Omega)$ be a convergent BPS structure. Then for each non-active ray $r\in\mathbb{C}^*$, we seek a meromorphic function $\mathcal{X}_r:\mathbb{H}_r\rightarrow\mathbb{T}$ satisfying the following conditions:
\begin{enumerate}
\item[(RH1)] Let $r_-$,~$r_+\subset\mathbb{C}^*$ be non-active rays which form the boundary rays of an acute sector $\Delta$ taken in the clockwise order. Then, for $t\in\mathbb{H}_{r_-}\cap\mathbb{H}_{r_+}$ with $0<|t|\ll1$, the functions $\mathcal{X}_{r_\pm}(t)$ are holomorphic and satisfy 
\[
\mathcal{X}_{r_-}(t)=\mathbf{S}(\Delta)(\mathcal{X}_{r_+}(t)).
\]
\item[(RH2)] For each non-active ray $r\subset\mathbb{C}^*$ and each class $\gamma\in\Gamma$, 
\[
\exp(Z(\gamma)/t)\cdot\mathcal{X}_{r,\gamma}(t)\rightarrow\xi(\gamma)
\]
as $t\rightarrow0$ in~$\mathbb{H}_r$.
\item[(RH3)] For any class $\gamma$ and any non-active ray $r\subset\mathbb{C}^*$, there exists $k>0$ such that 
\[
|t|^{-k}<|\mathcal{X}_{r,\gamma}(t)|<|t|^k
\]
for $t\in\mathbb{H}_r$ with $|t|\gg0$.
\end{enumerate}
\end{problem}

Our goal in the next section is to solve this Riemann-Hilbert problem in examples where the BPS structure arises from a generic GMN differential and $\xi$ is the canonical basepoint provided by Proposition~\ref{prop:basepoint}. We will also consider a modified version of this problem which is obtained by dropping the condition (RH3). We call this modified problem the \emph{weak Riemann-Hilbert problem}.

\section{Solving the Riemann-Hilbert problem}
\label{sec:SolvingTheRiemannHilbertProblem}

We will now use the cluster coordinates introduced previously to construct meromorphic functions which solve the Riemann-Hilbert problem.

\subsection{Constructing the solution}
\label{sec:ConstructingTheSolution}

Let $(\mathbb{S},\mathbb{M})$ be an amenable marked bordered surface. Suppose we are given a point $(S,\phi)\in\mathscr{Q}^\pm(\mathbb{S},\mathbb{M})$ where the differential $\phi$ is complete. For now, we will also assume that this differential is saddle-free.

As explained in the previous section, the differential $\phi$ determines an associated Riemann-Hilbert problem. Our solution of this Riemann-Hilbert problem will depend on a choice of base meromorphic projective structure $\mathcal{P}$ on~$S$ which we take to be the uniformizing projective structure provided by Lemma~\ref{lem:Q2}. Once we have chosen a meromorphic projective structure $\mathcal{P}$ in this way, we form the one-parameter family of meromorphic projective structures given by 
\[
\mathcal{P}_\phi(t)=\mathcal{P}+t^{-2}\cdot\phi
\]
for $t\in\mathbb{C}^*$ satisfying $\Re(t)>0$. There is a local system of sets over $\{t\in\mathbb{C}^*:\Re(t)>0\}$ whose fiber over~$t$ parametrizes the set of markings for the pair $(S,t^{-2}\cdot\phi)$ by~$(\mathbb{S},\mathbb{M})$. Since $(S,\phi)$ is equipped with such a marking, we can use flat connection of this local system to get a marking for every $(S,t^{-2}\cdot\phi)$. This determines a marking for $\mathcal{P}_\phi(t)$, and hence we can think of this projective structure as a point of the space $\mathscr{P}(\mathbb{S},\mathbb{M})$. 

Since $\phi$ is equipped with a choice of signing, we see by replacing $\phi$ by $t^{-2}\cdot\phi$ in Lemma~\ref{lem:residueeigenvalues} that there is a distinguished eigenvalue for the monodromy around any double pole of~$\mathcal{P}_\phi(t)$. Let $\mathcal{V}$ be the set of all $t$ such that $\mathcal{P}_\phi(t)$ has an apparent singularity. Then for $t\not\in\mathcal{V}$, the projective structure $\mathcal{P}_\phi(t)$ is naturally a point of $\mathscr{P}^*(\mathbb{S},\mathbb{M})$, and we can apply the monodromy map of Theorem~\ref{thm:monodromymap} to get an associated framed local system $F(\mathcal{P}_\phi(t))$.

Since $\phi$ is complete and saddle-free, there is an associated tagged WKB triangulation~$\tau(\phi)$. The Fock-Goncharov coordinates with respect to~$\tau(\phi)$ provide a birational map from the space of framed local systems to the torus $\mathbb{T}_+=\Hom_{\mathbb{Z}}(\Gamma_\phi,\mathbb{C}^*)$, and for any $\gamma\in\Gamma_\phi$, we denote by $X_{\tau(\phi),\gamma}$ the composition of this birational map with the character of $\mathbb{T}_+$ provided by~$\gamma$.

\begin{theorem}[\cite{Allegretti19}, Theorem~7.16]
\label{thm:meromorphic}
For each class $\gamma\in\Gamma_\phi$, the assignment 
\[
\mathcal{Y}_{\phi,\gamma}:t\mapsto X_{\tau(\phi),\gamma}(F(\mathcal{P}_\phi(t)))
\]
extends to a meromorphic function $\mathcal{Y}_{\phi,\gamma}$ on $\{t\in\mathbb{C}^*:\Re(t)>0\}$.
\end{theorem}

Now let us consider the same setup but without the assumption that $\phi$ is saddle-free. As we have seen, the differential $\phi$ determines a BPS structure and hence a ray diagram. If $r\subset\mathbb{C}^*$ is a non-active ray, then we can write $r=\mathbb{R}_{>0}\cdot e^{i\theta}$ for some phase~$\theta$, and the rotated differential $\phi_\theta=e^{-2i\theta}\cdot\phi$ will be saddle-free. For any $\gamma\in\Gamma_\phi$ and any point $t\in\mathbb{H}_r$ in the half plane centered around~$r$, we define 
\[
\mathcal{X}_{r,\gamma}(t)=\xi(\gamma)\cdot\mathcal{Y}_{\phi_\theta,\gamma}(e^{-i\theta}\cdot t)
\]
where $\xi$ is the map defined by Proposition~\ref{prop:basepoint}. By Theorem~\ref{thm:meromorphic}, this defines a meromorphic function $\mathcal{X}_r:\mathbb{H}_r\rightarrow\mathbb{T}$.

\begin{proposition}
Let $r_-$,~$r_+\subset\mathbb{C}^*$ be non-active rays which form the boundary rays of an acute sector $\Delta$ taken in the clockwise order. Then, for $t\in\mathbb{H}_{r_-}\cap\mathbb{H}_{r_+}$ with $0<|t|\ll1$, the functions $\mathcal{X}_{r_\pm}(t)$ are holomorphic and satisfy 
\[
\mathcal{X}_{r_-}(t)=\mathbf{S}(\Delta)(\mathcal{X}_{r_+}(t)).
\]
\end{proposition}

\begin{proof}
The holomorphicity property was proved in Theorem~1.3(1) of~\cite{Allegretti19}. Let us write $r_\pm=e^{i\theta_\pm}\cdot\mathbb{R}_{>0}$. Then the rotated differentials $\phi_\pm=e^{-2i\theta_\pm}\cdot\phi$ are saddle-free. Using the Gauss-Manin connection, we can identify the hat-homology groups $\widehat{H}(\phi_{\theta_\pm})$ with the lattice $\Gamma_\phi$. For each class $\gamma\in\Gamma_\phi$, we have 
\[
\mathcal{Y}_{\phi_\pm,\gamma}(e^{-i\theta}\cdot t)=X_{\tau(\phi_\pm),\gamma}(F(\mathcal{P}_\phi)),
\]
and therefore the lemma follows from Proposition~\ref{prop:basepoint}(2).
\end{proof}

\subsection{Asymptotic behavior at zero}

We now proceed to study the asymptotic behavior of the functions~$\mathcal{X}_{r,\gamma}(t)$ as $t$ tends to zero. In~\cite{Allegretti19}, we showed that the functions $\mathcal{Y}_{\phi,\gamma}(t)$ coincide with the Borel sums of Voros symbols in a domain of the form 
\[
\mathbb{H}(\varepsilon)=\{t\in\mathbb{C}:\text{$t<\varepsilon$ and $\Re(t)>0$}\}
\]
for some $\varepsilon>0$. We used this relationship and the known asymptotic properties of Voros symbols~\cite{IwakiNakanishi} to prove the following.

\begin{theorem}[\cite{Allegretti19}, Theorem~1.5]
\label{thm:asymptoticssaddlefree}
For each class $\gamma\in\Gamma_\phi$, the function $\mathcal{Y}_{\phi,\gamma}(t)$ satisfies 
\[
\exp(Z_\phi(\gamma)/t)\cdot\mathcal{Y}_{\phi,\gamma}(t)\rightarrow1
\]
as $t\rightarrow0$ in $\mathbb{H}(\varepsilon)$.
\end{theorem}

Using this fact, we can derive a similar asymptotic property of the function $\mathcal{X}_{r,\gamma}(t)$.

\begin{proposition}
\label{prop:RH2}
For each non-active ray $r\subset\mathbb{C}^*$ and each class $\gamma\in\Gamma_\phi$, we have 
\[
\exp(Z_\phi(\gamma)/t)\cdot\mathcal{X}_{r,\gamma}(t)\rightarrow\xi(\gamma)
\]
as $t\rightarrow0$ in~$\mathbb{H}_r$.
\end{proposition}

\begin{proof}
Let us write $r=\mathbb{R}_{>0}\cdot e^{i\theta}$ for some phase $\theta\geq0$. Using the Gauss-Manin connection, we can identify $\widehat{H}(\phi_\theta)$ with the lattice~$\Gamma_\phi=\widehat{H}(\phi)$. Then for any $\gamma\in\Gamma_\phi$, we have 
\[
Z_{\phi_\theta}(\gamma)=e^{-i\theta}\cdot Z_\phi(\gamma)
\]
and hence $Z_{\phi_\theta}(\gamma)/\tilde{t}=Z_\phi(\gamma)/t$ where $\tilde{t}=e^{-i\theta}\cdot t$. By Theorem~\ref{thm:asymptoticssaddlefree}, we have 
\[
\exp(Z_\phi(\gamma)/t)\cdot\mathcal{X}_{r,\gamma}(t)=\xi(\gamma)\cdot\left(\exp(Z_{\phi_\theta}(\gamma)/\tilde{t})\cdot\mathcal{Y}_{\phi_\theta,\gamma}(\tilde{t})\right)\rightarrow\xi(\gamma)
\]
as $\tilde{t}\rightarrow0$ in~$\mathbb{H}(\varepsilon)$, or equivalently as $t\rightarrow0$ in~$\mathbb{H}_r$.
\end{proof}

\subsection{Asymptotic behavior at infinity}

Finally we will study the growth of the functions $\mathcal{X}_{r,\gamma}(t)$ as $t$ tends to infinity. Unfortunately our methods cannot be applied to an arbitrary choice of the quadratic differential $\phi$, and so in this subsection we will restrict attention to quadratic differentials having only poles of order two.

To study the behavior of the function $\mathcal{X}_{r,\gamma}(t)$ at infinity, we will change variables to $\eta=1/t$. Suppose $p$ is a pole of order two of the quadratic differential~$\phi$, and let $z$ be a local coordinate defined in a neighborhood of~$p$ so that $z(p)=0$. In this neighborhood, a polar differential for the projective structure~$\mathcal{P}_\phi(t)$ can be written in the form $Q(z,\eta)dz^{\otimes2}$. If $U$ is a disk centered at~0 in the $\eta$-plane, then there is a rank-2 vector bundle $E\rightarrow U$ where the fiber over $\eta\in U$ is the space of germs of solutions of the differential equation 
\begin{equation}
\label{eqn:etaschrodinger}
y''(z)-Q(z,\eta)\cdot y(z)=0
\end{equation}
at some fixed basepoint in the $z$-plane. For a generic $\eta\neq0$, the choice of signing for the differential~$\phi$ picks out a distinguished eigenline $\ell(\eta)$ for the monodromy of~\eqref{eqn:etaschrodinger} around $z=0$. Thus we have a section $\ell$ of the projective bundle $\mathbb{P}(E)$ over a dense open set~$V\subset U\setminus\{0\}$. By the remarks following Lemma~5.1 of~\cite{AllegrettiBridgeland}, we can find a point $\eta=\eta_0\in U$ where the monodromy of~\eqref{eqn:etaschrodinger} is diagonalizable with distinct eigenvalues. The fiber of $E$ over~$\eta_0$ has a basis determined by the eigenvectors, and we can choose a trivialization of $E$ which maps this basis to the standard basis for $\mathbb{C}^2$. This induces a trivialization of $\mathbb{P}(E)$, and hence we can regard $\ell$ as a function $V\rightarrow\mathbb{P}^1$.

\begin{lemma}
\label{lem:extendframing}
This $\ell$ extends to a holomorphic function $\ell:U\rightarrow\mathbb{P}^1$.
\end{lemma}

\begin{proof}
Replacing $\phi$ by $\eta^2\cdot\phi$ in Lemma~\ref{lem:residueeigenvalues}, we see that for $\eta\neq0$ the eigenvalues of the monodromy of~\eqref{eqn:etaschrodinger} around $z=0$ are given by $\lambda_\pm(\eta)=-\exp(\pm\Res_p(\phi)\eta/2)$. This expression clearly defines a holomorphic function on~$U$. By the results of~\cite{AllegrettiBridgeland}, Section~8, the monodromy $M(\eta)$ of~\eqref{eqn:etaschrodinger} around $z=0$ also depends holomorphically on~$\eta\in U$. Therefore so does the matrix 
\[
M(\eta)-\lambda_\pm(\eta)=
\left(\begin{array}{cc}
a_\pm(\eta) & b_\pm(\eta) \\
c_\pm(\eta) & d_\pm(\eta)
\end{array}\right)
\]
where we are working in the basis provided by the chosen trivialization. We claim that the diagonal entries of this matrix are not both identically zero. Indeed, for~$\eta=\eta_0$ the monodromy matrix $M(\eta)$ is diagonal with distinct eigenvalues. It follows that after subtracting the scalar matrix $\lambda_\pm(\eta)$ we get a matrix which again has distinct diagonal entries. This proves the claim. Now for $\eta\in V$, the choice of signing picks out a distinguished eigenvalue, say $\lambda_+(\eta)$, which has a 1-dimensional eigenspace. A vector $v=(v_1(\eta), v_2(\eta))^t$ in this eigenspace satisfies the conditions 
\[
a_+(\eta)v_1(\eta)+b_+(\eta)v_2(\eta)=0, \quad c_+(\eta)v_1(\eta)+d_+(\eta)v_2(\eta)=0.
\]
Thus there is an affine chart on~$\mathbb{P}^1$ in which the function $\ell$ is given by one of the formulas $\ell(\eta)=-b_+(\eta)/a_+(\eta)$ or $\ell(\eta)=-c_+(\eta)/d_+(\eta)$, depending on which of the denominators is not identically zero. Hence $\ell$ is given in this chart by a meromorphic function $U\rightarrow\mathbb{C}$. Equivalently, it extends to a holomorphic function $U\rightarrow\mathbb{P}^1$.
\end{proof}

Using this lemma, we can study the growth of the functions $\mathcal{Y}_{\phi,\gamma}(t)$ as $t$ tends to infinity.

\begin{proposition}
\label{prop:asymptoticinfinityY}
Assume the quadratic differential $\phi$ is saddle-free and has only poles of order two. Then for any class $\gamma$, there exists $k>0$ such that 
\[
|t|^{-k}<|\mathcal{Y}_{\phi,\gamma}(t)|<|t|^k
\]
for $|t|\gg0$.
\end{proposition}

\begin{proof}
As before, we will write $\eta=1/t$. Then for a generic $\eta\neq0$, there is a framed local system $F(\mathcal{P}_\phi(1/\eta))$ which associates a framing line $\ell_p(\eta)$ to each pole $p$ of the differential~$\phi$. By Lemma~\ref{lem:extendframing}, the line $\ell_p(\eta)$ depends holomorphically on the parameter~$\eta$, and in fact $\ell_p$ extends to a holomorphic function $U\rightarrow\mathbb{P}^1$ where $U$ is a disk centered at~0 in $\eta$-space. Let $\tilde{\ell}_p(\eta)$ be the line obtained by parallel transporting $\ell_p(\eta)$ to a fixed basepoint on the surface. Then $\tilde{\ell}_p$ is again a holomorphic function $U\rightarrow\mathbb{P}^1$. For any class $\gamma$, the function $\mathcal{Y}_{\phi,\gamma}(1/\eta)$ is a product of cross ratios of the lines $\tilde{\ell}_p(\eta)$, hence a meromorphic function of~$\eta\in U$. It follows that $\mathcal{Y}_{\phi,\gamma}(1/\eta)$ satisfies the required bounds for $\eta$ small.
\end{proof}

From this, we obtain the desired property of the functions $\mathcal{X}_{r,\gamma}(t)$.

\begin{proposition}
\label{prop:RH3}
Assume the quadratic differential $\phi$ has only poles of order two. Then for any class $\gamma$ and any non-active ray $r\subset\mathbb{C}^*$, there exists $k>0$ such that 
\[
|t|^{-k}<|\mathcal{X}_{r,\gamma}(t)|<|t|^k
\]
for $t\in\mathbb{H}_r$ with $|t|\gg0$.
\end{proposition}

\begin{proof}
If $r=\mathbb{R}_{>0}\cdot e^{i\theta}$, then by Proposition~\ref{prop:asymptoticinfinityY}, there exists $k>0$ such that 
\[
|t|^{-k}=|e^{-i\theta}\cdot t|^{-k}<|\xi(\gamma)\cdot\mathcal{Y}_{\phi_\theta,\gamma}(e^{-i\theta}\cdot t)|<|e^{-i\theta}\cdot t|^k=|t|^k
\]
for $|t|=|e^{-i\theta}\cdot t|\gg0$.
\end{proof}

\section{Stability conditions and the cluster variety}
\label{sec:StabilityConditionsAndTheClusterVariety}

We will now define the space of stability conditions and the cluster variety in the abstract setting of triangulated categories. In this section, $\mathcal{D}$ will always denote a $\Bbbk$-linear triangulated category with shift functor~$[1]$.

\subsection{Hearts and tilting}

The notion of a t-structure on a triangulated category $\mathcal{D}$ is a tool that allows one to see the different abelian subcategories of~$\mathcal{D}$. Associated to a t-structure is a full subcategory $\mathcal{A}\subset\mathcal{D}$ called the heart of the t-structure. In this paper, we will deal exclusively with t-structures satisfying an additional boundedness condition. Such a t-structure is determined by its heart, which can be characterized as follows.

\begin{definition}[\cite{Bridgeland07}, Lemma~3.2]
Let $\mathcal{D}$ be a triangulated category. Then the \emph{heart} of a bounded t-structure on~$\mathcal{D}$ is a full additive subcategory $\mathcal{A}\subset\mathcal{D}$ such that 
\begin{enumerate}
\item If $j>k$ are integers, then $\Hom_{\mathcal{D}}(A[j],B[k])=0$ for all objects $A$,~$B\in\mathcal{A}$.
\item For every object $E\in\mathcal{D}$, there is a finite sequence of integers 
\[
k_1>k_2>\dots>k_s
\]
and a collection of triangles 
\[
\xymatrix{
0=E_0 \ar[rr] & & E_1 \ar[ld] \ar[rr] & & E_2\ar[ld] \ar[r] & \cdots \ar[r] & E_{s-1} \ar[rr] & & \ar[ld] E_s=E \\
& A_1 \ar@{-->}[lu] & & A_2 \ar@{-->}[lu] & & & & A_s \ar@{-->}[lu]
}
\]
with $A_j\in\mathcal{A}[k_j]$ for all~$j$.
\end{enumerate}
\end{definition}

It is known that any heart in a triangulated category is a full abelian subcategory. A heart is said to be of \emph{finite length} if it is noetherian and artinian as an abelian category. Given full subcategories $\mathcal{A}$,~$\mathcal{B}\subset\mathcal{D}$, the \emph{extension closure} $\mathcal{C}=\langle\mathcal{A},\mathcal{B}\rangle\subset\mathcal{D}$ is defined as the smallest full subcategory of $\mathcal{D}$ containing both~$\mathcal{A}$ and~$\mathcal{B}$ such that if $X\rightarrow Y\rightarrow Z\rightarrow X[1]$ is a triangle in~$\mathcal{D}$ with $X$,~$Z\in\mathcal{C}$, then $Y\in\mathcal{C}$.

A pair $(\mathcal{A}_1,\mathcal{A}_2)$ of hearts in a triangulated category~$\mathcal{D}$ is called a \emph{tilting pair} if either of the equivalent conditions 
\[
\mathcal{A}_2\subset\langle\mathcal{A}_1,\mathcal{A}_1[-1]\rangle, \quad \mathcal{A}_1\subset\langle\mathcal{A}_2[1],\mathcal{A}_2\rangle
\]
is satisfied. In this case, we say that $\mathcal{A}_1$ is a \emph{left tilt} of~$\mathcal{A}_2$ and that $\mathcal{A}_2$ is a \emph{right tilt} of~$\mathcal{A}_1$. If $(\mathcal{A}_1,\mathcal{A}_2)$ is a tilting pair, then the full subcategories $\mathcal{T}=\mathcal{A}_1\cap\mathcal{A}_2[1]$ and $\mathcal{F}=\mathcal{A}_1\cap\mathcal{A}_2$ form a torsion pair. Conversely, if $(\mathcal{T},\mathcal{F})\subset\mathcal{A}$ is a torsion pair, then the category $\mathcal{A}_2=\langle\mathcal{F},\mathcal{T}[-1]\rangle$ is a heart, and the pair $(\mathcal{A}_1,\mathcal{A}_2)$ is a tilting pair. 

We will be interested in a special case of the tilting construction. Suppose that $\mathcal{A}$ is a finite length heart and $S\in\mathcal{A}$ is a simple object. Then we define full subcategories 
\[
S^\perp=\{E\in\mathcal{A}:\Hom_{\mathcal{A}}(S,E)=0\}, \quad {^\perp}S=\{E\in\mathcal{A}:\Hom_{\mathcal{A}}(E,S)=0\}.
\]
If $\langle S\rangle\subset\mathcal{A}$ denotes the full subcategory of $\mathcal{A}$ consisting of objects $E\in\mathcal{A}$ all of whose simple factors are isomorphic to~$S$, then the pairs $(\langle S\rangle,S^\perp)$ and $({^\perp}S,\langle S\rangle)$ are torsion pairs. The corresponding tilts
\[
\mu_S^-(\mathcal{A})=\langle S[1],{^\perp}S\rangle, \quad \mu_S^+(\mathcal{A})=\langle S^\perp,S[-1]\rangle
\]
are called the left tilt and right tilt of~$\mathcal{A}$ at~$S$, respectively.

\subsection{Quivers with potential}

A finite length heart $\mathcal{A}$ in a triangulated category naturally determines a quiver $Q(\mathcal{A})$ whose vertices are in bijection with isomorphism classes of simple objects in~$\mathcal{A}$. If $i$ and $j$ are vertices of~$Q(\mathcal{A})$ corresponding to simple objects $S_i$ and~$S_j$, respectively, then the number of arrows from $i$ to~$j$ in this quiver is given by $\dim_\Bbbk\Ext_{\mathcal{A}}^1(S_i,S_j)$. The quiver $Q(\mathcal{A})$ defined in this way is known as the \emph{Ext quiver} of~$\mathcal{A}$.

A triangulated category $\mathcal{D}$ is said to be of \emph{finite type} if for all objects $E$,~$F\in\mathcal{D}$, one has 
\[
\dim_\Bbbk\bigoplus_{i\in\mathbb{Z}}\Hom_{\mathcal{D}}^i(E,F)<\infty
\]
where $\Hom_\mathcal{D}^i(E,F)=\Hom_{\mathcal{D}}(E,F[i])$. For any triangulated category~$\mathcal{D}$ of finite type, there is a bilinear form $\langle -,-\rangle:K(\mathcal{D})\times K(\mathcal{D})\rightarrow\mathbb{Z}$ known as the \emph{Euler form} and given by 
\[
\langle E,F\rangle=\sum_{i\in\mathbb{Z}}(-1)^i\dim_\Bbbk\Hom_{\mathcal{D}}^i(E,F).
\]
A triangulated category $\mathcal{D}$ of finite type is called \emph{3-Calabi-Yau (CY$_3$)} if there are functorial isomorphisms 
\[
\Hom_{\mathcal{D}}^i(E,F)\cong\Hom_{\mathcal{D}}^{3-i}(F,E)^*.
\]
If $\mathcal{D}$ is a finite type triangulated category with this property, then the Euler form is skew-symmetric.

If $\mathcal{A}\subset\mathcal{D}$ is a finite length heart in a CY$_3$ triangulated category and $S_i$,~$S_j\in\mathcal{A}$ are simple objects, then we have $\Hom_{\mathcal{D}}^{<0}(S_i,S_j)=0$, $\Hom_{\mathcal{D}}^0(S_i,S_j)\cong\Bbbk^{\delta_{ij}}$, and $\Hom_{\mathcal{D}}^1(S_i,S_j)=\Ext_{\mathcal{A}}^1(S_i,S_j)$. From these facts and the CY$_3$ property, we deduce that 
\[
\langle S_i,S_j\rangle=|\{\text{arrows $j\rightarrow i$ in $Q(\mathcal{A})$}\}|-|\{\text{arrows $i\rightarrow j$ in $Q(\mathcal{A})$}\}|
\]
for all simple objects $S_i$,~$S_j\in\mathcal{A}$.

The next result shows that after choosing a potential, we can reverse the process described above and construct a category with a canonical heart having a given Ext quiver.

\begin{theorem}[\cite{BridgelandSmith}, Theorem~7.2]
Let $(Q,W)$ be a 2-acyclic quiver with potential. Then there is an associated CY$_3$ triangulated category $\mathcal{D}(Q,W)$ over~$\Bbbk$. It has a canonical finite length heart 
\[
\mathcal{A}(Q,W)\subset\mathcal{D}(Q,W)
\]
whose Ext quiver is isomorphic to~$Q$.
\end{theorem}

Explicitly, the category $\mathcal{D}(Q,W)$ is defined as the subcategory of the derived category of modules over the complete Ginzburg algebra of $(Q,W)$ consisting of modules with finite-dimensional total cohomology. Up to equivalence, this category depends only on the right equivalence class of $(Q,W)$. Moreover, we have the following result of Keller and Yang.

\begin{theorem}[\cite{KellerYang}, Theorem~3.2, Corollary~5.5]
\label{thm:mutationtilt}
Let $(Q,W)$ be a 2-acyclic quiver with potential, and let $k$ be a vertex of~$Q$. Suppose $(Q',W')=\mu_k(Q,W)$ is a quiver with potential obtained by mutation in the direction~$k$. Then there is a canonical pair of $\Bbbk$-linear triangulated equivalences 
\[
\Phi_\pm:\mathcal{D}(Q',W')\rightarrow\mathcal{D}(Q,W)
\]
such that if $S_k\in\mathcal{A}(Q,W)$ denotes the simple object corresponding to the vertex $k$, then the functors $\Phi_\pm$ induce tilts at~$S_k$ in the sense that 
\[
\Phi_\pm(\mathcal{A}(Q',W'))=\mu_{S_k}^\pm(\mathcal{A}(Q,W)).
\]
\end{theorem}

\subsection{The tilting graph}

The \emph{tilting graph} of a triangulated category $\mathcal{D}$ is the graph $\Tilt(\mathcal{D})$ whose vertices are the finite length hearts of $\mathcal{D}$, where two vertices are connected by an edge if the corresponding hearts are related by a tilt at a simple object. If $\mathcal{A}\subset\mathcal{D}$ is a finite length heart, then we will write $\Tilt_{\mathcal{A}}(\mathcal{D})$ for the connected component of $\Tilt(\mathcal{D})$ containing the vertex~$\mathcal{A}$. If $\mathcal{B}\subset\mathcal{D}$ is a finite length heart which is a vertex of $\Tilt_{\mathcal{A}}(\mathcal{D})$, then we say that $\mathcal{B}$ is \emph{reachable} from~$\mathcal{A}$.

Let us specialize to the case where the category $\mathcal{D}=\mathcal{D}(Q,W)$ is the CY$_3$ triangulated category associated to a nondegenerate quiver with potential. In this case, there is a canonical finite length heart $\mathcal{A}(Q,W)\subset\mathcal{D}(Q,W)$, and an arbitrary finite length heart in~$\mathcal{D}$ will be called \emph{reachable} if it is reachable from this canonical heart. We will write $\Tilt_\Delta(\mathcal{D})$ for the connected component of $\Tilt(\mathcal{D})$ whose vertices are the reachable hearts.

Let us denote by $\Aut(\mathcal{D})$ the group of all triangulated autoequivalences of~$\mathcal{D}$. Then there is a natural action of $\Aut(\mathcal{D})$ on the tilting graph $\Tilt(\mathcal{D})$. An autoequivalence will be called \emph{reachable} if its action on $\Tilt(\mathcal{D})$ preserves the component $\Tilt_\Delta(\mathcal{D})$. The reachable autoequivalences form a subgroup denoted 
\[
\Aut_\Delta(\mathcal{D})\subset\Aut(\mathcal{D}).
\]
A reachable autoequivalece is called \emph{negligible} if it acts trivially on $\Tilt_\Delta(\mathcal{D})$. The negligible autoequivalences form a subgroup 
\[
\Nil_\Delta(\mathcal{D})\subset\Aut_\Delta(\mathcal{D}),
\]
and we will be interested in the quotient 
\[
\cAut_\Delta(\mathcal{D})=\Aut_\Delta(\mathcal{D})/\Nil_\Delta(\mathcal{D})
\]
which acts effectively on $\Tilt_\Delta(\mathcal{D})$.

By work of Seidel and Thomas~\cite{SeidelThomas}, one can associate to each reachable heart $\mathcal{A}$ and simple object $S\in\mathcal{A}$ an autoequivalence $\Tw_S$ of $\mathcal{D}=\mathcal{D}(Q,W)$ known as a \emph{spherical twist}. These functors generate a subgroup of the group of all triangulated autoequivalences denoted 
\[
\Sph_{\mathcal{A}}(\mathcal{D})=\langle\Tw_S:S\in\mathcal{A}\text{ simple}\rangle\subset\Aut(\mathcal{D}).
\]
The following proposition summarizes the main facts we will need about the spherical twists.

\begin{proposition}[\cite{BridgelandSmith}, Proposition~7.1] Let $\mathcal{A}\subset\mathcal{D}$ be as above. Then 
\begin{enumerate}
\item For every simple object $S\in\mathcal{A}$, one has the identity 
\[
\Tw_S(\mu_S^-(\mathcal{A}))=\mu_S^+(\mathcal{A}).
\]
\item If $\mathcal{B}$ is reachable from $\mathcal{A}$, then $\Sph_{\mathcal{B}}(\mathcal{D})=\Sph_{\mathcal{A}}(\mathcal{D})$.
\end{enumerate}
\end{proposition}

By property~(2), we can identify the group $\Sph_{\mathcal{A}}(\mathcal{D})$ for any reachable heart $\mathcal{A}$ with the group $\Sph_\Delta(\mathcal{D})\subset\Aut_\Delta(\mathcal{D})$ of spherical twists associated to simple obects of the canonical heart $\mathcal{A}(Q,W)$. We will write 
\[
\cSph_\Delta(\mathcal{D})\subset\cAut_\Delta(\mathcal{D})
\]
for the image of $\Sph_\Delta(\mathcal{D})$ in $\cAut_\Delta(\mathcal{D})=\Aut_\Delta(\mathcal{D})/\Nil_\Delta(\mathcal{D})$. This group acts on $\Tilt_\Delta(\mathcal{D})$, and the quotient 
\[
\Exch_\Delta(\mathcal{D})=\Tilt_\Delta(\mathcal{D})/\cSph(\mathcal{D})
\]
is known as the \emph{heart exchange graph} of~$\mathcal{D}$. The quotient group 
\[
\mathcal{G}_\Delta(\mathcal{D})=\cAut_\Delta(\mathcal{D})/\cSph_\Delta(\mathcal{D})
\]
acts effectively on $\Exch_\Delta(\mathcal{D})$ and is known as the \emph{cluster modular group}.

\subsection{Stability conditions}

We can now recall Bridgeland's notion of a stability condition on a triangulated category, following the treatment in~\cite{Bridgeland07}. We begin by defining a \emph{stability function} on an abelian category $\mathcal{A}$ as a group homomorphism $Z:K(\mathcal{A})\rightarrow\mathbb{C}$ such that for all nonzero objects $E\in\mathcal{A}$, the complex number $Z(E)$ lies in the semi-closed upper half plane 
\begin{equation}
\label{eqn:upperhalfplane}
\mathcal{H}=\{r\exp(i\pi\phi):r>0\text{ and }0<\phi\leq1\}\subset\mathbb{C}.
\end{equation}
Given a stability function $Z:K(\mathcal{A})\rightarrow\mathbb{C}$, the \emph{phase} of a nonzero object $E\in\mathcal{A}$ is defined as 
\[
\phi(E)=(1/\pi)\arg Z(E)\in(0,1].
\]
A nonzero object $E\in\mathcal{A}$ is \emph{semistable} with respect to~$Z$ if every nonzero proper subobject $A\subset E$ satisfies $\phi(A)\leq\phi(E)$.

For a given stability function on an abelian category~$\mathcal{A}$, the semistable objects provide a way to filter arbitrary objects of~$\mathcal{A}$. A \emph{Harder-Narasimhan filtration} of a nonzero $E\in\mathcal{A}$ is a finite sequence of subobjects 
\[
0=E_0\subset E_1\subset\dots\subset E_{n-1}\subset E_n=E
\]
such that each factor $F_j=E_j/E_{j-1}$ is semistable, and 
\[
\phi(F_1)>\phi(F_2)>\dots>\phi(F_n).
\]
The stability function $Z$ is said to have the \emph{Harder-Narasimhan property} if every nonzero object of~$\mathcal{A}$ has a Harder-Narasimhan filtration. If $\mathcal{A}$ is a finite-length heart, then the Harder-Narasimhan property is automatically satisfied for any stability function on~$\mathcal{A}$.

The definition of a stability condition can now be given as follows.

\begin{definition}[\cite{Bridgeland07}, Proposition~5.3]
A \emph{stability condition} $(\mathcal{A},Z)$ on a triangulated category~$\mathcal{D}$ consists of the heart $\mathcal{A}$ of a bounded t-structure on~$\mathcal{D}$ together with a stability function $Z$ on~$\mathcal{A}$ having the Harder-Narasimhan property.
\end{definition}

If $\mathcal{A}\subset\mathcal{D}$ is the heart of a bounded t-structure on~$\mathcal{D}$, then there is an isomorphism $K(\mathcal{A})\cong K(\mathcal{D})$, and thus we can view a stability function $Z$ on~$\mathcal{A}$ as a homomorphism $Z:K(\mathcal{D})\rightarrow\mathbb{C}$ called the \emph{central charge}. We will always assume that the Grothendieck group of our category is a lattice $K(\mathcal{D})\cong\mathbb{Z}^{\oplus n}$ of finite rank, and we will restrict attention to stability conditions satisfying the \emph{support property} of~\cite{KontsevichSoibelman}: For some norm $\|\cdot\|$ on $K(\mathcal{D})\otimes_{\mathbb{Z}}\mathbb{R}$, there is a constant $C>0$ such that 
\[
\|\gamma\|<C\cdot|Z(\gamma)|
\]
for each class $\gamma\in K(\mathcal{D})$ represented by a semistable object. We write $\Stab(\mathcal{D})$ for the set of all stability conditions on~$\mathcal{D}$ satisfying this support property. The most important property of $\Stab(\mathcal{D})$ is that this set has a natural complex manifold structure.

\begin{theorem}[\cite{Bridgeland07}, Theorem~1.2]
\label{thm:forgetfulmap}
The set $\Stab(\mathcal{D})$ has the structure of a complex manifold such that the map 
\[
\Stab(\mathcal{D})\rightarrow\Hom_\mathbb{Z}(K(\mathcal{D}),\mathbb{C})
\]
taking a stability condition to its central charge is a local isomorphism.
\end{theorem}

It follows immediately from the way we have defined things that every stability condition has an associated heart. For an abelian subcategory $\mathcal{A}\subset\mathcal{D}$, let us write $\Stab(\mathcal{A})\subset\Stab(\mathcal{D})$ for the subset of stability conditions with heart~$\mathcal{A}$. If $\mathcal{A}$ is of finite length with finitely many simple objects $S_1,\dots,S_n\in\mathcal{A}$ up to isomorphism, then by the remarks above, we obtain an isomorphism 
\[
\Stab(\mathcal{A})\cong\mathcal{H}^n,
\]
where $\mathcal{H}$ is the semi-closed upper half plane defined in~\eqref{eqn:upperhalfplane}, by sending a stability condition $(\mathcal{A},Z)$ to the point $(Z(S_1),\dots,Z(S_n))$. It follows that if we let $\Stab_{\mathrm{tame}}(\mathcal{D})\subset\Stab(\mathcal{D})$ be the subset consisting of stability conditions whose heart is of finite length, then as sets we have 
\[
\Stab_{\mathrm{tame}}(\mathcal{D})=\coprod_{\mathcal{A}\in\Tilt(\mathcal{D})}\Stab(\mathcal{A}).
\]
In fact, the following result shows that this provides a cell decomposition of $\Stab_{\mathrm{tame}}(\mathcal{D})$ with dual graph given by the tilting graph $\Tilt(\mathcal{D})$.

\begin{proposition}[\cite{BridgelandSmith}, Lemma~7.9]
\label{prop:wallandchamber}
Let $\mathcal{A}_1$,~$\mathcal{A}_2\subset\mathcal{D}$ be finite length hearts. Then the closures of the sets $\Stab(\mathcal{A}_i)\subset\Stab(\mathcal{D})$ intersect if and only if the $\mathcal{A}_i$ are related by a tilt at a simple object. In this case, the intersection has real codimension one in $\Stab(\mathcal{D})$.
\end{proposition}

It follows from Proposition~\ref{prop:wallandchamber} that if $\mathcal{D}=\mathcal{D}(Q,W)$ for nondegenerate $(Q,W)$, then there is a distinguished connected component $\Stab_\Delta(\mathcal{D})\subset\Stab(\mathcal{D})$ containing stability conditions whose associated hearts are vertices in the distinguished component $\Tilt_\Delta(\mathcal{D})\subset\Tilt(\mathcal{D})$.

\subsection{Group actions}

Typically, one is interested quotients of the space of stability conditions by various group actions. The group $\Aut(\mathcal{D})$ of autoequivalences of a triangulated category~$\mathcal{D}$ acts naturally on the manifold $\Stab(\mathcal{D})$. Indeed, if we are given a point $(\mathcal{A},Z)\in\Stab(\mathcal{D})$ and an autoequivalence $\Phi\in\Aut(\mathcal{D})$, we obtain a stability condition $\Phi\cdot(\mathcal{A},Z)=(\mathcal{A}',Z')$ where 
\[
\mathcal{A}'=\Phi(\mathcal{A}), \quad Z'(E)=Z(\Phi^{-1}(E)).
\]
In the case where $\mathcal{D}=\mathcal{D}(Q,W)$ is the CY$_3$ triangulated category associated to a quiver with potential, there is an induced action of $\cSph_\Delta(\mathcal{D})$ on the component $\Stab_\Delta(\mathcal{D})$. We will be interested in the quotient 
\[
\Sigma(Q,W)=\Stab_\Delta(\mathcal{D})/\cSph_\Delta(\mathcal{D})
\]
which has a natural action of $\mathcal{G}_\Delta(\mathcal{D})$.

As explained in~\cite{BridgelandSmith}, Section~7.5, there is also a natural action of the group of complex numbers on the space~$\Stab(\mathcal{D})$. If we let $z\in\mathbb{C}$ act on the space of central charges by mapping $Z\in\Hom_{\mathbb{Z}}(K(\mathcal{D}),\mathbb{C})$ to $e^{-i\pi z}\cdot Z$, then the forgetful map in Theorem~\ref{thm:forgetfulmap} is $\mathbb{C}$-equivariant. The $\mathbb{C}$-action on $\Stab(\mathcal{D})$ commutes with the action of $\Aut(\mathcal{D})$.

\subsection{The cluster Poisson variety}
\label{sec:TheClusterPoissonVariety}

We now recall the definition of the cluster variety. By a \emph{seed}, we mean a tuple $\mathbf{i}=(\Gamma,\{e_i\}_{i\in I}, \langle -,-\rangle)$ consisting of a lattice $\Gamma$ of finite rank, a basis $\{e_i\}_{i\in I}$ for~$\Gamma$ indexed by a set~$I$, and an integer-valued skew-symmetric bilinear form $\langle -,-\rangle$ on~$\Gamma$. For example, let $\mathcal{A}\in\Tilt_\Delta(\mathcal{D})$ be a finite length heart. The Grothendieck group $K(\mathcal{D})\cong\mathbb{Z}^{\oplus n}$ is a lattice of rank~$n$ with a basis consisting of isomorphism classes of simple objects in~$\mathcal{A}$. It has a skew form $\langle -,-\rangle$ given by the Euler form. Thus we have a seed associated to~$\mathcal{A}$.

For any seed $\mathbf{i}=(\Gamma,\{e_i\}_{i\in I}, \langle -,-\rangle)$, we will consider the algebraic torus 
\[
\mathbb{T}_{\mathbf{i}}=\Hom_{\mathbb{Z}}(\Gamma,\mathbb{C}^*)\cong(\mathbb{C}^*)^n,
\]
and we will write $X_\gamma:\mathbb{T}_{\mathbf{i}}\rightarrow\mathbb{C}^*$ for the character corresponding to $\gamma\in\Gamma$. The bracket given on characters by 
\[
\{X_\alpha,X_\beta\}=\langle\alpha,\beta\rangle\cdot X_\alpha\cdot X_\beta
\]
defines a natural Poisson structure on $\mathbb{T}_{\mathbf{i}}$.

Suppose $\mathcal{A}$,~$\mathcal{A}'\in\Tilt_\Delta(\mathcal{D})$ are connected by an edge, and let $\mathbf{i}$, $\mathbf{i}'$ be the corresponding seeds. Then there is a bijection between classes of simple objects of~$\mathcal{A}$ and classes of simple objects of~$\mathcal{A}'$ since the associated Ext quivers are related by mutation. Let $S_i$~$(i\in I)$ be the simple objects in~$\mathcal{A}$ and $S_i'$~$(i\in I)$ the corresponding simple objects in~$\mathcal{A}'$ up to isomorphism. Denote by $\gamma_i=[S_i]$ and $\gamma_i'=[S_i']$ the isomorphism classes of these simple objects in $K(\mathcal{D})$. If $\mathcal{A}$ and~$\mathcal{A}'$ are related by a right tilt at the simple object~$S_k$, then from the proof of Theorem~7.1 in~\cite{BridgelandSmith}, one sees that the bases defined in this way are related by 
\[
\gamma_j'=
\begin{cases}
-\gamma_k & \text{if $j=k$} \\
\gamma_j+\max(\langle\gamma_k,\gamma_j\rangle,0)\cdot\gamma_k & \text{if $j\neq k$}.
\end{cases}
\]
There is a birational map $\mu_k:\mathbb{T}_{\mathbf{i}}\dashrightarrow\mathbb{T}_{\mathbf{i}'}$ given on functions by 
\[
\mu_k^*(X_\gamma)=X_\gamma\cdot(1+X_{\gamma_k})^{\langle\gamma,\gamma_k\rangle}.
\]
It preserves the Poisson structures of the two tori. If we write $X_j\coloneqq X_{\gamma_j}$, $X_j'\coloneqq X_{\gamma_j'}$, and $\varepsilon_{ij}=\langle\gamma_j,\gamma_i\rangle$, then it is well known that $\mu_k^*(X_j')$ is given by the formula in Proposition~\ref{prop:flipcoordinate} (see for example Lemma~2.11 in~\cite{FockGoncharov2}).

Let us denote by $\mathcal{T}_n$ the universal cover of $\Exch_\Delta(\mathcal{D})$, which is an $n$-regular tree. By what we have said, there is a well defined Poisson algebraic torus~$\mathbb{T}_t$ associated to each vertex $t\in\mathcal{T}_n$. If two vertices $t$ and $t'$ are connected by an edge in~$\mathcal{T}_n$, then there is an associated birational map $\mathbb{T}_t\dashrightarrow\mathbb{T}_{t'}$.

\begin{lemma}[\cite{GHK}, Proposition~2.4]
Let $\{Z_i\}$ be a collection of integral separated schemes of finite type over~$\mathbb{C}$ and suppose we have birational maps $f_{ij}:Z_i\dashrightarrow Z_j$ for all $i$, $j$ such that $f_{ii}$ is the identity and $f_{jk}\circ f_{ij}=f_{ik}$ as rational maps. Let $U_{ij}$ be the largest open subset of $Z_i$ such that $f_{ij}:U_{ij}\rightarrow f_{ij}(U_{ij})$ is an isomorphism. Then there is a scheme obtained by gluing the $Z_i$ along the open sets $U_{ij}$ using the maps $f_{ij}$.
\end{lemma}

To define the cluster variety, we apply this result to the tori $\mathbb{T}_t$ and the birational maps between them.

\begin{definition}
The \emph{cluster Poisson variety} $\mathscr{X}^{\mathrm{cl}}(Q)$ is the scheme obtained by gluing the tori $\mathbb{T}_t$ for all $t\in\mathcal{T}_{n}$ using the birational maps $\mu_k$ defined above.
\end{definition}

Our definition of the cluster modular group $\mathcal{G}_\Delta(\mathcal{D})$ is equivalent to its usual definition in cluster theory as a group generated by cluster transformations. It follows that there is a natural action of $\mathcal{G}_\Delta(\mathcal{D})$ on the cluster Poisson variety~$\mathscr{X}^{\mathrm{cl}}(Q)$. In this paper, we will employ a more concrete description of this action which is valid for cluster varieties arising from triangulated surfaces. For further discussion of this action from the categorical point of view, see Sections~5 and~7 of~\cite{Goncharov}.

\section{Categories from surfaces}

In this final section, we will interpret our earlier results in terms of stability conditions and the cluster variety, proving the main results from the introduction.

\subsection{Preliminaries}

Let us fix an amenable marked bordered surface $(\mathbb{S},\mathbb{M})$, and a signed triangulation $(T_0,\epsilon_0)$ of $(\mathbb{S},\mathbb{M})$. Then there is an associated quiver $Q=Q(T_0)$ with a canonical potential $W=W(T_0,\epsilon_0)$. For the remainder of this paper, we will write $\mathcal{D}=\mathcal{D}(Q,W)$ for the CY$_3$ triangulated category associated to this quiver with potential.

Note that there are two graphs associated to the data $(\mathbb{S},\mathbb{M})$: the graph $\Tri_{\bowtie}(\mathbb{S},\mathbb{M})$ of tagged triangulations and the heart exchange graph $\Exch_\Delta(\mathcal{D})$. There is an action of the signed mapping class group $\MCG^\pm(\mathbb{S},\mathbb{M})$ on $\Tri_{\bowtie}(\mathbb{S},\mathbb{M})$ and an action of the cluster modular group $\mathcal{G}_\Delta(\mathcal{D})$ on $\Exch_\Delta(\mathcal{D})$. The following result gives the basic link between these objects.

\begin{theorem}[\cite{BridgelandSmith}, Theorems~9.8 and~9.9]
\label{thm:graphsandgroups}
Take notation as in the previous two paragraphs. Then 
\begin{enumerate}
\item There is an isomorphism of graphs 
\[
\Tri_{\bowtie}(\mathbb{S},\mathbb{M})\cong\Exch_\Delta(\mathcal{D}).
\]
\item There is an isomorphism of groups 
\[
\MCG^\pm(\mathbb{S},\mathbb{M})\cong\mathcal{G}_\Delta(\mathcal{D}).
\]
\end{enumerate}
Moreover, under these isomorphisms, the action of the cluster modular group on the heart exchange graph coincides with the action of the signed mapping class group on the graph of tagged triangulations.
\end{theorem}

\subsection{Quadratic differentials and stability conditions}

In~\cite{BridgelandSmith}, Bridgeland and Smith considered a moduli space $\Quad(\mathbb{S},\mathbb{M})$ parametrizing equivalence classes of pairs $(S,\phi)$ where $S$ is a compact Riemann surface and $\phi$ is a GMN differential on~$S$ whose associated marked bordered surface is isomorphic to~$(\mathbb{S},\mathbb{M})$. Two pairs $(S_1,\phi_1)$ and $(S_1,\phi_2)$ are considered to be equivalent if there is an isomorphism $f:S_1\rightarrow S_2$ such that $f^*(\phi_2)=\phi_1$.

As explained in Section~6 of~\cite{BridgelandSmith}, the space $\Quad(\mathbb{S},\mathbb{M})$ is either empty or is a complex orbifold of dimension~$n$ given by~\eqref{eqn:dimension}. There is a dense open subset $\Quad(\mathbb{S},\mathbb{M})_0\subset\Quad(\mathbb{S},\mathbb{M})$ consisting of complete differentials, and the hat-homology groups $\widehat{H}(\phi)$ define a local system over this open set. A slightly subtle point in~\cite{BridgelandSmith} is that this local system has monodromy of order two around each component of the divisor parametrizing differentials with a simple pole, and hence it cannot be extended to a local system on the larger orbifold $\Quad(\mathbb{S},\mathbb{M})$. For this reason, we consider the $2^{|\mathbb{P}|}$-fold branched cover 
\[
\Quad^\pm(\mathbb{S},\mathbb{M})\rightarrow\Quad(\mathbb{S},\mathbb{M})
\]
obtained by choosing a sign for the residue $\Res_p(\phi)$ of a differential $\phi\in\Quad(\mathbb{S},\mathbb{M})$ at each pole~$p$ of order two. As shown in Lemma~6.2 of~\cite{BridgelandSmith}, the pullback of the hat-homology local system on $\Quad(\mathbb{S},\mathbb{M})_0$ extends to a local system on $\Quad^\pm(\mathbb{S},\mathbb{M})$. We consider the quotient orbifold 
\begin{equation}
\label{eqn:Quadheart}
\Quad_\heartsuit(\mathbb{S},\mathbb{M})=\Quad^\pm(\mathbb{S},\mathbb{M})/\mathbb{Z}_2^{\oplus\mathbb{P}}
\end{equation}
where $\mathbb{Z}_2^{\oplus\mathbb{P}}$ acts in the obvious way on the signs. Note that $\mathbb{P}$ forms a nontrivial local system of sets over the moduli space $\Quad(\mathbb{S},\mathbb{M})$, and hence this quotient must be understood in the category of spaces over $\Quad(\mathbb{S},\mathbb{M})$. More concretely, we can locally trivialize $\mathbb{P}$, and then \eqref{eqn:Quadheart} is defined by taking local quotients by the groups $\mathbb{Z}_2^{\oplus\mathbb{P}}$ over each trivializing local set and gluing these together to get the global quotient. Note that the space~\eqref{eqn:Quadheart} differs from $\Quad(\mathbb{S},\mathbb{M})$ only in its orbifold structure; it has a larger automorphism group along the divisor parametrizing differentials with simple zeros. By construction, the hat homology groups provide a local system on this space. There is a natural $\mathbb{C}$-action on $\Quad_\heartsuit(\mathbb{S},\mathbb{M})$ where a complex number $z\in\mathbb{C}$ takes a differential $\phi$ to the differential $e^{-2\pi iz}\cdot\phi$.

\begin{theorem}[\cite{BridgelandSmith}, Theorem~11.2]
\label{thm:BStheorem}
Let $(\mathbb{S},\mathbb{M})$ be an amenable marked bordered surface. Then there is a $\mathbb{C}$-equivariant isomorphism of complex orbifolds 
\begin{equation}
\label{eqn:orbifolds}
\Quad_\heartsuit(\mathbb{S},\mathbb{M})\cong\Stab_\Delta(\mathcal{D})/\cAut_\Delta(\mathcal{D}).
\end{equation}
\end{theorem}

We will be interested in a slight variant of this result involving the spaces $\mathscr{Q}^\pm(\mathbb{S},\mathbb{M})$ and $\Sigma(Q,W)$. To prove this modified result, note that $\Quad_\heartsuit(\mathbb{S},\mathbb{M})$ is the quotient of $\mathscr{Q}^\pm(\mathbb{S},\mathbb{M})$ by the action of the signed mapping class group $\MCG^\pm(\mathbb{S},\mathbb{M})$ and that the space $\mathscr{Q}^\pm(\mathbb{S},\mathbb{M})$ admits a $\mathbb{C}$-action where a complex number $z\in\mathbb{C}$ maps $\phi$ to $e^{-2\pi iz}\cdot\phi$.

\begin{lemma}
\label{lem:coverdifferentials}
The projection 
\[
q:\mathscr{Q}^\pm(\mathbb{S},\mathbb{M})\rightarrow\Quad_\heartsuit(\mathbb{S},\mathbb{M})
\]
is a $\mathbb{C}$-equivariant covering map.
\end{lemma}

\begin{proof}
The $\mathbb{C}$-equivariance is immediate from the definition of the $\mathbb{C}$-action on the two spaces. Note that $\Quad^\pm(\mathbb{S},\mathbb{M})$ is a cover of $\Quad_\heartsuit(\mathbb{S},\mathbb{M})$. It is the quotient of $\mathscr{Q}^\pm(\mathbb{S},\mathbb{M})$ by the action of the mapping class group $\MCG(\mathbb{S},\mathbb{M})$. It therefore suffices to show that the quotient map $\mathscr{Q}^\pm(\mathbb{S},\mathbb{M})\rightarrow\Quad^\pm(\mathbb{S},\mathbb{M})$ is a covering map, or equivalently that the action of the mapping class group is properly discontinuous.

Let $g=g(\mathbb{S})$ be the genus of~$\mathbb{S}$ and $d$ the number of boundary components of the associated surface $\mathbb{S}'$ defined in Section~\ref{sec:MarkedBorderedSurfaces}. Let $\mathscr{T}(g,d)$ be the Teichm\"uller space parametrizing Riemann surfaces of genus~$g$ with $d$ punctures, and let $\MCG(g,d)$ be the usual mapping class group acting on $\mathscr{T}(g,d)$. There is a natural group homomorphism $h:\MCG(\mathbb{S},\mathbb{M})\rightarrow\MCG(g,d)$ whose kernel is finite and consists of elements that change the markings and a map 
\[
\pi:\mathscr{Q}^\pm(\mathbb{S},\mathbb{M})\rightarrow\mathscr{T}(g,d)
\]
sending a signed, marked quadratic differential to its underlying marked Riemann surface. This map $\pi$ is equivariant with respect to the natural action of $\MCG(\mathbb{S},\mathbb{M})$ on $\Quad^\pm(\mathbb{S},\mathbb{M})$ and the action of $\MCG(\mathbb{S},\mathbb{M})$ on $\mathscr{T}(g,d)$ induced by~$h$.

Let $\alpha\in\mathscr{Q}^\pm(\mathbb{S},\mathbb{M})$ be a point, and write $S=\pi(\alpha)$. Then since $\MCG(g,d)$ acts properly discontinuously on $\mathscr{T}(g,d)$, we can find a neighborhood $S\in U\subset\mathscr{T}(g,d)$ such that for $\varphi\in\MCG(g,d)$, we have 
\[
\varphi(U)\cap U\neq\emptyset\implies \varphi=1.
\]
Let $V\subset\Quad^\pm(\mathbb{S},\mathbb{M})$ be the set of signed differentials defined on Riemann surfaces in~$U$. There is a local system of sets over~$V$ parametrizing choices of marking for the points of~$V$. Let $W\subset \mathscr{Q}^\pm(\mathbb{S},\mathbb{M})$ be a local flat section of this local system containing~$\alpha$. Since the local system has discrete fibers, $W$ is an open neighborhood of~$\alpha$. Now if $f\in\MCG(\mathbb{S},\mathbb{M})$ satisfies
\begin{equation}
\label{eqn:properlydiscontinuous}
f(W)\cap W\neq\emptyset,
\end{equation}
then $h(f)(U)\cap U\neq\emptyset$ so $h(f)=1$. Therefore $f$ is one of the finitely many elements in the kernel of~$h$. If $f\neq1$, then $f$ changes the marking and therefore cannot satisfy~\eqref{eqn:properlydiscontinuous}. It follows that $f=1$. Hence the action of $\MCG(\mathbb{S},\mathbb{M})$ is properly discontinuous.
\end{proof}

\begin{lemma}
\label{lem:coverstability}
The natural projection 
\[
p:\Sigma(Q,W)\rightarrow\Stab_\Delta(\mathcal{D})/\cAut_\Delta(\mathcal{D})
\]
is a $\mathbb{C}$-equivariant covering map.
\end{lemma}

\begin{proof}
The space on the right hand side is known to be an orbifold, and therefore the action of $\cAut_\Delta(\mathcal{D})$ on $\Stab_\Delta(\mathcal{D})$ is properly discontinuous. It follows that the action of $\cAut_\Delta(\mathcal{D})/\cSph_\Delta(\mathcal{D})$ on $\Sigma(Q,W)=\Stab_\Delta(\mathcal{D})/\cSph_\Delta(\mathcal{D})$ is properly discontinuous. This proves that $p$ is a covering map. It is $\mathbb{C}$-equivariant because the $\mathbb{C}$-action on $\Stab_\Delta(\mathcal{D})$ commutes with the action of $\cAut_\Delta(\mathcal{D})$.
\end{proof}

\begin{theorem}
\label{thm:identifymanifolds}
Let $(\mathbb{S},\mathbb{M})$ be an amenable marked bordered surface. Then there is an isomorphism of complex manifolds 
\[
\mathscr{Q}^\pm(\mathbb{S},\mathbb{M})\cong\Sigma(Q,W)
\]
which is equivariant with respect to the actions of $\MCG^\pm(\mathbb{S},\mathbb{M})\cong\mathcal{G}_\Delta(\mathcal{D})$ and~$\mathbb{C}$.
\end{theorem}

\begin{proof}
Let us denote by $B_0\subset\mathscr{Q}^\pm(\mathbb{S},\mathbb{M})$ the set consisting of all complete, saddle-free differentials. Choose a basepoint $\phi_0$ in this set. If $\gamma\in\pi_1(\mathscr{Q}^\pm(\mathbb{S},\mathbb{M}),\phi_0)$, then by Proposition~5.8 of~\cite{BridgelandSmith}, $\gamma$ can be represented by a loop which lies in~$B_0$ except at finitely many points, each of which corresponds to a differential with a unique horizontal saddle trajectory. Thus $\gamma$ determines a path on the graph $\Tri_{\bowtie}(\mathbb{S},\mathbb{M})$ of tagged triangulations. By Theorem~\ref{thm:graphsandgroups}, there is a corresponding path in $\Exch_\Delta(\mathcal{D})$. Since this graph is dual to the cell decomposition of $\Sigma(Q,W)$, we get an element $\eta\in\pi_1(\Sigma(Q,W),\sigma_0)$ where $\sigma_0\in\Sigma(Q,W)$ is a basepoint projecting to the same point of~\eqref{eqn:orbifolds} as~$\phi_0$. Let $K$ be the isomorphism of Theorem~\ref{thm:BStheorem}. By the construction of this isomorphism, we have 
\[
(K\circ q)_*(\gamma)=p_*(\eta)
\]
where $p$ and~$q$ are the covering maps of Lemmas~\ref{lem:coverstability} and~\ref{lem:coverdifferentials}, respectively. Thus 
\[
(K\circ q)_*\pi_1(\mathscr{Q}^\pm(\mathbb{S},\mathbb{M}))\subset p_*\pi_1(\Sigma(Q,W))
\]
and a standard lifting result from the theory of covering spaces implies that there exists a map $R$ making the following diagram commute:
\[
\xymatrix{
\mathscr{Q}^\pm(\mathbb{S},\mathbb{M}) \ar[d]_{q} \ar@{.>}[rr]^{R} & & \Sigma(Q,W) \ar[d]^{p} \\
\Quad_\heartsuit(\mathbb{S},\mathbb{M}) \ar[rr]_-{K} & & \Stab_\Delta(\mathcal{D})/\cAut_\Delta(\mathcal{D}).
}
\]
This map is holomorphic since $K$ is an isomorphism of orbifolds and $p$ and~$q$ are covering maps.

Similarly, if $\eta\in\pi_1(\Sigma(Q,W),\sigma_0)$ is any element, then $\eta$ can be represented by a loop that meets finitely many walls defined by a condition of the form $Z(S)=0$ for a unique simple object~$S$. Then, arguing as before, we can find an element $\gamma\in\pi_1(\mathscr{Q}^\pm(\mathbb{S},\mathbb{M}),\phi_0)$ such that $(K^{-1}\circ p)_*(\eta)=q_*(\gamma)$. It follows that there exists a map $L:\Sigma(Q,W)\rightarrow\mathscr{Q}^\pm(\mathbb{S},\mathbb{M})$ commuting with the covering maps. This map $L$ is holomorphic and inverse to~$R$, and hence $R$ is an isomorphism of complex manifolds.

This isomorphism is $\mathbb{C}$-equivariant because the maps $p$, $q$, and~$K$ are. If $\phi\in B_0$ then there is an associated tagged triangulation $\tau\in\Tri_{\bowtie}(\mathbb{S},\mathbb{M})$. The corresponding stability condition $R(\phi)$ has an associated heart, and this is exactly the vertex $\mathcal{A}\in\Exch_\Delta(\mathcal{D})$ corresponding to $\tau$ under the isomorphism of Theorem~\ref{thm:graphsandgroups}. If $g$ is an element of the group $\MCG^\pm(\mathbb{S},\mathbb{M})$, then $g\cdot\phi$ is an element of~$B_0$ whose associated WKB triangulation is $g\cdot\tau$. Since $\phi$ and $g\cdot\phi$ are mapped by~$q$ to the same point of $\Quad_\heartsuit(\mathbb{S},\mathbb{M})$, the images $R(\phi)$ and $R(g\cdot\phi)$ must map to the same point of $\Stab_\Delta(\mathcal{D})/\cAut_\Delta(\mathcal{D})$, and hence they differ by an element of $\mathcal{G}_\Delta(\mathcal{D})$. The heart of the stability condition $R(g\cdot\phi)$ is $g\cdot\mathcal{A}$, so this element must be~$g$. Hence the restriction $R|_{B_0}$ is equivariant with respect to the actions of $\MCG^\pm(\mathbb{S},\mathbb{M})\cong\mathcal{G}_\Delta(\mathcal{D})$. Since $B_0$ is dense in $\mathscr{Q}^\pm(\mathbb{S},\mathbb{M})$ and the group actions are continuous, it follows that $R$ has the required equivariance property.
\end{proof}

\subsection{Local systems and the cluster variety}

For any tagged triangulation $\tau\in\Tri_{\bowtie}(\mathbb{S},\mathbb{M})$, we have seen that the Fock-Goncharov coordinates provide a birational map $X_\tau:\mathscr{X}(\mathbb{S},\mathbb{M})\dashrightarrow(\mathbb{C}^*)^n$ from the space of framed local systems to an algebraic torus. By Lemma~9.10 of~\cite{BridgelandSmith}, the Grothendieck group $K(\mathcal{D})$ is naturally identified with the lattice spanned by the tagged arcs of~$\tau$. Therefore the torus $\Hom_\mathbb{Z}(K(\mathcal{D}),\mathbb{C}^*)\cong(\mathbb{C}^*)^n$ appearing in the definition of the cluster variety is identified with the torus parametrizing the Fock-Goncharov coordinates. Moreover, the isomorphism of lattices identifies the Euler form $\langle -,-\rangle$ with the skew form defined by the exchange matrix~$\varepsilon_{ij}$. Using these facts, one sees that the birational map defined by Proposition~\ref{prop:flipcoordinate} coincides with the map $\mu_k$ appearing in the definition of the cluster variety. Thus we see that the Fock-Goncharov coordinates provide a canonical birational map $\mathscr{X}(\mathbb{S,\mathbb{M}})\dashrightarrow\mathscr{X}^{\mathrm{cl}}(Q)$ which restricts to a regular embedding from the set $\mathscr{X}^*(\mathbb{S},\mathbb{M})$ of generic framed local systems into~$\mathscr{X}^{\mathrm{cl}}(Q)$. This map is equivariant with respect to the actions of $\MCG^\pm(\mathbb{S},\mathbb{M})\cong\mathcal{G}_\Delta(\mathcal{D})$ on these two spaces.

\subsection{A map from stability conditions to the cluster variety}

We can now prove the first main result of this paper.

\begin{theorem}
\label{thm:main1}
There is a dense open set $\Sigma^*(Q,W)\subset\Sigma(Q,W)$ and a $\mathcal{G}_\Delta(\mathcal{D})$-equivariant continuous map 
\[
\widehat{F}:\Sigma^*(Q,W)\rightarrow\mathscr{X}^{\mathrm{cl}}(Q)
\]
from this set to the cluster variety. If $\sigma\in\Sigma^*(Q,W)$ lies in the cell corresponding to some vertex $\mathcal{A}\in\Exch_\Delta(\mathcal{D})$, then for $z\in\mathbb{C}$ with $-\frac{1}{2}<\Re(z)<\frac{1}{2}$ and $\Im(z)\gg0$, the point $\widehat{F}(z\cdot\sigma)$ lies in the algebraic torus corresponding to~$\mathcal{A}$.
\end{theorem}

\begin{proof}
By Theorem~\ref{thm:identifymanifolds}, there is an isomorphism of manifolds $\Sigma(Q,W)\cong\mathscr{Q}^\pm(\mathbb{S},\mathbb{M})$. We define $\Sigma^*(Q,W)$ to be the set corresponding to $\mathscr{Q}^*(\mathbb{S},\mathbb{M})$ under this isomorphism. The set $\mathscr{X}^*(\mathbb{S},\mathbb{M})$ of generic framed local systems can be considered as a subset of the cluster variety, and the map $\widehat{F}$ is identified with the monodromy map of Proposition~\ref{prop:monodromydifferentials}. The $\mathcal{G}_\Delta(\mathcal{D})$-equivariance follows from the mapping class group equivariance of the monodromy map.

Note that if $\sigma$ lies in the cell corresponding to $\mathcal{A}\in\Exch_\Delta(\mathcal{D})$, then $\sigma$ corresponds, under the isomorphism of Theorem~\ref{thm:identifymanifolds}, to a complete saddle-free differential whose associated tagged triangulation $\tau$ is the vertex of $\Tri_{\bowtie}(\mathbb{S},\mathbb{M})$ corresponding to~$\mathcal{A}$. The second statement then says that the framed local system obtained by applying the map of Proposition~\ref{prop:monodromydifferentials} to $\hbar^{-2}\cdot\phi$ has well defined Fock-Goncharov coordinates with respect to~$\tau$ where we have written $\hbar=e^{i\pi z}$. This follows from Theorem~1.3(1) in~\cite{Allegretti19}.
\end{proof}

We can give a description of the subset $\Sigma^*(Q,W)\subset\Sigma(Q,W)$ in the language of stability conditions. Indeed, suppose $\sigma\in\Sigma(Q,W)$ is a stability condition corresponding to the quadratic differential $\phi\in\mathscr{Q}^\pm(\mathbb{S},\mathbb{M})$. By definition, we have $\sigma\in\Sigma^*(Q,W)$ if and only if $\phi\in\mathscr{Q}^*(\mathbb{S},\mathbb{M})$. If $\sigma\not\in\Sigma^*(Q,W)$ then $\widehat{F}(\phi)$ has trivial monodromy around some~$p\in\mathbb{P}$. In particular, this monodromy has eigenvalues~$\pm1$, and by Lemma~\ref{lem:residueeigenvalues}, the residue $\Res_p(\phi)$ is an integer multiple of~$2\pi i$. As shown in Section~2.4 of~\cite{BridgelandSmith}, one has for each $p\in\mathbb{P}$ a natural class $\beta_p\in K(\mathcal{D})$ in the kernel of the Euler form whose central charge satisfies $Z_\sigma(\beta_p)=\Res_p(\phi)$. In particular, if none of these central charges is an integer multiple of $2\pi i$, then we have $\sigma\in\Sigma^*(\mathbb{S},\mathbb{M})$.

\subsection{Riemann-Hilbert problems from Donaldson-Thomas theory}

A choice of stability condition $\sigma$ on the CY$_3$ triangulated category $\mathcal{D}$ naturally determines a BPS structure $(\Gamma_\sigma,Z_\sigma,\Omega_\sigma)$ given by the following data:
\begin{enumerate}[label=(\alph*)]
\item $\Gamma_\sigma=K(\mathcal{D})$ is the Grothendieck group of~$\mathcal{D}$ equipped with the Euler form $\langle -,-\rangle$.
\item $Z_\sigma$ is the central charge of the stability condition.
\item $\Omega_\sigma(\gamma)$ is the BPS invariant as defined in Donaldson-Thomas theory.
\end{enumerate}

By the results of~\cite{BridgelandSmith}, a generic GMN differential $\phi\in\mathscr{Q}^\pm(\mathbb{S},\mathbb{M})$ corresponds under the isomorphism of Theorem~\ref{thm:identifymanifolds} to a stability condition $\sigma\in\Sigma(Q,W)$, and the BPS~structure $(\Gamma_\sigma,Z_\sigma,\Omega_\sigma)$ equals the BPS~structure $(\Gamma_\phi,Z_\phi,\Omega_\phi)$ defined in Section~\ref{sec:TheRiemannHilbertProblem}. In particular, it is convergent and defines a Riemann-Hilbert problem. The results of Section~\ref{sec:SolvingTheRiemannHilbertProblem} then imply the following, which is the second main result of this paper.

\begin{theorem}
If $\sigma\in\Sigma(Q,W)$ is a generic stability condition, then the meromorphic functions~$\mathcal{X}_r$ constructed in Section~\ref{sec:SolvingTheRiemannHilbertProblem} provide a solution of the weak Riemann-Hilbert problem associated to the BPS structure $(\Gamma_\sigma,Z_\sigma,\Omega_\sigma)$. If we assume moreover that the surface $\mathbb{S}$ is closed, then these functions provide a solution of the full Riemann-Hilbert problem.
\end{theorem}

It is interesting to ask whether our solution of the Riemann-Hilbert problem has any uniqueness properties. As mentioned in Section~\ref{sec:ConstructingTheSolution}, our solution depends on a choice of meromorphic projective structure which we take to be the one constructed by uniformization in Section~\ref{sec:AnEmbeddingOfModuliSpaces}. It is possible to replace this meromorphic projective structure by a different one, but the choice of meromorphic projective structure should satisfy the property in Lemma~\ref{lem:Q2} so that we can apply the results of~\cite{Allegretti19}. The uniformizing projective structure seems to be the most canonical choice satisfying this property.

\bibliographystyle{amsplain}

\end{document}